\documentclass[dvips,preprint]{imsart}
\RequirePackage[OT1]{fontenc}
\RequirePackage{amsthm, amsmath, amsfonts, amssymb}
\RequirePackage[colorlinks,citecolor=blue,urlcolor=blue]{hyperref}

\usepackage{amsmath,amssymb,amsthm,bm,graphicx,enumitem}
\usepackage{hyperref}

\RequirePackage{xr}

\startlocaldefs
  \newtheorem{thm}{Theorem}[section]
   \newtheorem{lemma}{Lemma}[section]
   \newtheorem{prop}{Proposition}[section]
   \newtheorem{cor}{Corollary}[section]
   
   \newtheorem{remark}{Remark}[section]

   \newtheorem{cond}{Condition}[section]

\allowdisplaybreaks[1]

\endlocaldefs

\begin{document}

\begin{frontmatter}
\title{
Bootstrapping Two-phase Sampling}
\runtitle{Bootstrapping Two-phase Sampling}

\author{\fnms{Takumi } \snm{Saegusa}\ead[label=e1]{tsaegusa@uw.edu}}
\address{Department of Biostatistics\\ University of Washington\\ Seattle, WA  98195-7232 \\ \printead{e1}\\}
\affiliation{University of Washington}

\runauthor{Saegusa}

\begin{abstract}
We propose a nonparametric bootstrap procedure for two-phase stratified sampling without replacement.
In this design, a weighted likelihood estimator is known to have smaller asymptotic variance than under the convenient assumption of independence often made in practice.
Variance estimation, however, has not been well studied for semiparametric models where variance may not have a closed form.
Motivated by semiparametric inference, we establish conditional weak convergence of bootstrap inverse probability weighted empirical processes with several variants of calibration.
Two main obstacles to applying existing bootstrap empirical process theory are the dependent and biased sample due to sampling design, and the complex limiting processes of the linear combinations of Brownian bridge processes.
To address these issues, the proposed bootstrap weights take the form of the product of two weights corresponding to randomness from each phase and stratum.
We apply our bootstrap to weighted likelihood estimation and establish two $Z$-theorems for a general semiparametric model where a nuisance parameter can be estimated either at a regular or a non-regular rate.
We show different bootstrap calibration methods proposed in the survey sampling literature yield different bootstrap asymptotic distributions.
\end{abstract}



\end{frontmatter}

\section{Introduction}
\label{sec:1}
Two-phase sampling was originally proposed in \cite{Neyman:38} for estimation of finite population parameters in sample surveys.
The design has become extensively used in epidemiological studies where a parameter of interest is defined via a probability distribution for an infinite population.
At the first phase, a large sample is obtained from a population.
Some variables whose information is easier to collect are measured for stratification.
At the second phase, a subsample is drawn without replacement from each stratum to obtain other variables that are costly or difficult to measure.
Careful choice of stratifying variables improves the efficiency of the design by collecting important variables with less cost.
Examples of this design include the exposure stratified case control study \cite{White:1982} and the stratified case cohort study \cite{MR1767493}.

Various estimation procedures have been proposed
(see \cite{MR1450191,MR1680310,MR1965682,MR2125853,MR2066254,MR2090916,Breslow:2009a}
among many), but until recently dependence due to sampling without replacement has been largely ignored in the biostatistical literature for mathematical convenience.
This dependence of observations is a critical factor to differentiate our statistical problem from simply treating two-phase sampling as a special case of missing data problems by assuming Bernoulli sampling.
The question of dependence was solved when \cite{MR2325244} successfully established weak convergence of the Inverse Probability Weighted (IPW) empirical process based on the exchangeably weighted bootstrap empirical process theory \cite{MR1245301}.
The limiting process is a linear combination of independent Brownian bridge processes reflecting randomness from the sampling from a population at the first phase and stratified sampling without replacement at the second phase.
With further developments of empirical process tools \cite{MR3059418},
\cite{MR2325244,Breslow:2009a, Breslow:2009b,MR3059418} studied asymptotic properties of the Weighted Likelihood Estimator (WLE) and its improvement in efficiency by estimated weights \cite{MR1294730} and various calibrations \cite{MR1173804,MR3059418} in a general semiparametric model.
These results found that the asymptotic variances of the WLEs are generally smaller than under Bernoulli sampling which assumes independence.

In this paper, we propose and study a bootstrap procedure for two-phase stratified sampling without replacement.
As in Efron's original bootstrap paper \cite{MR515681}, our primary motivation is the variance estimation problem, raised in \cite{MR3059418} with an emphasis on a general semiparametric model.
A difficulty in this model is that asymptotic variances of the WLEs may contain unknown functions or may not have a closed form.
A similar problem for the MLE with complete data was treated in \cite{MR1693616} using a numerical derivative of the log likelihood.
However, its extension to our problem only estimates part of variance \cite{Saegusa,Saegusa-Var:2014}, and its application is limited to WLEs.
An alternative approach is nonparametric bootstrap inference \cite{MR2860827}, but variance is overestimated if sampling is without replacement instead of Bernoulli sampling.  
Our proposed bootstrap procedure overcomes these difficulties, and  yields the correct variance of  more general estimators such as IPW $M$- and $Z$-estimators.

The main goal of this paper is to establish conditional weak convergence of our bootstrap IPW empirical processes indexed by a class of functions.
This generality beyond bootstrapping random variables is required for bootstrap inference for a general semiparametric model \cite{WellnerZhan:96,MR2722459}.
For complete data, \cite{MR1245301} established weak convergence of exchangeably weighted bootstrap empirical processes including Efron's bootstrap.
Important differences from our case are (1) the limiting process is  a single Brownian bridge process, and (2) data are required to be i.i.d. from a population.
For dependent data due to sampling, various bootstrap procedures have been proposed for complex survey designs (see
\cite{gross,MR740906,MR941020,MR1183077,MR1185197,MR1310222,MR1351010}
to name a few).
Because these methods primarily concern variables in a finite population and randomness only from sampling design, asymptotic theory is formulated differently (see e.g. \cite{MR648029}), and hence extensions to weak convergence is usually not straightforward.
Our aim is to provide theory and tools for extending the bootstrap empirical process theory to a dependent and biased sample from the two-phase sampling design.

The main contributions of our paper are three-fold.
First, we propose a novel bootstrap procedure for two-phase sampling, and adopt a bottom-up approach to proving conditional weak convergence of the bootstrap IPW empirical processes.
To obtain the complex limiting processes, our proposed bootstrap weights take the form of the product of the i.i.d. weights and the weights proposed in \cite{MR740906,gross} for stratified sampling in a finite population.
These weights yield randomness from different phases of sampling and different strata.
The main theoretical difficulty is non-i.i.d. observations and non-exchangeable bootstrap weights (see Remark 3.1 below for details), which violate assumptions in the bootstrap empirical process theory \cite{MR1245301,MR1385671}.
To address these issues, our proof takes three steps: (1) decompose the bootstrap IPW empirical process into the phase I and II bootstrap IPW empirical processes, (2)
establish weak convergence of the phase II bootstrap IPW empirical process conditional on the phase I bootstrap weights, (3) establish weak convergence of the phase I bootstrap IPW empirical process, and (4) compute the entire covariance functions of these two processes.
This method of proof allows for separate analysis of the different phases and hence would have applications to other complex designs.

The second contribution is application to IPW $Z$-estimation in a general semiparametric model where an infinite-dimensional nuisance parameter can be estimated either at a regular or a non-regular rate.
With complete data, $Z$-estimation was studied for the former case by \cite{STAN:STAN9,MR1385671} and for the latter case by \cite{MR1394975}.
Bootstrap $Z$-estimation was treated for the former and  the latter cases by \cite{WellnerZhan:96} and \cite{MR2202406,MR2722459} respectively.
With two-phase sampling data, IPW $Z$-estimation was studied with emphasis on WLEs by \cite{MR2325244,MR2391566,MR3059418} for both cases.
Our results here cover bootstrap $Z$-estimation for both regular and non-regular cases. 
Conditions in our theorems are almost identical to those for the MLE with complete data  \cite{MR1385671,MR1394975}, and are exactly the same as those in the non-bootstrap case of \cite{MR3059418} under two-phase sampling.
Because these conditions are formulated in terms of complete data, some of them have been already established for a specific complete data model or empirical process theory helps to verify them in a straightforward way.
We also prove a general theorem for the rate of convergence of the bootstrap IPW $M$-estimators of a nuisance parameter under weak conditions.

The third contribution is the comparison of various calibration methods under bootstrap.
Calibration \cite{MR1173804} and its variants \cite{MR2988409,MR3059418} are statistical techniques that aim at improving the efficiency of IPW $Z$-estimators.
We study two bootstrap calibration methods proposed in survey sampling (see e.g. \cite{MR2990349,Preston:2009}) and their extension to centered calibration \cite{MR3059418}.
Our results show different bootstrap calibration techniques lead to different bootstrap asymptotic distributions of WLEs in  a general semiparametric model.
This difference plays an important role in bias correction illustrated in our numerical study and data analysis.

The rest of the paper is organized as follows.
In Section \ref{sec:2}, we introduce basic notations and review previous results of two-phase sampling.
We describe our bootstrap procedures with several variants of calibration in Section \ref{sec:3}.
Conditional weak convergence and other asymptotic results are presented in Section \ref{sec:4}.
Section \ref{sec:5} concerns application to weighted likelihood estimation in a general semiparametric model.
Performance of our method is illustrated in simulation and a real data example in Section \ref{sec:6}.
All proofs are collected in the Appendix.

\section{Sampling, Calibrations, and IPW Empirical Processes}
\label{sec:2}
We introduce basic notations and review previous results in \cite{MR2325244,MR3059418}.
\subsection{Sampling}
Let $W=(X,U)\in\mathcal{W}=\mathcal{X}\times\mathcal{U}$ with distribution $\tilde{P}_0$.
Here $X$ is a vector of variables of interest with distribution $P_0$ and $U$ is a vector of auxiliary variables.
At the first phase under two-phase sampling,  we only observe a coarsening $\tilde{X}=\tilde{X}(X)$ (e.g. screening test) of $X$ (e.g. gold standard) in addition to the auxiliary variables $U$ (e.g. mailing address).
Auxiliary variables $U$ are useful for stratification and improving efficiency of estimation involving $X$.
We call $V=(\tilde{X},U)\in \mathcal{V}=\tilde{\mathcal{X}}\times \mathcal{U}$ the phase I variables.
The phase I sample space $\mathcal{V}$ for $V$ is partitioned into the $J$ sampling strata $\mathcal{V}_j$ with $\sum_{j=1}^J \mathcal{V}_j = \mathcal{V}$ for stratified sampling at the second phase.
We denote the stratum probability for the $j$th stratum by $\nu_j\equiv \tilde{P}_0(V\in\mathcal{V}_j)$,
and the conditional expectation given membership in the $j$th stratum by $P_{0|j}(\cdot) \equiv \tilde{P}_0(\cdot |V\in\mathcal{V}_j)$.

With complete data, we would observe $W_1,\ldots,W_N$ i.i.d. as $W$.
Under two-phase sampling, the observed data at the first phase is $V_1,\ldots,V_N$ i.i.d. as $V$.
At the second phase, a subsample is drawn without replacement from each stratum by which $n_j$ items out of $N_j = \{i\leq N:V_i\in\mathcal{V}_j\}$ are selected in the $j$th stratum.
We observe $X_i$ for the sampled item at the second phase.
We denote the sampling indicator by $\xi_i\in\{0,1\}$ with $1$ if sampled at the second phase and $0$ otherwise.
The sampling probability is $\tilde{P}_0(\xi_i=1|V_i\in\mathcal{V}_j) = n_j/N_j \equiv \pi_0(V_i)$.
The observed data in the entire process are $(V_i,X_i\xi_i,\xi_i),i=1,\ldots,N$.

Throughout, we use a doubly subscripted notation:  for example, $V_{j,i}$ denotes $V$ for the $i$th subject in the $j$th stratum.
We assume that there is a constant $\sigma>0$ such that $0<\sigma\leq \pi_0(v)\leq 1$ for every $v\in \mathcal{V}$ and that $n_j/N_j\rightarrow p_j>0$ for $j=1,\ldots, J$ as $N\rightarrow \infty$.
Let $\pi_{0,\infty}(V_i) = \sum_{i=1}^J p_j 1_{\mathcal{V}_j}(V_i)$ be the limiting sampling probability.
Note that phase II sample sizes $n_j$ are at the disposal of a designer of the two-phase study.
We denote $ a \lesssim b$ to mean $a \le K b$ for some constant $K\in (0,\infty)$.
We write $|\cdot|$ for a Euclidean distance.
For a function $z:T\mapsto \mathbb{R}$, we write $\lVert z\rVert_T \equiv \sup_{t\in T}|z(t)|$, and $\ell^\infty(\mathcal{F})$ for the set of bounded functionals on $\mathcal{F}$.

\subsection{IPW empirical process}
Our main result concerns weak convergence of the bootstrapped IPW empirical process.
Define the IPW empirical measure by
\begin{equation*}
\mathbb{P}_N^\pi
=\frac{1}{N}\sum_{i=1}^N\frac{\xi_i}{\pi_0(V_i)}\delta_{X_i}
=\frac{1}{N}\sum_{j=1}^J\sum_{i=1}^{N_j}\frac{\xi_{j,i}}{n_j/N_j}\delta_{X_{j,i}},
\end{equation*}
and let $\mathbb{G}_N^\pi=\sqrt{N}(\mathbb{P}_{N}^{\pi}-P_0)$ be the IPW empirical process where $\delta_{X_i}$ denotes a Dirac measure placing a unit mass on $X_i$.
Compare these with the empirical measure and process with complete data given by $\mathbb{P}_N = N^{-1}\sum_{i=1}^N \delta_{X_i}$ and $\mathbb{G}_N = \sqrt{N}(\mathbb{P}_N - P_0)$.
When an index function maps from $\mathcal{W}$ instead of $\mathcal{X}$, we understand $\delta_{X_i}$ and $P_0$ in the definitions above as $\delta_{W_i}$ and $\tilde{P}_0$.
We use this abuse of notation for other (bootstrap) IPW empirical processes.

\subsection{Calibrations}
The WLE only uses observations sampled at the second phase and is generally inefficient.
Two basic methods for improving efficiency are estimated weights \cite{MR1294730} and calibration \cite{MR1173804}.
These methods adjust weights in the WLE by utilizing information in $V$ available for all observations.
Here we treat only calibration and centered calibration \cite{MR3059418} because calibration and estimated weights are equivalent under some transformation, and because (within-stratum) centered calibration have guaranteed efficiency gains unlike the other methods treated in \cite{MR3059418} (see \cite{MR3059418} for systematic comparison of these methods in different designs).
These methods can make use of part of $V$ or its transformation but we only consider $V$ in calibrations for a notational simplicity.
Let $\tilde{V}_i \equiv V_i -P_0V$.

The calibration method ``models'' the sampling probability by $\pi_\alpha(V)\equiv\pi_0(V)/G_{c}(V;\alpha)$ where $G_{c}(V;\alpha)\equiv G(V^T\alpha)$ for a known differentiable function $G$ with $G(0)=1$ and $(d/dt)G(t)|_{t=0}=\dot{G}(0)>0$.
The estimator  $\hat{\alpha}_N^c$ of $\alpha$ is solution to the calibration equation
\begin{equation}
\label{eqn:caleqn2}
\frac{1}{N}\sum_{i=1}^N\frac{\xi_iG_{c}(V_i;\alpha)}{\pi_0(V_i)}V_i
=\frac{1}{N}\sum_{i=1}^NV_i.
\end{equation}
This equation equates the calibrated IPW average of $V$ from the phase II sample with its phase I average whereby the phase II sample becomes more representative of the phase I sample.
The centered calibration method ``models'' the sampling probability by $\pi_\alpha(V)\equiv\pi_0(V)/G_{cc}(V;\alpha)$ where
$G_{cc}(V;\alpha)\equiv G((\pi_0(V)^{-1}-1)(V-\mathbb{P}_NV)^T\alpha)$.
The estimator  $\hat{\alpha}_N^{cc}$ of $\alpha$ is solution to the calibration equation
\begin{equation}
\label{eqn:ccaleqn}
\frac{1}{N}\sum_{i=1}^N\frac{\xi_iG_{cc}(V_i;\alpha)}{\pi_0(V_i)}(V_i-\mathbb{P}_NV)
=0, \quad \alpha\in\mathcal{A}\subset \mathbb{R}^k.
\end{equation}
The calibrated IPW empirical measure $\mathbb{P}_N^{\pi,c}$ is defined by replacing the sampling probability $\pi_0(v)$ by the calibrated sampling probability $\pi_{\hat{\alpha}_N^c}(v)$ in $\mathbb{P}_N^\pi$.
The IPW empirical measure with centered calibration $\mathbb{P}_N^{\pi,cc}$ is defined similarly.
The IPW empirical processes with calibration and centered calibration are $\mathbb{G}_{N}^{\pi,c}=\sqrt{N}(\mathbb{P}_{N}^{\pi,c}-P_0)$ and $\mathbb{G}_{N}^{\pi,cc}=\sqrt{N}(\mathbb{P}_{N}^{\pi,cc}-P_0)$ respectively.

The following theorem \cite{MR2325244,MR3059418} concerns weak convergence of the IPW empirical processes with calibrations.
The goal of this paper is to establish the corresponding result for our bootstrap IPW empirical processes.

\begin{cond}[Calibrations]
\label{cond:cal}
\noindent
(1) Estimators $\hat{\alpha}_N^c$ and $\hat{\alpha}_N^{cc}$ are solutions to the calibration equations (\ref{eqn:caleqn2}) and  (\ref{eqn:ccaleqn}), respectively.

\noindent (2)  $V\in\mathbb{R}^k$ is not concentrated at $0$ and has bounded support.

\noindent(3)  $G$ is a strictly increasing, bounded, and continuously differentiable function on $\mathbb{R}$ satisfying $G(0)=0$ with bounded derivative $\dot{G}$.

\noindent(4) (i) $P_0V^{\otimes 2}$ is finite and positive definite.
(ii) $P_0\{(\pi_{0,\infty}(V)^{-1} -1)\tilde{V}^{\otimes 2}\}$
is finite and positive definite.

\noindent (5) The ``true'' parameter $\alpha_0=0$.
\end{cond}

\begin{thm}[\cite{MR2325244,MR3059418}]
\label{thm:fpsD}
Let $\mathcal{F}$ be a $P_0$-Donsker class with $\| P_0 \|_{\cal F}<\infty$.
Suppose that Condition \ref{cond:cal} holds.
Then,
\begin{eqnarray}
\label{eqn:fpsD}
&&\mathbb{G}_N^\pi
\rightsquigarrow
\mathbb{G}^{\pi}
\equiv\mathbb{G}+\sum_{j=1}^J\sqrt{\nu_j}\sqrt{\frac{1-p_j}{p_j}}\mathbb{G}_j,\\
&&\mathbb{G}_N^{\pi,\#}
\rightsquigarrow
\mathbb{G}^{\pi,\#}
\equiv \mathbb{G}+\sum_{j=1}^J\sqrt{\nu_j}\sqrt{\frac{1-p_j}{p_j}}\mathbb{G}_j(\cdot-Q_{\#}\cdot),
\end{eqnarray}
in $\ell^\infty(\mathcal{F})$ where $\#\in\{c,cc\}$, the
$P_0$-Brownian bridge process $\mathbb{G}$ and the $P_{0|j}$-Brownian bridge processes $\mathbb{G}_j$ are all independent, and
the maps $Q_c$ and $Q_{cc}$ from $\mathcal{F}$ to $\mathbb{R}$ are given by
\begin{eqnarray*}
Q_{c}f&\equiv&P_0(f V^T)\{P_0V^{\otimes 2}\}^{-1}V, \\
Q_{cc}f&\equiv&P_0\{(\pi_{0,\infty}^{-1}(V)-1)f \tilde{V}^T\}[P_0\{(\pi_{0,\infty}^{-1}(V)-1)\tilde{V}^{\otimes 2}\}]^{-1}\tilde{V}.
\end{eqnarray*}
\end{thm}
Theorem \ref{thm:fpsD} shows that the limiting processes are the linear combinations of the Brownian bridge processes.
The process $\mathbb{G}$ corresponds to sampling at the first phase, because this is also the limiting process when complete data are available (to see this, set $p_j=1,j=1,\ldots,J$ to sample all observations).
The process $\mathbb{G}_j$ corresponds to sampling at the second phase for the $j$th stratum.
One can see from the form of the weight $\{\nu_j(1-p_j)/p_j\}^{1/2}$ for $\mathbb{G}_j$ that a smaller stratum (i.e., small $\nu_j$) or more observations at the second phase (i.e., large $p_j$) reduce variance due to sampling from the $j$th stratum.

\section{Our bootstrap}
\label{sec:3}
We describe our bootstrap procedure.
\subsection{Bootstrap weights}
A bootstrap procedure assigns a bootstrap weight of a random variable to each observation.
For example, a bootstrap weight of Efron's bootstrap is a count of how many times an observation is sampled with replacement in a bootstrap sample.
These weights as a whole follow the multinomial distribution with parameters $n$ and $(1/n,\ldots,1/n)$ where $n$ is a sample size.
Various exchangeable bootstrap weights with different distributions yield weak convergence of the bootstrap empirical process to a
Brownian bridge process $\mathbb{G}$ (see Theorem \ref{thm:fpsD}) up to constant \cite{MR1245301,MR1385671}.
However, the limiting process of our problem consists of multiple stratum-specific processes $\mathbb{G}_{j}$ as well as $\mathbb{G}$.
Thus we first construct our bootstrap weights for each phase and each stratum separately, and then define an ultimate bootstrap weight as the product of the phase I and II bootstrap weights.
We describe this construction below.

Our phase I bootstrap weights are i.i.d. within a stratum.
This reflects randomness due to sampling from a population at the first phase (see Appendix of \cite{MR2325244} for the equivalence of sampling from a population
and stratified sampling after sampling from a multinomial distribution of stratum membership).
Specifically, let the phase I bootstrap weights $W_{N_j,j,i}^{(1)}$,
$i=1,\ldots,N_j$, for the $j$th stratum with $j=1,\ldots,J$, be i.i.d. $W_j^{(1)}\sim P_{W_j^{(1)}}$, satisfying
\begin{eqnarray}
\label{eqn:ph1bootwt}
&&P(W_{j}^{(1)} >   0)=1, \
EW_{j}^{(1)}=1,\ \mathrm{Var}(W_{j}^{(1)}) = p_j/(2-p_j) \equiv c_{j}^2,\\
&& \lVert W_{j}^{(1)}\rVert_{2,1} = \int_0^\infty \{P(W_{j}^{(1)}>x)\}^{1/2}dx<\infty.\nonumber
\end{eqnarray}

Our phase II bootstrap weights is based on the following bootstrap procedure developed for a stratified sample in a finite population
\cite{gross,MR740906}.
Suppose for simplicity that $N_j = 3n_j$ and that we are interested in the $j$th stratum. We first create an ``artificial population'' of sample size $N_j$ by copying $n_j$ observations sampled at the second phase three times.
We then obtain a bootstrap sample of size $n_j$ from this artificial population by sampling without replacement.
If $N_j$ is not divisible by $n_j$, we create two artificial population of different sizes.
This bootstrap procedure corresponds to the (mixture of) multivariate hypergeometric distribution(s).

For a formal definition, let $MH_d(N,n,(m_1,\ldots,m_d))$ denote the multivariate hypergeometric distribution (see \cite{MR1429617} for details) where $n$ balls are sampled without replacement from the population consisting of the disjoint subgroups of size $m_i,i=1,\ldots,d$, with $\sum_{i=1}^dm_i=N$.
For the $j$th stratum with $N_j = n_jk_j+r_j,k_j,r_j\in\mathbb{N},0\leq r_j<n_j$, let $(\tilde{W}_{j,1}^{(2)},\ldots,\tilde{W}_{j,n_j}^{(2)})\in\mathbb{R}^{n_j}$ be a vector of exchangeable weights that follow the mixture of the multivariate hypergeometric distribution
\begin{eqnarray}
\label{eqn:ph2bootwt}
&&MH_{n_j}(n_jk_j,n_j,(k_j,\ldots,k_j)),\quad \qquad \qquad \qquad\mbox{with probability } s_j,\\
&&MH_{n_j}(n_j(k_j+1),n_j,(k_j+1,\ldots,k_j+1)), \quad \mbox{with probability }1-s_j,\nonumber
\end{eqnarray}
where
\begin{equation*}
s_j\equiv\left(1-\frac{r_j}{n_j}\right)\left(1-\frac{r_j}{N_j-1}\right)
\end{equation*}
and $1-s_j$ are mixing probabilities.
We define $W_{n_j,j,i}^{(2)}$ by $W_{n_j,j,i}^{(2)}=0$ if $\xi_{j,i}=0$ and $W_{n_j,j,i}^{(2)}=\tilde{W}_{j,k}^{(2)}$ where the observation $(V_{j,i},\xi_{j,i}X_{j,i},\xi_{j,i})$ has the $k$th smallest index $i$ among the observations with $\xi_{j,i}=1$ in the $j$th stratum.
Note that phase I and II bootstrap weights are independent.

Now we define the two-phase bootstrap weights for the $i$th
observation in the $j$th stratum by
\begin{equation}
\label{eqn:bootwt}
W_{N_j,n_j,j,i}\equiv W_{N_j,j,i}^{(1)}W_{n_j,j,i}^{(2)}, \quad j=1,\ldots,J, i=1,\ldots,N_j.
\end{equation}
We write $W_{Ni}$, $W_{Ni}^{(1)}$, and $W_{Ni}^{(2)}$ for $W_{N_j,n_j,j,i}$, $W_{N_j,j,i}^{(1)}$, and $W_{n_j,j,i}^{(2)}$ respectively when we do not specify the stratum where the observation belongs.
Define the bootstrap IPW empirical measure
\begin{equation}
\label{eqn:bootEmpMeas}
\hat{\mathbb{P}}_N^\pi \equiv \frac{1}{N}\sum_{i=1}^NW_{Ni}\frac{\xi_i}{\pi_0(V_i)}\delta_{X_i}
\end{equation}
and we let $\tilde{\mathbb{G}}_N^\pi \equiv \sqrt{N}(\hat{\mathbb{P}}_N^\pi-\mathbb{P}_N^\pi)$ be the ``uncentered''
bootstrap IPW empirical process (further centering will be introduced in 3.3 below).
We also define the phase I and phase II bootstrap empirical measures $\hat{\mathbb{P}}_N^{\pi,(1)}$ and $\hat{\mathbb{P}}_N^{\pi,(2)}$ by replacing $W_{Ni}$ in (\ref{eqn:bootEmpMeas}) by $W_{Ni}^{(1)}$ and $W_{Ni}^{(2)}$, respectively.

\begin{remark}[Dependence of observatons and non-exchangeability of bootstrap weights]
The observed data  $(V_i,X_i\xi_i,\xi_i),i=1,\ldots,N$, are dependent through the sampling indicator $\xi$.
Also, our bootstrap weights are not exchangeable because marginal distributions of our bootstrap weights differ depending on stratum membership and the sampling indicators.
These are major obstacles to applying the exchangeably weighted bootstrap empirical process theory \cite{MR1245301}.
However, observations sampled at the second phase in the same stratum are i.i.d. $P_{0|j}$ conditional on $\xi =1$ and $V\in\mathcal{V}_j$ (see Remark 4.3 of \cite{MR2253102}).
Moreover, the phase I bootstrap weights ($W_{N_j,j,i}^{(1)},i=1,\ldots,N_j$) and the phase bootstrap weights among observations sampled at the second phase ($\tilde{W}_{j,i}^{(2)},i=1,\ldots,,n_j$) are exchangeable in the same stratum.
These key observations lead to decomposition of the bootstrap IPW empirical processes and applying the exchangeably weighted bootstrap empirical process theory to stratum-specific bootstrap IPW empirical processes.
\end{remark}

\subsection{Calibrations}
We introduce bootstrap calibrations.
There are several possibilities to carry out calibrations under bootstrap.
Recall that calibration aims at equating the IPW average with the phase I average.
The IPW average in a bootstrap sample can be equated with the phase I average, the IPW average, the calibrated IPW average, or the average in an artificial population.
In this paper we only discuss the first two cases proposed in survey sampling.
\cite{Saegusa} includes discussion of the third case.
As shown below, different choice yields different asymptotic distributions.

A standard method for calibration in a bootstrap sample is calibration to the phase I average (see e.g. \cite{MR2990349,Preston:2009}).
We call this method bootstrap calibration.
Let $\hat{\hat{\alpha}}_N^{bc}$ be the solution to the bootstrap calibration equation
\begin{equation}
\label{eqn:survbootcaleqn}
\mathbb{P}_N^\pi W_N^{(2)}G_{c}(V;\alpha)V = \mathbb{P}_NV,\quad \alpha\in\mathcal{A}\subset \mathbb{R}^k.
\end{equation}
Define the bootstrap IPW empirical measure with bootstrap calibration by
\begin{equation*}
\hat{\mathbb{P}}_{N}^{\pi,bc} =  \frac{1}{N}\sum_{i=1}^NW_{Ni} \frac{\xi_iG_{c}(V_i;\hat{\hat{\alpha}}_N^{bc})}{\pi_0(V_i)}\delta_{X_i}
\end{equation*}
and the bootstrap IPW empirical process with bootstrap calibration by $\tilde{\mathbb{G}}_{N}^{\pi,bc}=\sqrt{N}(\hat{\mathbb{P}}_{N}^{\pi,bc}-\mathbb{P}_N^{\pi,c})$.
For centered calibration, we define the bootstrap IPW empirical measure with bootstrap centered calibration by
\begin{equation*}
\hat{\mathbb{P}}_{N}^{\pi,bcc} =  \frac{1}{N}\sum_{i=1}^NW_{Ni} \frac{\xi_iG_{cc}(V_i;\hat{\hat{\alpha}}_N^{bcc})}{\pi_0(V_i)}\delta_{X_i}
\end{equation*}
and the bootstrap IPW empirical process with bootstrap centered calibration by
$\tilde{\mathbb{G}}_{N}^{\pi,bcc}=\sqrt{N}(\hat{\mathbb{P}}_{N}^{\pi,bcc}-\mathbb{P}_N^{\pi,cc})$ where $\hat{\hat{\alpha}}_N^{bcc}$ is a solution to the bootstrap calibration equation
\begin{equation}
\label{eqn:survbootccaleqn}
\mathbb{P}_N^\pi W_N^{(2)}G_{cc}(V;\alpha)(V - \mathbb{P}_NV) = 0,\quad \alpha\in\mathcal{A}\subset \mathbb{R}^k.
\end{equation}

Another calibration method proposed in \cite{MR2990349} is calibration to the IPW average replacing $\mathbb{P}_NV$ by $\mathbb{P}_N^\pi V$ in the bootstrap calibration equations above.
As in \cite{Saegusa}, we call this method bootstrap single calibration (see also ``double calibration'' in \cite{Saegusa}).
Define the bootstrap IPW empirical measure with bootstrap single calibration by
\begin{equation*}
\hat{\mathbb{P}}_{N}^{\pi,bsc} =  \frac{1}{N}\sum_{i=1}^NW_{Ni} \frac{\xi_iG_{c}(V_i;\hat{\hat{\alpha}}_N^{bsc})}{\pi_0(V_i)}\delta_{X_i}
\end{equation*}
and the bootstrap IPW empirical process with bootstrap single calibration by $\tilde{\mathbb{G}}_{N}^{\pi,bsc}=\sqrt{N}(\hat{\mathbb{P}}_{N}^{\pi,bsc}-\mathbb{P}_N^{\pi})$ where $\hat{\hat{\alpha}}_N^{bsc}$ is the solution to the bootstrap single calibration equation
\begin{equation}
\label{eqn:singlebootcaleqn}
\mathbb{P}_N^\pi W_N^{(2)}G_{c}(V;\alpha)V = \mathbb{P}^\pi_NV,\quad \alpha\in\mathcal{A}\subset \mathbb{R}^k.
\end{equation}
For centered calibration, we define the bootstrap IPW empirical measure with bootstrap single centered calibration by
\begin{equation*}
\hat{\mathbb{P}}_{N}^{\pi,bscc} =  \frac{1}{N}\sum_{i=1}^NW_{Ni} \frac{\xi_iG_{cc}(V_i;\hat{\hat{\alpha}}_N^{bscc})}{\pi_0(V_i)}\delta_{X_i}
\end{equation*}
and the bootstrap  IPW empirical process with bootstrap  single centered calibration by $\tilde{\mathbb{G}}_{N}^{\pi,bscc}=\sqrt{N}(\hat{\mathbb{P}}_{N}^{\pi,bscc}-\mathbb{P}_N^{\pi})$ where $\hat{\hat{\alpha}}_N^{bscc}$ is a solution to the bootstrap calibration equation
\begin{equation}
\label{eqn:singlebootccaleqn}
\mathbb{P}_N^\pi W_N^{(2)}G_{cc}(V;\alpha)(V - \mathbb{P}_N^\pi V) = 0,\quad \alpha\in\mathcal{A}\subset \mathbb{R}^k.
\end{equation}

\begin{remark}[Phase II bootstrap weights for bootstrap calibration]
In contrast to the bootstrap IPW empirical measures and processes, all bootstrap calibration equations above only involve the phase II bootstrap weights.
The reason is that calibration methods only affect the phase II variance (see e.g. \cite{MR3059418}).
This formulation of bootstrap calibration allows for applications to a bootstrap procedure only involving the randomness at the second phase (e.g. \cite{Saegusa-Var:2014}).
\end{remark}
\begin{remark}[Centering of bootstrap IPW empirical measures]
Both $\tilde{\mathbb{G}}_N^{\pi,bsc}$ and $\tilde{\mathbb{G}}_N^{\pi,bscc}$ have centering by $\mathbb{P}_N^\pi$ while $\tilde{\mathbb{G}}_N^{\pi,bsc}$ and $\tilde{\mathbb{G}}_N^{\pi,bscc}$ has centering by the corresponding calibrated IPW empirical measures $\mathbb{P}_N^{\pi,c}$ and $\mathbb{P}_N^{\pi,cc}$ respectively.
This difference yields different bootstrap asymptotic distributions for inference in a general semiparametric model in Section \ref{sec:5}.
\end{remark}

The following condition is a bootstrap alternative of Condition \ref{cond:cal}.
\begin{cond}
\label{cond:bootcal}
Conditions \ref{cond:cal}(b)-(e) hold and estimators
$\hat{\hat{\alpha}}_N^{bc}$, $\hat{\hat{\alpha}}_N^{bcc}$,
$\hat{\hat{\alpha}}_N^{bsc}$, and $\hat{\hat{\alpha}}_N^{bscc}$ are solutions to the calibration equations
(\ref{eqn:survbootcaleqn}), (\ref{eqn:survbootccaleqn}), (\ref{eqn:singlebootcaleqn}), and (\ref{eqn:singlebootccaleqn}), respectively.
\end{cond}

\subsection{Bootstrap IPW empirical processes}
We require further centering of the bootstrap IPW empirical processes.
Define the centered bootstrap IPW empirical processes by
\begin{eqnarray*}
\hat{\mathbb{G}}_N^\pi\cdot \equiv \tilde{\mathbb{G}}_N^\pi(\cdot -\mathbb{P}_N^\pi \cdot),\quad
\hat{\mathbb{G}}_N^{\pi,*\#}\cdot \equiv\tilde{\mathbb{G}}_N^{\pi,*\#} (\cdot -\mathbb{P}_N^{\pi,\#} \cdot)
\end{eqnarray*}
with $*\in \{b,bs\}$ and $\# \in \{c,cc\}$.
This further centering yields a subtle but important difference in limiting processes of bootstrap IPW empirical processes (compare Theorems \ref{thm:fpsDboot} and \ref{thm:fpsDboot2} below).
\begin{remark}
While centering by constant does not change the (IPW) empirical processes (e.g. $\mathbb{G}_N(\cdot -\mathbb{P}_N\cdot) = \mathbb{G}_N$)
we have $\hat{\mathbb{G}}_N^\pi \neq \tilde{\mathbb{G}}_N^\pi$ and $\hat{\mathbb{G}}_N^{\pi,*\#} \neq \tilde{\mathbb{G}}_N^{\pi,*\#}$ in general because of the property of the phase I bootstrap weights that $N^{-1}\sum_{i=1}^NW_{Ni}^{(1)} \neq 1$ in general.
\end{remark}

\section{Main Results}
\label{sec:4}
We establish the Glivenko-Cantelli and Donsker theorems for bootstrap IPW empirical processes under two-phase sampling.
\cite{MR3059418} showed that the Glivenko-Cantelli and Donsker properties for i.i.d. data are inherited to data from two-phase sampling.
Here we show that these properties continue to hold for bootstrap IPW empirical processes.

\subsection{Probability Spaces}
We define the probability space for $(V,\xi X,\xi)$ and the bootstrap weight $W$.
Let $\underline{W}_j^{(1)}=\{W_{N_j,j,i}^{(1)}:i=1,\ldots,N_j, N_j=1,2,.\ldots\}$ be a triangular array defined on the probability space  $(\mathcal{Z}_j^{(1)},\mathcal{E}_j^{(1)},P_{W_j^{(1)}})$ for $j=1,\ldots,J$.
Let also $\underline{W}_j^{(2)}=\{W_{n_j,j,i}^{(2)}:i=1,\ldots,n_j, n_j=1,2,.\ldots\}$ be a triangular array defined on the probability space  $(\mathcal{Z}_j^{(2)},\mathcal{E}_j^{(2)},P_{W_j^{(2)}})$ for $j=1,\ldots,J$.
$P_{W_j^{(1)}}$ and $P_{W_j^{(2)}}$ are the conditional probability measures given the phase I sample size $\sum_{i=1}^N I_{\mathcal{V}_j}(V_i)$ and the phase II sample size $\sum_{i=1}^N \xi_{i}I_{\mathcal{V}_j}(V_i)$ in the $j$th stratum, respectively.
Recall that the phase I and phase II bootstrap weights are independent.
Define the probability spaces $(\mathcal{Z}^{(1)},\mathcal{E}^{(1)},P_W^{(1)})=\prod_{j=1}^J(\mathcal{Z}_j^{(1)},\mathcal{E}_j^{(1)},P_{W_j}^{(1)})$ and $(\mathcal{Z}^{(2)},\mathcal{E}^{(2)},P_W^{(2)})=\prod_{j=1}^J(\mathcal{Z}_j^{(2)},\mathcal{E}_j^{(2)},P_{W_j}^{(2)})$ for phase I and phase II bootstrap weights.
Let the probability space $(\mathcal{Z},\mathcal{E},P_W)=(\mathcal{Z}^{(1)},\mathcal{E}^{(1)},P_{W}^{(1)}) \times (\mathcal{Z}^{(2)},\mathcal{E}^{(2)},P_{W}^{(2)})$ for the whole bootstrap weights.
We denote the probability space for $(V_i,\xi_iX_i,\xi_i),i=1,2,\ldots,$ as $(\mathcal{X}^\infty,\mathcal{B}^\infty,P^\infty)$ (with abuse of notations), and denote
\begin{equation*}
(\mathcal{X}^\infty\times \mathcal{Z},\mathcal{B}^\infty\times \mathcal{E},Pr)=
(\mathcal{X}^\infty,\mathcal{B}^\infty,P^\infty)
\times
(\mathcal{Z},\mathcal{E},P_W),
\end{equation*}
where $Pr \equiv P^{\infty} \times P_W$.
We let $P^*$ and $P_*$ denote the outer and the inner probability, respectively, corresponding to $P^\infty$ (see Section 1.2 of \cite{MR1385671} for details).

\subsection{Order Notations for Bootstrap}
To study the conditional asymptotic behavior of bootstrap given data, we define the order notations in probability for bootstrap.
For a real function $\Delta_N$ defined on the joint probability space $(\mathcal{X}^\infty\times \mathcal{Z},\mathcal{B}^\infty\times \mathcal{E},Pr)$, we say that $\Delta_N$ is of an order $o_{P_W^*}(1)$ in $P^*$-probability if for any $\epsilon>0$ and $\eta>0$,
\begin{equation*}
P^*\left\{P_W^*\left(|\Delta_N|>\epsilon\right)>\eta\right\}\rightarrow 0, \quad \mbox{as }N\rightarrow \infty.
\end{equation*}
This definition was introduced and studied in \cite{WellnerZhan:96}.
We say that $\Delta_N$ is of an order $O_{P_W^*}(1)$ in $P^*$-probability if for any $\eta>0$ and for every $M_N \rightarrow \infty$,
\begin{equation*}
P^*\left\{P_W^*\left(|\Delta_N|>M_N\right)>\eta\right\}\rightarrow 0,\quad \mbox{as }N\rightarrow \infty.
\end{equation*}
Note that this definition is slightly different from one introduced in \cite{MR2722459}.
The definitions of $o_{P_{W^{(k)}}^*}(1)$ and $O_{P_{W^{(k)}}^*}(1)$ in $P^*$-probability with $k=1,2$ are defined analogously.
We also define the bootstrap version of the almost sure convergence.
For a real function $\Delta$ defined on $(\mathcal{X}^\infty\times \mathcal{Z},\mathcal{B}^\infty\times \mathcal{E},Pr)$, we say $\Delta_N\rightarrow \Delta$ in outer $P_{W|\infty}$-almost surely if for every fixed $x\notin N_X$ with $N_X\in\mathcal{B}^\infty$ and $P(N_X)=0$, $|\Delta_N-\Delta|^*\rightarrow 0$ $P_W$-almost surely with an associated null set set $N_{W|x}\in\mathcal{E}$ depending on $x$ where measurability here is with respect to $P_W$.
We summarize several results in the following lemma.

\begin{lemma}
\label{lemma:bootorder}
Let $\Delta_N$ and $\Gamma_N$ be real functions defined on the joint probability space $(\mathcal{X}^\infty\times \mathcal{Z},\mathcal{B}^\infty\times \mathcal{E},Pr)$.\\
(1) If $\Delta_N=o_{Pr^*}(1)$ (resp. $O_{Pr^*}(1)$) , then
$\Delta_N=o_{P_W^*}(1)$ (resp. $O_{P_W^*}(1)$) in
$P^*$-probability. The converse is true if $\Delta_N$ is measurable.\\
(2) If $\Delta_N$ is a real function on
$(\mathcal{X}^\infty,\mathcal{B}^\infty,P^\infty)$ and $\Delta_N=o_{P^*}(1)$ (resp. $O_{P^*}(1)$) , then
$\Delta_N=o_{Pr*}(1)$ (resp. $O_{Pr^*}(1)$) in
$P^*$-probability.\\
(3) If $\Delta_N=o_{P_W^*}(1)$ and $\Gamma_N=O_{P_W^*}(1)$ in
$P^*$-probability, then $\Delta_N\Gamma_N = o_{P_W^*}(1)$ in $P^*$-probability.\\
(4) If $\Delta_N=O_{P_W^*}(1)$ and $\Gamma_N=O_{P_W^*}(1)$ in
$P^*$-probability, then $\Delta_N\Gamma_N = O_{P_W^*}(1)$ in $P^*$-probability.\\
(5) $\Delta_N=o_{P_W^*}(1)$ in $P^*$-probability if and only if every subsequence $\Delta_{N'}$ has a further subsequence $\Delta_{N^{''}}$ such that $\Delta_{N^{''}}^*\rightarrow 0$ in  $P_{W|\infty}$-almost surely.\\
The above results (1)-(5) hold if we replace $P_W^*$ by $P_{W^{(k)}}^*$ with $k=1,2$.
\end{lemma}

With these order notations, we say that an estimator $\theta_n$ is consistent for $\theta_0$ in $P^*$-probability if $d(\theta_n,\theta_0) = o_{P^*_W}(1)$ in $P^*$-probability where $\theta_n$, and $\theta$ are elements of a parameter space $\Theta$ equipped with semimetric $d$.

\subsection{Bootstrap Glivenko-Cantelli theorem}
The following is the Glivenko-Cantelli theorem for our bootstrap IPW empirical measures.
Let $\hat{\mathbb{P}}_N^{\pi,(1)}$ and $\hat{\mathbb{P}}_N^{\pi,(2)}$  be the phase I and II bootstrap IPW empirical measures obtained by replacing the bootstrap weights $W_{Ni}$ by the phase I bootstrap weights $W_{Ni}^{(1)}$ and phase II bootstrap weights $W_{Ni}^{(2)}$ in $\hat{\mathbb{P}}_N^{\pi}$, respectively.

\begin{thm}
\label{thm:fpsPiGCboot}
Let $\mathcal{F}$ be a $P_0$-Glivenko-Cantelli class with $\lVert P_0\rVert_{\mathcal{F}} <\infty$.
Then
\begin{equation}
\label{eqn:fpsPiGCTPboot}
\lVert \hat{\mathbb{P}}_N^{\pi}-P_0\rVert_{\mathcal{F}}\rightarrow_{P_W^*} 0, \ \mbox{in $P^*$-probability}.
\end{equation}
This also holds if we replace $P_0$ by $\mathbb{P}_N^\pi$, $\mathbb{P}_N^{\pi,(1)}$, $\mathbb{P}_N^{\pi,(2)}$ or $\mathbb{P}_N^{\pi,*}$ (assuming Condition \ref{cond:cal}), or if we replace $P_0$ by $\hat{\mathbb{P}}_N^{\pi}$ or $\hat{\mathbb{P}}_N^{\pi,*\#}$ (assuming Condition \ref{cond:bootcal}) with $*\in \{b,bs\}$ and $\#\in\{c,cc\}$.
\end{thm}

\subsection{Bootstrap Donsker theorem}
We present two bootstrap Donsker theorems for our bootstrap IPW empirical processes.
The first theorem concerns the uncentered bootstrap IPW empirical processes.
\begin{thm}
\label{thm:fpsDboot}
Let $\mathcal{F}$ be a $P_0$-Donsker class with $\lVert P_0\rVert_{\mathcal{F}}<\infty$.
Suppose that Conditions \ref{cond:cal} and \ref{cond:bootcal} hold.
Then
\begin{eqnarray*}
&&\tilde{\mathbb{G}}_N^\pi\rightsquigarrow
\tilde{\mathbb{G}}^{\pi}\equiv
\tilde{\mathbb{G}}+\sum_{j=1}^J\sqrt{\nu_j}\sqrt{\frac{1-p_j}{p_j}}\mathbb{G}_j, \\
&&\tilde{\mathbb{G}}_N^{\pi,*\#}\rightsquigarrow
\tilde{\mathbb{G}}^{\pi,\#}\equiv
\tilde{\mathbb{G}}+\sum_{j=1}^J\sqrt{\nu_j}\sqrt{\frac{1-p_j}{p_j}}\mathbb{G}_j(\cdot-Q_{\#}\cdot),
\end{eqnarray*}
in $\ell^\infty(\mathcal{F})$ in $P^*$-probability where $*\in \{b,bs\}$ and $\# \in \{c,cc\}$, $P_0$- Brownian motion process $\tilde{\mathbb{G}}$ and $P_{0|j}$-Brownian bridge processes $\mathbb{G}_j$ are all independent and $Q_{\#}$ are defined in Theorem \ref{thm:fpsD}.
\end{thm}

The second theorem concerns the centered bootstrap IPW empirical processes.
\begin{thm}
\label{thm:fpsDboot2}
Let $\mathcal{F}$ be a $P_0$-Donsker class with $\lVert P_0\rVert_{\mathcal{F}}<\infty$.
Suppose that Conditions \ref{cond:cal} and \ref{cond:bootcal} hold.
Then
\begin{eqnarray*}
&&\hat{\mathbb{G}}_N^\pi\rightsquigarrow
\mathbb{G}^{\pi}=
\mathbb{G}+\sum_{j=1}^J\sqrt{\nu_j}\sqrt{\frac{1-p_j}{p_j}}\mathbb{G}_j, \\
&&\hat{\mathbb{G}}_N^{\pi,*\#} \rightsquigarrow
\mathbb{G}^{\pi,\#}=
\mathbb{G}+\sum_{j=1}^J\sqrt{\nu_j}\sqrt{\frac{1-p_j}{p_j}}\mathbb{G}_j(\cdot-Q_{\#}\cdot),
\end{eqnarray*}
in $\ell^\infty(\mathcal{F})$ in $P^*$-probability where $*\in \{b,bs\}$ and $\# \in \{c,cc\}$, and  $\mathbb{G}$, $\mathbb{G}_j$ and $Q_{\#}$ are defined in Theorem \ref{thm:fpsD}.
\end{thm}

\begin{remark}
The limiting processes in Theorem \ref{thm:fpsD} and Theorem \ref{thm:fpsDboot} are the same only when $P_0f = 0$ for every $f\in \mathcal{F}$ because Theorem \ref{thm:fpsDboot} involves the Brownian motion process $\tilde{\mathbb{G}}$, not the Brownian bridge process $\mathbb{G}$.
However, uncentered bootstrap IPW empirical processes lead to simple bootstrap inference for the IPW $M$- and $Z$-estimators in a general semiparametric model discussed in the next section.
\end{remark}

The proof of Theorem \ref{thm:fpsDboot} goes as follows.
Let $\hat{\mathbb{G}}_N^{\pi,(1)}$ and $\hat{\mathbb{G}}_N^{\pi,(2)}$ be the phase I and II bootstrap IPW empirical processes obtained by replacing $W_{Ni}$ by $W_{Ni}^{(1)}$ and $W_{Ni}^{(2)}$ in $\hat{\mathbb{G}}_N^{\pi}$, respectively.
Note that  $\tilde{\mathbb{G}}_N^{\pi,(1)}=\sqrt{N}\mathbb{P}_N^\pi (W_N^{(1)}-1)\cdot$ and $\tilde{\mathbb{G}}_N^{\pi,(2)}=\sqrt{N}\mathbb{P}_N^\pi (W_N^{(2)}-1)\cdot$.
We decompose the bootstrap IPW empirical measure to obtain
\begin{eqnarray*}
\tilde{\mathbb{G}}_N^\pi
= \sqrt{N}(\hat{\mathbb{P}}_N^{\pi,(1)} - \mathbb{P}_N^\pi) + \sqrt{N}(\hat{\mathbb{P}}_N^\pi - \hat{\mathbb{P}}_N^{\pi,(1)})
\equiv \tilde{\mathbb{G}}_N^{\pi,(1)} + \tilde{\mathbb{G}}_N^{\pi,(2)}W_N^{(1)}\cdot,
\end{eqnarray*}
Conditionally on data, the first term 
$\tilde{\mathbb{G}}_N^{\pi,(1)}$ involves randomness due to the phase I bootstrap weights.
The second term $\tilde{\mathbb{G}}_N^{\pi,(2)}W_N^{(1)}\cdot =\sqrt{N}\mathbb{P}_N^\pi W_{N}^{(1)}(W_N^{(2)}-1)\cdot$ involves randomness due to the phase II bootstrap weights given the phase I bootstrap weights as well as data.
Thus, we first establish weak convergence of $\tilde{\mathbb{G}}_N^{\pi,(2)}W_N^{(1)}\cdot$ conditionally on the phase I bootstrap weights and then obtain weak convergence of $\tilde{\mathbb{G}}_N^{\pi,(1)}$.
Combining limiting processes and calculating covariance functions complete the proof.

Establishing weak convergence  of $\tilde{\mathbb{G}}_N^{\pi,(1)}$ and  $\tilde{\mathbb{G}}_N^{\pi,(2)}W_N^{(1)}\cdot$ involves the extension of existing results.
For the phase I bootstrap IPW empirical processes, we prove the uncentered conditional multiplier central limit theorem.
Previous results only cover the centered conditional multiplier central limit theorem and the uncentered unconditional multiplier central limit theorem
(see Theorem 2.9.6 and Corollary 2.9.4 of \cite{MR1385671}, respectively).
Our result provides a rigorous justification of the weighted bootstrap of \cite{MR2202406} which counted on the unconditional result.

\begin{lemma}
\label{lemma:ph1D}
Let $X_1,\ldots,X_n$ be i.i.d. $P_0$, and $w_1,\ldots,w_n$ be i.i.d. random variables with $Ew_1=0$, $\mbox{Var}(w_1)=c^2$ and $\lVert w_1\rVert_{2,1}<\infty$, independent of $X_1,\ldots,X_n$.
Define $\tilde{\mathbb{G}}_n \equiv n^{-1/2}\sum_{i=1}^n w_i\delta_{X_i}$ and $\tilde{\mathbb{G}}\equiv \mathbb{G}+Z_0P_0$ where $\mathbb{G}$ is a $P_0$-Brownian bridge process independent of the standard normal random variable $Z_0$.
Let $BL_1$ be the set of functions $h:\ell^{\infty}(\mathcal{F})\mapsto [0,1]$ such that $|h(z_1)-h(z_2)|\leq \lVert z_1-z_2\rVert_{\mathcal{F}}$ for every $z_1,z_2.$
The expectation with respect to $w_1,w_2,\ldots$ is denoted by $E_w$.
For a $P_0$-Donsker class $\mathcal{F}$ with $\lVert P_0\rVert_{\mathcal{F}}<\infty$ the following hold.\\
\noindent (1) The sequence $\tilde{\mathbb{G}}_n$ is asymptotically measurable and $\sup_{h\in BL_1}|E_w h(\tilde{\mathbb{G}}_n)-h(c\tilde{\mathbb{G}})|\rightarrow 0$ in outer probability.

\noindent (2) If $P_0\lVert f-P_0f\rVert_{\mathcal{F}}^2<\infty$, then $\sup_{h\in BL_1}|E_w h(\tilde{\mathbb{G}}_n)-h(c\tilde{\mathbb{G}})|\rightarrow 0$ outer almost surely, and the sequence $|E_w h(\tilde{\mathbb{G}}_n)^*-h(c\tilde{\mathbb{G}})_*|\rightarrow 0$ almost surely for every $h\in BL_1$
where $h(\tilde{\mathbb{G}}_n)^*$ and $h(\tilde{\mathbb{G}}_n)_*$ denote measurable majorants and minorants with respect to $(w_1,\ldots,w_n,X_1,\ldots,X_n)$ jointly.
\end{lemma}
For the phase II bootstrap IPW empirical process, \cite{Saegusa-Var:2014} generalized the bootstrap CLT of \cite{MR740906} to the bootstrap Donsker theorem under weaker conditions.
We further extend this result to the bootstrap Donsker theorem with calibrations.
Define $\tilde{\mathbb{G}}_N^{\pi,(2),*\#}$ with $*\in \{b,bs\}$ and $\# \in \{c,cc\}$ by replacing $W_N$ by $W_N^{(2)}$ in $\tilde{\mathbb{G}}_N^{\pi,*\#}$.
\begin{lemma}
\label{lemma:ph2D}
Let $\mathcal{F}$ be a $P_0$-Donsker class.
Then
\begin{eqnarray*}
&&\tilde{\mathbb{G}}_N^{\pi,(2)}\rightsquigarrow \mathbb{G}^{\pi,(2)}\equiv
\sum_{j=1}^J\sqrt{\nu_j}\sqrt{\frac{1-p_j}{p_j}}\mathbb{G}_j, \ \mbox{in $\ell^\infty(\mathcal{F})$}
\end{eqnarray*}
in $P^*$-probability.
Suppose moreover that Conditions \ref{cond:cal} and \ref{cond:bootcal} hold with $\lVert P_0\rVert_{\mathcal{F}}<\infty$.
Then for $*\in \{b,bs\}$ and $\# \in \{c,cc\}$
\begin{eqnarray*}
&&\tilde{\mathbb{G}}_N^{\pi,(2),*\#}\rightsquigarrow
\sum_{j=1}^J\sqrt{\nu_j}\sqrt{\frac{1-p_j}{p_j}}\mathbb{G}_j(\cdot-Q_{\#}\cdot), \ \mbox{in $\ell^\infty(\mathcal{F})$}
\end{eqnarray*}
in $P^*$-probability.
Here $\mathbb{G}_j$ and $Q_{\#}$ are defined in Theorem \ref{thm:fpsD}.
\end{lemma}

\section{Applications to a General Semiparametric Model}
\label{sec:5}
We consider bootstrapping WLEs in a general semiparametric model under two-phase sampling (See also Theorem \ref{thm:survbootzthm1} in Section \ref{subsec:semipara} 
for applications to a general IPW $Z$-estimators in a general statistical model).
Our theorems cover two cases where an infinite-dimensional  parameter can be estimated either at a regular or a non-regular rate.
For concrete examples, see \cite{MR1385671,MR2325244,MR3059418} for the former case and \cite{MR1394975,MR3059418} for the latter case.

Let $\mathcal{P}=\{P_{\theta,\eta}:\theta\in\Theta,\eta\in H\}$ be a semiparametric model parametrized by a finite dimensional parameter $\theta \in \Theta\subset \mathbb{R}^p$ and the infinite-dimensional nuisance parameter $\eta\in H$ where the nuisance parameter space $H$ is a subset of some Banach space $(\mathcal{B},\lVert\cdot \rVert)$. 
Let $P_0=P_{\theta_0,\eta_0}$ denote the true distribution.
The WLE
$(\hat{\theta}_N,\hat{\eta}_N)$ is a solution to the following weighted likelihood equations
\begin{eqnarray}
\label{eqn:wlik}
&&\Psi_{N,1}^\pi(\theta,\eta)=\mathbb{P}_N^\pi\dot{\ell}_{\theta,\eta}=o_{P^*}(N^{-1/2}),
\nonumber \\
&&\lVert \Psi_{N,2}^\pi(\theta,\eta)h\rVert_{\mathcal{H}}
=\lVert \mathbb{P}_N^\pi (B_{\theta,\eta}h-P_{\theta,\eta}B_{\theta,\eta}h)\rVert_{\mathcal{H}}
=o_{P^*}(N^{-1/2}),
\end{eqnarray}
where $\dot{\ell}_{\theta,\eta}\in \mathcal{L}_2^0(P_{\theta,\eta})^p$ is the score function for $\theta$, and the score operator $B_{\theta,\eta}:\mathcal{H}\mapsto \mathcal{L}_2^0(P_{\theta,\eta})$ is the bounded linear operator mapping a direction $h$ in some Hilbert space $\mathcal{H}$ of one-dimensional submodels for $\eta$ along which $\eta\rightarrow \eta_0$.
Note that $\eta -\eta_0 \in \mathcal{H}$ because for $\eta(t)\equiv (1-t)\eta_0 + t\eta$, $\eta(0)=\eta_0$ and $(d/dt)\eta(t)|_{t=0}=\eta-\eta_0$.

\subsection{Regular Rate for a Nuisance Parameter}
\label{subsec:5-1}
We consider the case where the infinite-dimensional parameter can be estimated at a regular rate (i.e., $\sqrt{N}\lVert \hat{\eta}-\eta_0\rVert = O_{P^*}(1)$).
We assume the following condition for the WLEs.
\begin{cond}[Consistency]
\label{cond:wlereg1}
The estimator $(\hat{\theta}_N,\hat{\eta}_N)$ is consistent for $(\theta_0,\eta_0)$ and solves the weighted likelihood equations (\ref{eqn:wlik}), where $\mathbb{P}_N^\pi$
may be replaced by $\mathbb{P}_N^{\pi,\#}$ with the corresponding estimators $(\hat{\theta}_{N,\#},\hat{\eta}_{N,\#}),\#\in\{c,cc\}$.
\end{cond}
The corresponding bootstrap WLE $(\hat{\hat{\theta}}_N,\hat{\hat{\eta}}_N)$ is a solution to the bootstrap weighted likelihood equations
\begin{eqnarray}
\label{eqn:wlikbootsurv}
&&\hat{\Psi}_{N,1}^{\pi}(\theta,\eta)=\hat{\mathbb{P}}_N^{\pi}\dot{\ell}_{\theta,\eta}=o_{P^*_W}\left(N^{-1/2}\right),
\nonumber \\
&&\left\lVert \hat{\Psi}_{N,2}^{\pi}(\theta,\eta)h\right\rVert_{\mathcal{H}}
=\left\lVert \hat{\mathbb{P}}_N^{\pi} (B_{\theta,\eta}h-P_{\theta,\eta}B_{\theta,\eta}h)\right\rVert_{\mathcal{H}}
=o_{P^*_W}\left(N^{-1/2}\right),
\end{eqnarray}
in $P^*$-probability.
We replace Condition \ref{eqn:wlik} by the following condition for our bootstrap WLEs.
\begin{cond}[Consistency]
\label{cond:bootwlereg1}
The bootstrap estimator $(\hat{\hat{\theta}}_N,\hat{\hat{\eta}}_N)$ is consistent for $(\theta_0,\eta_0)$ in $P^*$-probability and
solves the bootstrap weighted likelihood equations (\ref{eqn:wlikbootsurv}) in $P^*$-probability where $\hat{\mathbb{P}}_N^\pi$
may be replaced by $\hat{\mathbb{P}}_N^{\pi,*\#}$ with the corresponding estimators $(\hat{\hat{\theta}}_{N,*\#},\hat{\hat{\eta}}_{N,*\#})$ and corresponding maps $\hat{\Psi}_{N,k}^{\pi,*\#},k=1,2,$ where $*\in\{b,bs\}$ and $\#\in\{c,cc\}$.
\end{cond}

The rest of the conditions are shared by both WLEs and bootstrap WLEs.
Note that these conditions are formulated in terms of complete data.
\begin{cond}[Asymptotic equicontinuity]
\label{cond:wlereg2}
Let $\mathcal{F}_1(\delta)=\{\dot{\ell}_{\theta,\eta}:|\theta-\theta_0|+\lVert \eta-\eta_0\rVert<\delta\}$
and $\mathcal{F}_2(\delta)=\{B_{\theta,\eta}h-P_{\theta,\eta}B_{\theta,\eta}h:h\in\mathcal{H},|\theta-\theta_0|
        +\lVert \eta-\eta_0\rVert<\delta\}$.
There exists a $\delta_0>0$ such that
(1) $\mathcal{F}_k(\delta_0),k=1,2,$ are $P_0$-Donsker and
$\sup_{h\in\mathcal{H}}P_0|f_j-f_{0,j}|^2\rightarrow 0$, as
$|\theta-\theta_0|+\lVert \eta-\eta_0\rVert \rightarrow 0$,
for every $f_j\in \mathcal{F}_j(\delta_0),j=1,2$, where
$f_{0,1}=\dot{\ell}_{\theta_0,\eta_0}$ and $f_{0,2}=B_0h-P_0B_0h$, and
(2) $\mathcal{F}_k(\delta_0),k=1,2$, have integrable envelopes.
\end{cond}

\begin{cond}
\label{cond:wlereg3}
The map
$\Psi = (\Psi_1,\Psi_2):\Theta\times H\mapsto \mathbb{R}^p\times \ell^{\infty}(\mathcal{H})$
with components
\begin{eqnarray*}
&&\Psi_1(\theta,\eta)
\equiv P_0\Psi_{N,1}(\theta,\eta)
=P_0\dot{\ell}_{\theta,\eta},
\nonumber \\
&&\Psi_2(\theta,\eta)h
\equiv P_0\Psi_{N,2}(\theta,\eta)
=P_0B_{\theta,\eta}h-P_{\theta,\eta}B_{\theta,\eta}h,\quad h\in \mathcal{H},
\end{eqnarray*}
has a continuously invertible Fr\'{e}chet derivative map
$\dot{\Psi}_0=(\dot{\Psi}_{11},\dot{\Psi}_{12},\dot{\Psi}_{21},\dot{\Psi}_{22})$
at $(\theta_0,\eta_0)$  given by
$\dot{\Psi}_{ij}(\theta_0,\eta_0)h=P_0(\dot{\psi}_{i,j,\theta_0,\eta_0,h})$,
$i,j\in \{1,2\}$ in terms of $L_2(P_0)$-derivatives of $\psi_{1,\theta,\eta,h}=\dot{\ell}_{\theta,\eta}$
and $\psi_{2,\theta,\eta,h}=B_{\theta,\eta}h-P_{\theta,\eta}B_{\theta,\eta}h$;  that is,
\begin{eqnarray*}
\sup_{h\in \mathcal{H}}
\left\{P_0\left(\psi_{i,\theta,\eta_0,h}-\psi_{i,\theta_0,\eta_0,h}
           -\dot{\psi}_{i1,\theta_0,\eta_0,h}(\theta-\theta_0)\right)^2\right\}^{1/2}
&=&   o(| \theta-\theta_0|),\\
           \sup_{h\in \mathcal{H}}
           \left\{P_0\left(\psi_{i,\theta_0,\eta,h}-\psi_{i,\theta_0,\eta_0,h}
                      -\dot{\psi}_{i2,\theta_0,\eta_0,h}(\eta-\eta_0)\right)^2\right\}^{1/2}
&=&  o(\lVert \eta-\eta_0\rVert).
\end{eqnarray*}
Furthermore, $\dot{\Psi}_0$ admits a partition
\begin{eqnarray*}
(\theta-\theta_0,\eta-\eta)
\mapsto
\left(
\begin{array}{cc}
\dot{\Psi}_{11} & \dot{\Psi}_{12}\\
\dot{\Psi}_{21} & \dot{\Psi}_{22}\\
\end{array}
\right)
\left (
\begin{array}{c}
\theta-\theta_0 \\
\eta-\eta_0\\
\end{array}\right),
\end{eqnarray*}
where
\begin{eqnarray*}
&&\dot{\Psi}_{11}(\theta-\theta_0)
       =-P_{\theta_0,\eta_0}\dot{\ell}_{\theta_0,\eta_0}\dot{\ell}_{\theta_0,\eta_0}^T(\theta-\theta_0),\\
&&\dot{\Psi}_{12}(\eta-\eta_0)=-\int B^*_{\theta_0,\eta_0}\dot{\ell}_{\theta_0,\eta_0}d(\eta-\eta_0),
 \\
&&\dot{\Psi}_{21}(\theta-\theta_0)h
       =-P_{\theta_0,\eta_0}B_{\theta_0,\eta_0}h\dot{\ell}_{\theta_0,\eta_0}^T(\theta-\theta_0),
\\
&&\dot{\Psi}_{22}(\eta-\eta_0)h=-\int B^*_{\theta_0,\eta_0}B_{\theta_0,\eta_0}h d(\eta-\eta_0).
\end{eqnarray*}
Here $B^*_{\theta,\eta}$ is the adjoint of $B_{\theta,\eta}$ and $B^*_{\theta_0,\eta_0}B_{\theta_0,\eta_0}$ is continuously invertible.
\end{cond}

Let $\tilde{I}_0=P_0[(I-B_0(B^*_0B_0)^{-1}B_0^*)\dot{\ell}_0\dot{\ell}_0^T]$
be the efficient information for $\theta$ and
$\tilde{\ell}_0=\tilde{I}_0^{-1}(I-B_0(B^*_0B_0)^{-1}B_0^*)\dot{\ell}_0$
be the efficient influence function for $\theta$ for the semiparametric model with complete data.

The next theorem is Theorem 3.1 of \cite{MR3059418} regarding asymptotic distributions of the WLEs.
\begin{thm}[\cite{MR3059418}]
\label{thm:zthm1}
Under Conditions \ref{cond:cal}, \ref{cond:wlereg1}, \ref{cond:wlereg2}, \ref{cond:wlereg3},
\begin{eqnarray*}
\begin{array}{llll}
\sqrt{N}(\hat{\theta}_N-\theta_0)
&=\quad \mathbb{G}_N^\pi\tilde{\ell}_0+o_{P^*}(1)
&\rightsquigarrow \quad Z
&\sim \quad N_p(0,\Sigma),\\
\sqrt{N}(\hat{\theta}_{N,\#}-\theta_0)
&=\quad \mathbb{G}_N^{\pi,\#}\tilde{\ell}_0+o_{P^*}(1)
&\rightsquigarrow \quad Z_\#
&\sim \quad N_p(0,\Sigma_\#),
\end{array}
\end{eqnarray*}
where $\#\in\{c,cc\}$, and
\begin{eqnarray*}
&&\Sigma \equiv I_0^{-1} + \sum_{j=1}^J\nu_j\frac{1-p_j}{p_j}\mathrm{Var}_{0|j}(\tilde{\ell}_0),\\
&&\Sigma_{\#} \equiv I_0^{-1} + \sum_{j=1}^J\nu_j\frac{1-p_j}{p_j}\mathrm{Var}_{0|j}((I-Q_{\#})\tilde{\ell}_0).
\end{eqnarray*}
\end{thm}

The following theorem ensures that our bootstrap WLEs yield the same asymptotic distributions.
Note that the bootstrap WLEs with bootstrap single calibrations are centered by the WLE, not by the calibrated WLEs.
We discuss this issue in Section \ref{subsec:comp} below.
\begin{thm}
\label{thm:bootzthm1}
Under Conditions \ref{cond:cal}, \ref{cond:bootcal}, \ref{cond:bootwlereg1}-\ref{cond:wlereg3},
\begin{eqnarray*}
\begin{array}{llll}
\sqrt{N}(\hat{\hat{\theta}}_N-\hat{\theta}_N)
&=\quad \tilde{\mathbb{G}}_N^\pi\tilde{\ell}_0+o_{P^*_W}(1)
&\rightsquigarrow \quad Z
&\sim \quad N_p(0,\Sigma),\\
\sqrt{N}(\hat{\hat{\theta}}_{N,b\#}-\hat{\theta}_{N,\#})
&=\quad \tilde{\mathbb{G}}_N^{\pi,b\#}\tilde{\ell}_0+o_{P^*_W}(1)
&\rightsquigarrow \quad Z_\#
&\sim \quad N_p(0,\Sigma_\#), \\
\sqrt{N}(\hat{\hat{\theta}}_{N,bs\#}-\hat{\theta}_{N})
&=\quad \tilde{\mathbb{G}}_N^{\pi,bs\#}\tilde{\ell}_0+o_{P^*_W}(1)
&\rightsquigarrow \quad Z_\#
&\sim \quad N_p(0,\Sigma_\#),
\end{array}
\end{eqnarray*}
in $P^*$-probability where $\#\in\{c,cc\}$.
\end{thm}

\subsection{Non-regular Rate for a Nuisance Parameter}
\label{subsec:5-2}
We consider the case where the infinite-dimensional nuisance parameter may not have a $\sqrt{N}$-convergence rate.
Unlike the previous case, we do not require the WLE solves the weighted likelihood equations for all $h\in\mathcal{H}$.
For $\underline{h}=(h_1,\ldots,h_p)^T$ with $h_k\in \mathcal{H}, k=1,\ldots,p$, let
$B_{\theta,\eta}\left[\underline{h}\right]
=\left(B_{\theta,\eta}h_1,\ldots,B_{\theta,\eta}h_p\right)^T$.
We assume that the WLE $(\hat{\theta}_N,\hat{\eta}_N)$ solves the weighted likelihood equations
\begin{eqnarray}
\label{eqn:wlik2}
&&\Psi_{N,1}^\pi(\theta,\eta,\alpha)=\mathbb{P}_N^{\pi}\dot{\ell}_{\theta,\eta}
        =o_{P^*}\left(N^{-1/2}\right),\nonumber \\
&&\Psi_{N,2}^\pi(\theta,\eta,\alpha)\left[\underline{h}_0\right]
         =\mathbb{P}_N^{\pi} B_{\theta,\eta}[\underline{h}_0]=o_{P^*}\left(N^{-1/2}\right),
\end{eqnarray}
where $\underline{h}_0$ is defined in Condition \ref{cond:wlenonreg2} below.
For the WLE we assume:
\begin{cond}[Consistency and rate of convergence]
\label{cond:wlenonreg1}
An estimator $(\hat{\theta}_N,\hat{\eta}_N)$ of $(\theta_0,\eta_0)$ satisfies
$|\hat{\theta}_N-\theta_0|= o_{P^*}(1)$, and
$\lVert \hat{\eta}_N-\eta_0\rVert = O_{P^*}(N^{-\beta})$ for some $\beta>0$, and solves the weighted likelihood equations (\ref{eqn:wlik2}) where $\mathbb{P}_N^\pi$
may be replaced by $\mathbb{P}_N^{\pi,\#}$ with the corresponding estimators $(\hat{\theta}_{N,\#},\hat{\eta}_{N,\#})$ where $\#\in\{c,cc\}$.
\end{cond}
The corresponding bootstrap WLE solves the bootstrap weighted likelihood equations
\begin{eqnarray}
\label{eqn:wlikboot2}
&&\hat{\Psi}_{N,1}^{\pi}(\theta,\eta,\alpha)=\hat{\mathbb{P}}_N^{\pi}\dot{\ell}_{\theta,\eta}
        =o_{P^*_W}\left(N^{-1/2}\right),\nonumber \\
&&\hat{\Psi}_{N,2}^{\pi}(\theta,\eta,\alpha)\left[\underline{h}_0\right]
         =\hat{\mathbb{P}}_N^{\pi} B_{\theta,\eta}[\underline{h}_0]=o_{P^*_W}\left(N^{-1/2}\right),
\end{eqnarray}
in $P^*$-probability.
For the bootstrap WLEs we assume:
\begin{cond}[Consistency and rate of convergence]
\label{cond:wlenonregboot1}
The bootstrap estimator $(\hat{\hat{\theta}}_{N},\hat{\hat{\eta}}_{N})$  satisfies
$|\hat{\hat{\theta}}_{N}-\theta_0|= o_{P_W^*}(1)$, and
$\lVert \hat{\hat{\eta}}_{N}-\eta_0\rVert = O_{P_W^*}(N^{-\beta})$ in $P^*$-probability for $\beta$ in Condition \ref{cond:wlenonreg1}, and solves the bootstrap weighted likelihood equations (\ref{eqn:wlikboot2}) where $\hat{\mathbb{P}}_N^\pi$ may be replaced by $\hat{\mathbb{P}}_N^{\pi,*\#}$ with the corresponding estimators $(\hat{\hat{\theta}}_{N,*\#},\hat{\hat{\eta}}_{N,\#})$ for $*\in\{b,bs\}$ and $\#\in\{c,cc\}$.
\end{cond}

The rest of conditions are common to WLEs and bootstrap WLEs.
\begin{cond}[Positive information]
\label{cond:wlenonreg2}
There is an $\underline{h}_0= (h_{0,1},\ldots,h_{0,p})$, where
$h_{0,k}\in \mathcal{H}$ for $ k=1,\ldots, p$, such that
\begin{equation*}
P_0 \{  ( \dot{\ell}_{0} -B_{0}[\underline{h}_0]  ) B_{0}h  \}= 0
\end{equation*}
for all $h\in \mathcal{H}$.
Furthermore, the efficient information
$I_0\equiv P_0\left(\dot{\ell}_{0}-B_{0}[\underline{h}_0]\right)^{\otimes 2} $
for $\theta$ for the semiparametric model with complete data
 is finite and nonsingular.
Denote the efficient influence function for the semiparametric model
with complete data by $\tilde{\ell}_0 \equiv I_0^{-1}(\dot{\ell}_{0}-B_{0}[\underline{h}_0])$.
\end{cond}

\begin{cond}[Asymptotic equicontinuity]
\label{cond:wlenonreg3}
(1) For any $\delta_N\downarrow 0$ and $C>0$,
\begin{eqnarray*}
&&\sup_{|\theta-\theta_0|\leq \delta_N,\lVert \eta-\eta_0\rVert\leq CN^{-\beta}}
\left|\mathbb{G}_N(\dot{\ell}_{\theta,\eta}-\dot{\ell}_{0})\right| = o_{P^*}(1),\\
&&\sup_{|\theta-\theta_0|\leq \delta_N,\lVert \eta-\eta_0\rVert\leq CN^{-\beta}}
\left|\mathbb{G}_N(B_{\theta,\eta}-B_{0})[\underline{h}_0]\right| = o_{P^*}(1).
\end{eqnarray*}
(2) There exists a $\delta>0$ such that the classes
$\left\{\dot{\ell}_{\theta,\eta}:|\theta-\theta_0| +\lVert \eta-\eta_0\rVert \leq \delta\right\}$
and $\left\{B_{\theta,\eta}\left[\underline{h}_0\right]:|\theta-\theta_0|+\lVert \eta-\eta_0\rVert\leq \delta\right\}$ are $P_0$-Glivenko-Cantelli and have integrable envelopes.
Moreover, $\dot{\ell}_{\theta,\eta}$ and $B_{\theta,\eta}[\underline{h}_0]$ are
continuous with respect to $(\theta,\eta)$ either pointwise or in $L_1(P_0)$.
\end{cond}

\begin{cond}[Smoothness of the model]
\label{cond:wlenonreg4}
For some $\alpha>1$ satisfying $\alpha \beta >1/2$ and for $(\theta,\eta)$
in the neighborhood $\{(\theta,\eta):|\theta-\theta_0|\leq \delta_N,\lVert \eta-\eta_0\rVert\leq CN^{-\beta}\}$,
\begin{eqnarray*}
&&\left|P_0\left\{\dot{\ell}_{\theta,\eta}-\dot{\ell}_{0}
       +\dot{\ell}_{0}(\dot{\ell}_{0}^T(\theta-\theta_0)
       +B_{0}(\eta-\eta_0))\right\}\right|\\
&&\quad=o\left(|\theta-\theta_0|\right)+O\left(\lVert \eta-\eta_0\rVert^{\alpha}\right),\\
&&\left|P_0\left\{(B_{\theta,\eta}-B_{0})[\underline{h}_0]
         +B_{0}[\underline{h}_0](\dot{\ell}_{0}^T
              (\theta-\theta_0)+B_{0}(\eta-\eta_0))\right\}\right|\\
&&\quad =  o\left(|\theta-\theta_0|\right)+O\left(\lVert \eta-\eta_0\rVert^{\alpha}\right).
\end{eqnarray*}
\end{cond}

The next theorem is the $Z$-theorem for the WLEs (Theorem 3.2 of \cite{MR3059418}).
\begin{thm}[\cite{MR3059418}]
\label{thm:zthm2}
Under Conditions \ref{cond:cal}, \ref{cond:wlenonreg1}, \ref{cond:wlenonreg2}-\ref{cond:wlenonreg4},
\begin{eqnarray*}
\begin{array}{llll}
\sqrt{N}(\hat{\theta}_N-\theta_0)
&=\quad \mathbb{G}_N^\pi\tilde{\ell}_0+o_{P^*}(1)
&\rightsquigarrow \quad Z
&\sim \quad N_p(0,\Sigma),\\
\sqrt{N}(\hat{\theta}_{N,\#}-\theta_0)
&=\quad \mathbb{G}_N^{\pi,\#}\tilde{\ell}_0+o_{P^*}(1)
&\rightsquigarrow \quad Z_\#
&\sim \quad N_p(0,\Sigma_\#),
\end{array}
\end{eqnarray*}
where $\#\in\{c,cc\}$, $\Sigma$ and $\Sigma_\#$ are as
defined in Theorem~ \ref{thm:zthm1},
but now $I_0$ and $\tilde{\ell}_0$ are defined in Condition \ref{cond:wlenonreg2}.
\end{thm}

Our bootstrap $Z$-theorem again yields the same asymptotic distributions.
\begin{thm}
\label{thm:bootzthm2}
Under Conditions \ref{cond:cal}, \ref{cond:bootcal}, \ref{cond:wlenonregboot1}-\ref{cond:wlenonreg4},
\begin{eqnarray*}
\begin{array}{llll}
\sqrt{N}(\hat{\hat{\theta}}_N-\hat{\theta}_N)
&=\quad \tilde{\mathbb{G}}_N^\pi\tilde{\ell}_0+o_{P^*_W}(1)
&\rightsquigarrow \quad Z
&\sim \quad N_p(0,\Sigma),\\
\sqrt{N}(\hat{\hat{\theta}}_{N,b\#}-\hat{\theta}_{N,\#})
&=\quad \tilde{\mathbb{G}}_N^{\pi,b\#}\tilde{\ell}_0+o_{P^*_W}(1)
&\rightsquigarrow \quad Z_\#
&\sim \quad N_p(0,\Sigma_\#),\\
\sqrt{N}(\hat{\hat{\theta}}_{N,bs\#}-\hat{\theta}_N)
&=\quad \tilde{\mathbb{G}}_N^{\pi,bs\#}\tilde{\ell}_0+o_{P^*_W}(1)
&\rightsquigarrow \quad Z_\#
&\sim \quad N_p(0,\Sigma_\#),
\end{array}
\end{eqnarray*}
in $P^*$-probability where $\#\in\{c,cc\}$, $\Sigma$, $\Sigma_\#$, $I_0$, and $\tilde{\ell}_0$ are defined in the same way as in Theorem \ref{thm:zthm2}.
\end{thm}
Rates of convergence of the bootstrap WLEs in Condition \ref{cond:wlenonregboot1} can be established in the same way as those of the WLEs (see Theorem 5.2 of \cite{MR3059418}) if we impose the boundedness of the phase I bootstrap weights.
\begin{lemma}
\label{lemma:rateboot}
Suppose $W_{Ni}^{(1)}\leq M<\infty$ for every $i$ for some constant $M$.
Let $\mathcal{M}=\{m_\theta:\theta\in\Theta\}$ be the set of criterion functions
and define $\mathcal{M}_\delta=\{m_\theta-m_{\theta_0}:d(\theta,\theta_0)<\delta\}$
for some fixed $\delta>0$ where $d$ is a semimetric on the parameter space $\Theta$.

\noindent (1)
Suppose that for every $\theta$ in a neighborhood of $\theta_0$,
\begin{equation*}
P_0(m_\theta-m_{\theta_0})\lesssim
-d^2(\theta,\theta_0).
\end{equation*}
Assume that there exists a function $\phi_N$ such that $\delta\mapsto \phi_N(\delta)/\delta^\alpha$
is decreasing for some $\alpha<2$ (not depending on $N$) and for every $N$,
\begin{equation*}
E^*\lVert \mathbb{G}_N\rVert_{\mathcal{M}_\delta}
\lesssim \phi_N(\delta).
\end{equation*}
If the estimator $\hat{\hat{\theta}}_N$ satisfies
$\hat{\mathbb{P}}_N^\pi m_{\hat{\hat{\theta}}_N}\geq \hat{\mathbb{P}}_N^\pi m_{\theta_0}-O_{P^*_W}(r^{-2}_N)$ and $\hat{\hat{\theta}}_N =\theta_0+o_{P_W^*}(1)$ in $P^*$-probability,
then $r_Nd(\hat{\hat{\theta}}_N,\theta_0)=O_{P^*_W}(1)$ in $P^*$-probability
for every sequence $r_N$ such that $r_N^2\phi_N(1/r_N)\leq \sqrt{N}$ for every $N$.

\noindent (2)
Suppose Conditions \ref{cond:cal} and \ref{cond:bootcal} hold and let  $*\in\{b,bs\}$ and $\#\in\{c,cc\}$ be fixed.
Suppose also that for every $\theta\in\Theta$ in a neighborhood of $\theta_0$,
\begin{equation*}
P_0\{G_{\#}(V;\alpha)(m_\theta-m_{\theta_0})\}\lesssim
-d^2(\theta,\theta_0) + |\alpha-\alpha_0|^2.
\end{equation*}
Assume that
\begin{equation*}
E^*\lVert \mathbb{G}_N\rVert_{G^\#\mathcal{M}_\delta}
\lesssim \phi_N(\delta),
\end{equation*}
with $\phi_N$ having the same properties as in (1) where
$G^{\#}\mathcal{M}_\delta\equiv \{G_{\#}(\cdot; \alpha)f:|\alpha|\leq \delta,\alpha\in\mathcal{A}_N,f\in\mathcal{M}_\delta\}$
for some $\mathcal{A}_N\subset \mathcal{A}$.
If the estimator $\hat{\hat{\theta}}_{N,*\#}$ satisfies $\hat{\mathbb{P}}_N^{\pi,*\#} m_{\hat{\hat{\theta}}_{N,*\#}}\geq \hat{\mathbb{P}}_N^{\pi,*\#} m_{\theta_0}-O_{P^*_W}(r^{-2}_N)$ and $\hat{\hat{\theta}}_{N,*\#}=\theta_0+o_{P_W^*}(1)$ in $P^*$-probability,
then $r_Nd(\hat{\hat{\theta}}_{N,*\#},\theta_0)=O_{P^*_W}(1)$ in $P^*$-probability
for every sequence $r_N$ such that $r_N^2\phi_N(1/r_N)\leq \sqrt{N}$ for every $N$.
\end{lemma}

\subsection{Comparison of Bootstrap Calibrations}
\label{subsec:comp}
It is expected from the plug-in principle that bootstrap asymptotic distributions should involve centering by the original estimators.
In this view, the bootstrap WLEs with bootstrap calibrations have ``right'' centering by the corresponding calibrated WLEs in Theorems \ref{thm:bootzthm1} and \ref{thm:bootzthm2}.
In contrast, the bootstrap WLEs with bootstrap single calibrations  have centering by the plain-vanilla WLE $\hat{\theta}_N$ to yield the same asymptotic distributions of the calibrated WLEs.
The next corollary concerns the centering of the bootstrap WLE with bootstrap single calibration by calibrated WLE either unconditionally or conditionally on data.

\begin{cor}
Under the Conditions of Theorems \ref{thm:bootzthm1} and \ref{thm:bootzthm2},
\begin{equation*}
\sqrt{N}(\hat{\hat{\theta}}_{N,bs\#}-\hat{\theta}_{N,\#})
\rightsquigarrow Z_{\#} + O_{P_W^*}(1),
\end{equation*}
in $P^*$-probability, and unconditionally
\begin{equation*}
\sqrt{N}(\hat{\hat{\theta}}_{N,bs\#}-\hat{\theta}_{N,\#})
\rightsquigarrow Z_{bs\#}\sim N(0,\Sigma_{bs\#}),
\end{equation*}
where $\# \in \{c,cc\}$ and
\begin{eqnarray*}
\Sigma_{bs\#} \equiv I_0^{-1} + \sum_{j=1}^J\nu_j\frac{1-p_j}{p_j}\mathrm{Var}_{0|j}((I-2Q_{\#})\tilde{\ell}_0).
\end{eqnarray*}
\end{cor}

A practical implication of this result depends on a purpose of bootstrap inference.
If variance estimation is of interest, then centering does not matter because non-bootstrap WLEs are constant in a bootstrap sample.
If bias correction is of interest, a more careful consideration would be required to determine which calibration method should be used for bias correction of which estimator.
Bias correction is beyond the scope of the present paper, and we only show this phenomenon in simulation studies in the next section.

\section{Numerical Results}
\label{sec:6}
\subsection{Simulation}
We apply our bootstrap procedure to the weighted likelihood estimation for the Cox model with right censoring.
In this model, the efficient influence function in the complete data model is known up to parameters.
Thus the standard estimator of variance can be computed from the IPW sample variances at the estimated parameters across and within strata (see page 285 of \cite{MR3059418} for details).
We can use this standard estimator as a benchmark to evaluate our bootstrap estimator of variance.
We first generated 1000 data sets to see finite sample properties of the WLEs. Then we chose several data sets and generated 1000 bootstrap samples based on each data set.

Let $Y=\min\{T,C\}$ be the minimum of time to event $T$ or censoring time $C\sim\mbox{Unif}(0,1.1)$ with censoring indicator $\Delta=I(T\leq C)$.
A binary variable $V$ as well as $Y$ and $\Delta$ are available for all observations while the exposure $X\in\{0,1\}$ of interest is only available for a subsample.
The exposure $X$ has prevalence 50 percent, and is related to $V$ by sensitivity $P(V = 1|X=1) = \alpha$ and specificity $P(V = 0|X=0) = \beta$.
Three strata are formed based on $V$ and $\Delta$: a stratum of uncensored observations (Strata 1), a stratum of censored observations with $V=0$ (Strata 2) and a stratum of censored observations with $V=1$ (Strata 3) with sampling probabilities $P(\xi=1|\Delta=1)=1$, $P(\xi=1|\Delta=0, V=0)=[.3N_2]/N_2$, and $P(\xi=1|\Delta=0, V=1)=[.3N_3]/N_3$.
The hazard function is given by
\begin{equation*}
\lambda(t|x)=\lambda_0(t)\exp(\theta x).
\end{equation*}
where $\lambda_0$ is the baseline hazard function and $\theta$ is a regression coefficient.
We take $\lambda_0=.1$ in our simulations.
Calibration and within-stratum centered calibration \cite{MR3059418} are carried out on $Y$.

Table \ref{tbl:size} shows averages of phase I and II sample sizes across strata, and censoring proportions at the first phase across simulations.
We compare two different sample sizes (small/large) and two different correlation structure for $X$ and $V$ (strongly correlated/uncorrelated).

\begin{table}[htbp]
\begin{tabular}{|c|c|c|c|c|c|c|c|}\hline
 $\theta$&$\alpha,\beta$&$N$& $n$& Strata 1 & Strata 2 & Strata 3 & Cens Prop\\\hline
$\log 2$&(.9,.9)&400&142&31 (31)&54(181)&57(188)& .922\\\hline
$\log 2$&(.9,.9)&800&280&62(62)&108(362)&113(376)&.922\\\hline
$\log 2$&(.5,.5)&400&141&31(31)&55(184)&55(185)&.922\\\hline
$\log 2$&(.5,.5)&800&283&62(62)&110(368)&111(370)&.922\\\hline
\end{tabular}
\caption{Sample size and censoring proportion}
\label{tbl:size}
\end{table}

Table \ref{tbl:res} summarizes results from our simulations.
We select three data sets for each case, and  compare our bootstrap estimators with the WLEs and their standard variance estimators based on a single data set.
This comparison is more appropriate to our conditional asymptotic results given data than comparison between empirical means and variances and averages of corresponding bootstrap estimators over all data sets.
The latter would be suitable for joint asymptotic results but these would be less interesting and outside the scope of our paper.

As expected from our theoretical results, our bootstrap mean and variance of WLEs in all cases well approximate results based on an original sample from which bootstrap samples were generated.
Of particular interest is that the means of the bootstrap WLEs with bootstrap single calibrations are closer to the mean of the plain vanilla WLE than the means of corresponding calibrated WLEs.
This is in line with our theoretical results discussed in Section \ref{subsec:comp}.
As clearly seen from difference among data sets, our bootstrap estimates depend on an original sample.
As $N$ becomes larger, an original sample is expected to yield WLEs and their standard variance estimates closer to the corresponding population quantities with high probability whereby our bootstrap estimators would behave ``well'' as expected from our ``in probability'' statements in Sections \ref{subsec:5-1} and \ref{subsec:5-2}.
This is also seen in our simulation when increasing the phase I sample size from $N=400$ to $N=800$.

[Table \ref{tbl:res} is about here.]

\subsection{Data analysis}
We analyze data from the National Wilms Tumor Study \cite{pmid2544249,pmid9440748}.
In this study, 3915 patients with Wilms tumor diagnosed during 1980-1994 were followed until the disease progression or death.
The baseline covariates are age at diagnosis, stage of disease (I-IV), histology (favorable/unfavorable) from the registering institution and the central reference laboratory, and tumor diameter.
We took a subsample from this study to create a two-phase design as considered in \cite{Breslow:2009a, Breslow:2009b}.
Because variables were measured for all patients, we compare WLEs with the MLE with complete data.
Nine strata were formed based on age (less than or greater than one year of age), severity of stage (I-II versus III-IV), and institutional histology in addition to a censoring indicator.
Moreover, histology from the central reference laboratory was treated as the gold standard (sensitivity 74$\%$ and specificity $98\%$) only known for patients sampled at the second phase.
At the second phase all patients were sampled except three strata.
For the first stratum, 120 patients were sampled from 452 patients with favorable institutional histology, stage I or II and less than one year of age.
For the second stratum, 160 patients were sampled from 1620 patients with favorable institutional histology, stage I or II and greater than one year of age.
For the third stratum, 120 patients were sampled fro-rm 914 patients with favorable institutional histology, stage III or IV and greater than one year of age.
The overall phase II sample size is 1329.
See \cite{Breslow:2009a, Breslow:2009b} for more details.

The statistical model is the Cox model with right censoring as in \cite{Breslow:2009a, Breslow:2009b}.
The estimators considered are the plain WLE and the WLE with within-stratum centered calibration on the time to event or censoring, stage (I-IV), and age (continuous).
Table \ref{tbl:nwt1} summarizes results from three estimators for one simulated data set, and corresponding bootstrap estimators.
All point estimates based on a single data set are similar to each other and their $95\%$ confidence intervals all include point estimates of the MLE based on the complete data.
The within-stratum centered calibration improved efficiency over the plain WLE unlike calibration.
These results were well approximated by our bootstrap method.

[Table \ref{tbl:nwt1} is about here.]


\section{Appendix}

\subsection{Additional Notations}
We introduce several notations.
For a probability space $(\Omega,\mathcal{A},P)$ and a map $T:\Omega \rightarrow \mathbb{R}\cup \{\pm \infty\}$, we denote the outer and inner expectations by $E^*$ and $E_*$, and denote the minimal measurable majorant and maximal measurable minorant of $T$ by $T^*$ and $T_*$.
See Section 1.1 of \cite{MR1385671} for their precise definitions and basic results.
We work on several different probability spaces so that these definitions should be understood with a suitable probability space depending on the context.
We omit specifying a probability space unless confusion arises.
For convenience, we let $\tilde{\pi}_0(v) \equiv (1-\pi_0(v))/\pi_0(v)$ and $\tilde{\pi}_\infty(v) \equiv (1-\pi_{0,\infty}(v))/\pi_{0,\infty}(v)$ and let also $G_{cc,\infty}(V;\alpha) = G(\tilde{\pi}_\infty(V)\tilde{V})$.

To study bootstrap IPW empirical processes, we define
bootstrap IPW empirical measures and processes at the first and second
phases for every stratum.
These definitions are motivated by the proof of Theorem
\ref{thm:fpsD} in \cite{MR2325244}, which we briefly discuss here.
The IPW empirical process can be written as
\begin{eqnarray*}
 \mathbb{G}_N^\pi
&=& \mathbb{G}_N  +\sum_{j=1}^J \sqrt{\frac{N_j}{N}}
         \left(\frac{N_j}{n_j}\right)\mathbb{G}_{j,N_j}^\xi
\equiv \mathbb{G}_N + \mathbb{G}_N^{\pi,(2)}
\end{eqnarray*}
where
$\mathbb{G}_{j,N_j}^\xi\equiv\sqrt{N_j}(\mathbb{P}_{j,N_j}^\xi-(n_j/N_j)\mathbb{P}_{j,N_j})$ is the finite sampling empirical process for the $j$th stratum with
$\mathbb{P}_{j,N_j}^\xi \equiv
N_j^{-1}\sum_{i=1}^{N_j}\xi_{j,i}\delta_{X_{j,i}}$,
$\mathbb{P}_{j,N_j}\equiv N_j^{-1}\sum_{i=1}^{N_j}\delta_{X_{j,i}}$,
$j=1,\ldots,J$ (see \cite{MR2325244}).
We also denote $\mathbb{P}_{j,n_j}^\xi \equiv (N_j/n_j)\mathbb{P}_{j,N_j}^\xi,j=1,\ldots,J$.
\cite{MR2325244} established weak convergence of $\mathbb{G}_N$
and $\mathbb{G}_{j,N_j}^\xi, j=1,\ldots,J$, piece by piece.
Our proof extends this idea of decomposition to bootstrap.
As seen in the next paragraph, however, our decomposition of the bootstrap IPW
empirical process is not completely parallel to the decomposition
above.
This is because our proofs require more involved arguments.

We define the corresponding bootstrap IPW empirical measures and
processes.
Recall that \begin{equation*}
\tilde{\mathbb{G}}_N^{\pi}=\tilde{\mathbb{G}}_N^{\pi,(1)} + \tilde{\mathbb{G}}_N^{\pi,(2)}W^{(1)}_N\cdot,
\end{equation*}
Here the phase I bootstrap IPW empirical process $\tilde{\mathbb{G}}_N^{\pi,(1)}$ and
the phase II bootstrap IPW empirical process $\tilde{\mathbb{G}}_N^{\pi,(2)}$ are
\begin{eqnarray*}
&&\tilde{\mathbb{G}}_N^{\pi,(1)}=\sqrt{N}(\hat{\mathbb{P}}_N^{\pi,(1)}-\mathbb{P}_N^\pi),
\quad \tilde{\mathbb{G}}_N^{\pi,(2)}=\sqrt{N}(\hat{\mathbb{P}}_N^{\pi,(2)}-\mathbb{P}_N^\pi),
\end{eqnarray*}
where the phase I and II bootstrap IPW empirical measures are defined by
\begin{eqnarray*}
&&\hat{\mathbb{P}}_N^{\pi,(1)} \equiv \frac{1}{N}\sum_{i=1}\frac{\xi_i}{\pi_0(V_i)}W_{Ni}^{(1)}\delta_{X_i} = \mathbb{P}_N^\pi W_N^{(1)}\cdot,\\
&&\hat{\mathbb{P}}_N^{\pi,(2)} \equiv \frac{1}{N}\sum_{i=1}\frac{\xi_i}{\pi_0(V_i)}W_{Ni}^{(2)}\delta_{X_i} = \mathbb{P}_N^\pi W_N^{(2)}\cdot.
\end{eqnarray*}
The bootstrap IPW empirical measure and process for the $j$th stratum are
\begin{eqnarray*}
\hat{\mathbb{P}}_{j,n_j}^{\xi} \equiv \frac{1}{n_j}\sum_{i=1}^{N_j}W_{N_j,j,i}^{(1)}W_{n_j,j,i}^{(2)}\xi_{j,i}\delta_{X_{j,i}},\quad
\tilde{\mathbb{G}}_{j,n_j}^{\xi}&\equiv&\sqrt{n_j}(\hat{\mathbb{P}}_{j,n_j}^{\xi}-\mathbb{P}_{j,n_j}^{\xi}).
\end{eqnarray*}
The phase I bootstrap IPW empirical measure and process for the $j$th stratum are
\begin{eqnarray*}
\hat{\mathbb{P}}_{j,n_j}^{\xi,(1)} \equiv \frac{1}{n_j}\sum_{i=1}^{N_j}W_{N_j,j,i}^{(1)}\xi_{j,i}\delta_{X_{j,i}},\quad
\tilde{\mathbb{G}}_{j,n_j}^{\xi,(1)}&\equiv&\sqrt{n_j}(\hat{\mathbb{P}}_{j,n_j}^{\xi,(1)}-\mathbb{P}_{j,n_j}^{\xi}),
\end{eqnarray*}
and the phase II bootstrap IPW empirical measure and process for the $j$th stratum are
\begin{eqnarray*}
\hat{\mathbb{P}}_{j,n_j}^{\xi,(2)} \equiv \frac{1}{n_j}\sum_{i=1}^{N_j}W_{n_j,j,i}^{(2)}\xi_{j,i}\delta_{X_{j,i}},\quad
\tilde{\mathbb{G}}_{j,n_j}^{\xi,(2)}&\equiv&\sqrt{n_j}(\hat{\mathbb{P}}_{j,n_j}^{\xi,(2)}-\mathbb{P}_{j,n_j}^{\xi}).
\end{eqnarray*}
Note that $n_j$ appears in these definition in contrast to non-bootstrap cases (page 90 of \cite{MR2325244}).
With these notation we have
\begin{eqnarray*}
\hat{\mathbb{P}}_N^{\pi}=\sum_{j=1}^J\frac{N_j}{N}\hat{\mathbb{P}}^{\xi}_{j,n_j},\quad
\hat{\mathbb{P}}_N^{\pi,(1)}=\sum_{j=1}^J\frac{N_j}{N}\hat{\mathbb{P}}^{\xi,(1)}_{j,n_j},\quad
\hat{\mathbb{P}}_N^{\pi,(2)}=\sum_{j=1}^J\frac{N_j}{N}\hat{\mathbb{P}}^{\xi,(2)}_{j,n_j},
\end{eqnarray*}
and
\begin{eqnarray*}
&&\tilde{\mathbb{G}}_N^{\pi}
=\sum_{j=1}^J\sqrt{\frac{N_j}{N}}\sqrt{\frac{N_j}{n_j}}\tilde{\mathbb{G}}_{j,n_j}^{\xi},\quad
\tilde{\mathbb{G}}_N^{\pi,(1)}
=\sum_{j=1}^J\sqrt{\frac{N_j}{N}}\sqrt{\frac{N_j}{n_j}}\tilde{\mathbb{G}}_{j,n_j}^{\xi,(1)},\\
&& \tilde{\mathbb{G}}_N^{\pi,(2)}
=\sum_{j=1}^J\sqrt{\frac{N_j}{N}}\sqrt{\frac{N_j}{n_j}}\tilde{\mathbb{G}}_{j,n_j}^{\xi,(2)}.
\end{eqnarray*}
To see these, note, for example, that
\begin{eqnarray*}
&&\hat{\mathbb{P}}_N^{\pi,(1)}
=\frac{1}{N}\sum_{j=1}^J\frac{1}{n_j/N_j}\sum_{i=1}^{N_j}W_{n_j,j,i}^{(1)}\xi_{j,i}\delta_{X_{j,i}}
=\sum_{j=1}^J\frac{N_j}{N}\hat{\mathbb{P}}^{\xi,(1)}_{j,n_j},\\
&&\tilde{\mathbb{G}}_N^{\pi,(1)}
=\sum_{j=1}^J\frac{N_j}{\sqrt{N}}\left(\hat{\mathbb{P}}_{j,n_j}^{\xi,(1)}-\mathbb{P}^\xi_{j,n_j}\right)
=\sum_{j=1}^J\sqrt{\frac{N_j}{N}}\sqrt{\frac{N_j}{n_j}}\tilde{\mathbb{G}}_{j,n_j}^{\xi,(1)}.
\end{eqnarray*}
As seen in the expression of $\tilde{\mathbb{G}}_N^{\pi,(1)}$ as the
linear combination of $\tilde{\mathbb{G}}_{j,n_j}^{\xi,(1)}$, the
phase I bootstrap IPW empirical process does not have a limit process
$\mathbb{G}$ unlike $\mathbb{G}_N$ in the decomposition of
$\mathbb{G}_N^\pi$.
This difficulty seems inevitable except a trivial case of a single stratum
because we start from a biased sample to obtain a bootstrap sample.

We also define the phase I and II bootstrap empirical measures and processes for bootstrap calibrations.
Define the phase I and II bootstrap IPW empirical measures for bootstrap calibration by
\begin{eqnarray*}
&&\hat{\mathbb{P}}_N^{\pi,(1),bc} \equiv \frac{1}{N}\sum_{i=1}^NW_{Ni}^{(1)}\frac{\xi_i}{\pi_0(V_i)}G_c(V_i;\hat{\alpha}^{c}_N)\delta_{X_i},\\
&&\hat{\mathbb{P}}_N^{\pi,(2),bc} \equiv \frac{1}{N}\sum_{i=1}^NW_{Ni}^{(2)}\frac{\xi_i}{\pi_0(V_i)}G_c(V_i;\hat{\hat{\alpha}}^{bc}_N)\delta_{X_i},
\end{eqnarray*}
or $\hat{\mathbb{P}}_N^{\pi,(1),bc}\cdot= \hat{\mathbb{P}}_N^{\pi,(1)}G_c(V;\hat{\alpha}^{c}_N)\cdot$ and $\hat{\mathbb{P}}_N^{\pi,(2),bc}\cdot= \hat{\mathbb{P}}_N^{\pi,(2)}G_c(V;\hat{\hat{\alpha}}^{bc}_N)\cdot$, and define the phase I and II bootstrap IPW empirical processes with bootstrap calibration by $\tilde{\mathbb{G}}_N^{\pi,(1),bc} = \sqrt{N}(\hat{\mathbb{P}}_N^{\pi,(1),bc}-\mathbb{P}_N^{\pi,c})$ and $\tilde{\mathbb{G}}_N^{\pi,(2),bc} = \sqrt{N}(\hat{\mathbb{P}}_N^{\pi,(2),bc}-\mathbb{P}_N^{\pi,c})$, respectively.
Similarly for other bootstrap calibrations, we define the phase I and II bootstrap IPW empirical measures by
\begin{eqnarray*}
&&\hat{\mathbb{P}}_N^{\pi,(1),bcc}\cdot \equiv \hat{\mathbb{P}}_N^{\pi,(1)}G_{cc}(V;\hat{\alpha}^{cc}_N)\cdot,\quad \hat{\mathbb{P}}_N^{\pi,(2),bcc}\cdot \equiv \hat{\mathbb{P}}_N^{\pi,(2)}G_{cc}(V;\hat{\hat{\alpha}}^{bcc}_N)\cdot,\\
&&\hat{\mathbb{P}}_N^{\pi,(2),bsc}\cdot \equiv \hat{\mathbb{P}}_N^{\pi,(2)}G_{c}(V;\hat{\hat{\alpha}}^{bsc}_N)\cdot,
\quad \hat{\mathbb{P}}_N^{\pi,(2),bscc}\cdot \equiv \hat{\mathbb{P}}_N^{\pi,(2)}G_{cc}(V;\hat{\hat{\alpha}}^{bscc}_N)\cdot,
\end{eqnarray*}
and the phase I and II bootstrap IPW empirical processes by
\begin{eqnarray*}
&&\tilde{\mathbb{G}}_N^{\pi,(1),bcc} \equiv \sqrt{N}(\hat{\mathbb{P}}_N^{\pi,(1),bcc}-\mathbb{P}_N^{\pi,cc}), \quad \tilde{\mathbb{G}}_N^{\pi,(2),bcc} \equiv \sqrt{N}(\hat{\mathbb{P}}_N^{\pi,(2),bcc}-\mathbb{P}_N^{\pi,cc}),\\
&&\tilde{\mathbb{G}}_N^{\pi,(2),bsc} \equiv \sqrt{N}(\hat{\mathbb{P}}_N^{\pi,(2),bsc}-\mathbb{P}_N^{\pi}) ,
\quad\tilde{\mathbb{G}}_N^{\pi,(2),bscc} \equiv \sqrt{N}(\hat{\mathbb{P}}_N^{\pi,(2),bscc}-\mathbb{P}_N^{\pi}) ,
\end{eqnarray*}
respectively.
Note also that the phase I bootstrap IPW empirical processes corresponding to the single bootstrap calibrations are $\tilde{\mathbb{G}}_N^{\pi,(1)}$.

\subsection{Proofs}
\subsubsection{Order Notations for bootstrap}
\begin{proof}[Proof of Lemma \ref{lemma:bootorder}]

The statements (1) and (2) were proved in \cite{WellnerZhan:96} for
the case regarding the little $o$ notation (see also proof of Lemma 3 of
\cite{MR2722459} under the measurability assumption).
Thus we only prove the case regarding big $O$ notation for these statements (one can prove the case of the little $o$ notation based on a proof below).
We also omit proofs for the claim regarding $P^*_{W^{(k)}},k=1,2,$ since proofs are similar.
Let $\eta>0$ be an arbitrary constant.\\
\noindent (1) Let $M_N$ be an arbitrary sequence such that $M_N\rightarrow \infty$.
Suppose $\Delta_N=O_{Pr^*}(1)$.
Markov's inequality yields
\begin{eqnarray*}
P^*\left\{P_W^*(|\Delta_N|>M_N)>\eta\right\}
&\leq& \eta^{-1}E^*P_W^*(|\Delta_N|>M_N),
\end{eqnarray*}
where $E^*$ is the outer expectation with respect to $P^\infty$.
Apply Fubini's theorem (Lemma 1.2.6 of \cite{MR1385671}) to obtain
$E^*P_W^*(|\Delta_N|>M_N)
\leq Pr^*(|\Delta_N|>M_N).$
Since $Pr^*(|\Delta_N|>M_N)\rightarrow 0$ as $N\rightarrow \infty$ by assumption, we conclude that $\Delta_N=O_{P_W^*}(1)$ in $P^*$-probability.

For the second statement, let $\eta>0$ be arbitrary, and suppose that $\Delta_N=O_{P_W^*}(1)$ in $P^*$-probability and that $\Delta_N$ is measurable.
For every sequence $M_N\rightarrow \infty$, $P^*\left\{P_W^*(|\Delta_N|>M_N)>\eta\right\}\rightarrow 0$ as $N\rightarrow \infty$ by the definition of $\Delta_N=O_{P_W^*}(1)$ in $P^*$-probability.
Apply Fubini's theorem to obtain
\begin{eqnarray*}
Pr^*(|\Delta_N|>M_N)
&=&E^*\left[P_W^*(|\Delta_N|>M_N)I\{P_W^*(|\Delta_N|>M_N)>\eta\}\right]\\
&&\quad +E^*\left[P_W^*(|\Delta_N|>M_N)I\{P_W^*(|\Delta_N|>M_N)\leq \eta\}\right]\\
&\leq &E^*\left[P_W^*(|\Delta_N|>M_N)I\{P_W^*(|\Delta_N|>M_N)>\eta\}\right]+\eta\\
&\leq &E^*I\{P_W^*(|\Delta_N|>M_N)>\eta\}+\eta\\
&\leq &P^*\{P_W^*(|\Delta_N|>M_N)>\eta\}+\eta.
\end{eqnarray*}
Since $\Delta_N$ is $O_{P_W^*}(1)$ in $P^*$-probability, $P^*\{P_W^*(|\Delta_N|>M_N)>\eta\}\rightarrow 0$ as $N\rightarrow \infty$.
Since $\eta$ is arbitrary, we conclude that $Pr^*(|\Delta_N|>M_N) \rightarrow 0$ as $N\rightarrow \infty$.
This establishes the desired result.

\noindent (2) Let $M_N$ be an arbitrary sequence such that $M_N\rightarrow\infty$ as $N\rightarrow \infty$.
Because $\Delta_N$ is only defined on the probability space $(\mathcal{X}^\infty,\mathcal{B}^\infty,P^\infty)$, applying Lemma 1.2.3 of \cite{MR1385671} twice to obtain
\begin{eqnarray*}
Pr^*(|\Delta_N|\geq M_N) &=& PrI\{|\Delta_N|\geq M_N\}^*
=P^\infty I\{|\Delta_N|\geq \eta\}^* \\
&=& P^*(|\Delta_N|\geq
M_N)\rightarrow 0,\quad \mbox{ as }N\rightarrow \infty,
\end{eqnarray*}
as desired where $I$ is an indicator function of an event.

\noindent (3)
Note that for events $A$ and $B$ with $A\subset B$, $P^*(A)\subset P^*(B)$ by the definition
of the outer probability (see \cite{MR1385671}).
Note also that $P^*(A\cup B) \leq P^*(A)+P^*(B)$.
To see this, note that  $P^*(A) = P(A^*)$ (Lemma 1.2.3 of \cite{MR1385671}) and
$(S+T)^* \leq S^*+T^*$ for maps $S,T$ on the probability space (Lemma 1.2.2 of \cite{MR1385671}).
Thus,
\begin{eqnarray*}
&&P^*(A\cup B)
= E(I_A + I_{B\setminus A})^*
\leq EI_A^* +   EI_{B\setminus A}^*\\
&&= P^*(A) + P^*(B\setminus A)
\leq P^*(A) + P^*(B).
\end{eqnarray*}

Now, let $\epsilon>0$ be arbitrary.
Because $\Delta_N=o_{P^*_W}(1)$ in $P^*$-probability, there exists a sequence
$\epsilon_N\downarrow 0$ as $N\rightarrow 0$ such that
$P^*(P_W^*(|\Delta_N|\geq \epsilon_N)\geq \eta/2) \rightarrow 0$ as
$N\rightarrow \infty$.
Thus, it follows from the results on outer probability above that
\begin{eqnarray*}
&&P^*\left\{P_W^*(|\Delta_N\Gamma_N|>\epsilon)>\eta\right\}\\
&&= P^*\left\{P_W^*(|\Delta_N\Gamma_N|>\epsilon,|\Gamma_N|>\epsilon/\epsilon_N)+
  P_W^*(|\Delta_N\Gamma_N|>\epsilon,|\Gamma_N|\leq
  \epsilon/\epsilon_N)>\eta\right\}\\
&&\leq P^*\left\{P_W^*(|\Gamma_N|>\epsilon/\epsilon_N)+
  P_W^*(|\Delta_N|>\epsilon_N)>\eta\right\}\\
&&\leq P^*\left\{P_W^*(|\Gamma_N|>\epsilon/\epsilon_N)>\eta/2\right\}
  + P^*\left\{P_W^*(|\Delta_N|>\epsilon_N)>\eta/2\right\}\\
&&\rightarrow 0,\quad \mbox{as }N\rightarrow \infty.
\end{eqnarray*}

\noindent (4) Let $M_N$ be an arbitrary sequence such that $M_N\rightarrow
\infty$ as $N\rightarrow \infty$.
As in the proof of (3) we have
\begin{eqnarray*}
&&P^*\left\{P_W^*(|\Delta_N\Gamma_N|>M_N)>\eta\right\}\\
&&= P^*\left\{P_W^*(|\Delta_N\Gamma_N|>\epsilon,|\Gamma_N|>M_N^{1/2})+
  P_W^*(|\Delta_N\Gamma_N|>\epsilon,|\Gamma_N|\leq
  M_N^{1/2})>\eta\right\}\\
&&\leq P^*\left\{P_W^*(|\Gamma_N|>M_N^{1/2})+
  P_W^*(|\Delta_N|>M_N^{-1/2}M_N>\eta\right\})\\
&&\leq P^*\left\{P_W^*(|\Gamma_N|>M_N^{1/2})>\eta/2\right\}
  + P^*\left\{P_W^*(|\Delta_N|>M_N^{1/2})>\eta/2\right\}\\
&&\rightarrow 0,\quad \mbox{as }N\rightarrow \infty.
\end{eqnarray*}

\noindent (5) Let $\epsilon>0$ be arbitrary, and let $\Delta_{N'}$ be an arbitrary subsequence of $\Delta_N$.
Suppose that $\Delta_N=o_{P_W^*}(1)$ in $P^*$-probability.
It follows from Lemma 1.9.2 of \cite{MR1385671} that there exists a further subsequence $\{N^{''}\}$ of $\{N^{'}\}$ (depending on $\epsilon$) such that $P^*_W(|\Delta_{N^{''}}|>\epsilon) \rightarrow 0$, $P^{\infty}$-almost surely.
This implies that there exists a set $N_X\in \mathcal{B}^\infty$ such that for every $x \notin N_X$, $P^*_W(|\Delta_{N^{''}}|>\epsilon) \rightarrow 0$ and $P^\infty(N_X) =0$.
Fix $x\notin N_X$.
It follows from Lemma 1.9.2 of \cite{MR1385671} again that there exists a further subsequence $\{N^{'''}\}$ of $\{N^{''}\}$ such that $|\Delta_{N^{'''}}|^{*} \rightarrow 0$, $P_W$-almost surely. Here $|\Delta_{N^{'''}}|^{*}$  is a minimal measurable majorant of $|\Delta_{N^{'''}}|$ with respect to $P_W$.
This implies that there exists a set $N_{W|x}\in \mathcal{E}$ such that for every $w \in N_{W|x}^c$, $|\Delta_{N^{'''}}|^{*}\rightarrow 0$ and $P_W(N_{W|x}) =0$.
Hence $\Delta_{N^{'}}$ has a subsequence $\Delta_{N^{'''}}$ such that $\Delta_{N^{'''}}\rightarrow 0$ in $P_{W|\infty}$-almost surely. This establishes the first half of the statement.
Now suppose that every subsequence $\Delta_{N'}$ has a further subsequence $\Delta_{N^{''}}$ such that $\Delta_{N^{''}}^*\rightarrow 0$ in  $P_{W|\infty}$-almost surely.
Fix $x\notin N_X$.
It follows from Lemma 1.9.2 of \cite{MR1385671} and the assumption that there exists a subsequence $\{N^{''}\}$ of $\{N'\}$ such that $P_W^*(|\Delta_{N^{''}}|>\epsilon)\rightarrow 0$ for a fixed $x$.
Apply Lemma 1.9.2 of \cite{MR1385671} again to verify that there exists a further subsequence $\{N^{'''}\}$ of $\{N^{''}\}$ such that $P^*(P^*_W(|\Delta_{N^{'''}}|>\epsilon)>\eta)\rightarrow 0$.
Thus every subsequence of $\Delta_N$ is $o_{P_W^*}(1)$ in $P^*$-probability. This implies that $\Delta_N=o_{P_W^*}(1)$ in $P^*$-probability.
This completes the proof.
\end{proof}

\subsubsection{Calibration}
Consistency, and the limiting distributions of $\hat{\hat{\alpha}}_N^{*\#}$ are given by the following proposition.
The proof requires a Glivenko-Cantelli theorem for $\hat{\mathbb{P}}_N^\pi$ whose proof is independent of Proposition \ref{prop:alphaCalboot}.

\begin{prop}
\label{prop:alphaCalboot}
Suppose that Conditions \ref{cond:cal} and \ref{cond:bootcal} hold.
Let \begin{eqnarray*}
&& A_N \equiv -\dot{G}(0)^{-1}\left\{P_0 V^{\otimes 2}\right\}^{-1}
\tilde{\mathbb{G}}_N^{\pi,(2)} V , \\
&&A\equiv -\dot{G}(0)^{-1}\left\{P_0 V^{\otimes 2}\right\}^{-1}
         \sum_{j=1}^J\{\nu_jp_j^{-1}(1-p_j)\}^{1/2}\mathbb{G}_jV,\\
&&B_N \equiv -\dot{G}(0)^{-1}\left\{P_0 \tilde{\pi}_\infty(V)\tilde{V}^{\otimes 2}\right\}^{-1}
\tilde{\mathbb{G}}_N^{\pi,(2)}\tilde{V},\\
&& B\equiv  -\dot{G}(0)^{-1}\left\{P_0 \tilde{\pi}_\infty(V)\tilde{V}^{\otimes 2}\right\}^{-1}
         \sum_{j=1}^J\{\nu_jp_j^{-1}(1-p_j)\}^{1/2}\mathbb{G}_j\tilde{V},
\end{eqnarray*}
where $\mathbb{G}_j$ are independent $P_{0|j}$-Brownian bridge processes.
Then $|\hat{\hat{\alpha}}_N^{*\#}-\alpha_0|\rightarrow_{P_W^*}0$ with $*\in\{b,bs\}$ and $\#\in\{c,cc\}$, and
\begin{eqnarray*}
&&\sqrt{N}(\hat{\hat{\alpha}}_N^{bc} - \hat{\alpha}_N^{c})
=A_N + o_{P_W^*}(1)
\rightsquigarrow A,\\
&&\sqrt{N}(\hat{\hat{\alpha}}_N^{bcc} - \hat{\alpha}_N^{cc})
= B_N + o_{P_W^*}(1)
\rightsquigarrow B,\\
&&\sqrt{N}(\hat{\hat{\alpha}}_N^{bsc} - \alpha_0)
= A_N + o_{P_W^*}(1)
\rightsquigarrow A,\\
&&\sqrt{N}(\hat{\hat{\alpha}}_N^{bscc} - \alpha_0)
=B_N + o_{P_W^*}(1)
\rightsquigarrow B, \quad \mbox{in $P^*$-probability.}
\end{eqnarray*}
\end{prop}

\begin{proof}
First we consider bootstrap centered calibration with $\hat{\hat{\alpha}}_N=\hat{\hat{\alpha}}_N^{bcc}$ obtained as the solution to the equation (\ref{eqn:survbootccaleqn}).
Let $\hat{\alpha}_N = \hat{\alpha}_N^{cc}$.
Define $\hat{\Phi}_{N,bcc}(\alpha)\equiv \hat{\mathbb{P}}_N^{\pi,(2)} G_{cc}(V;\alpha)(V -\mathbb{P}_N V)$
and $\Phi_{cc}(\alpha)\equiv P_0G_{cc,\infty}(V;\alpha)\tilde{V}$.
Note that $\hat{\Phi}_{N,bcc}(\hat{\hat{\alpha}}_N)=0$ by (\ref{eqn:survbootccaleqn}) and $\Psi_{cc}(0)=0$.
We apply Lemma \ref{lemma:thm5.9}
for a consistency proof.
For the first condition of the lemma, we have
\begin{eqnarray*}
&&\sup_{\alpha\in\mathbb{R}^k}\left| \hat{\Phi}_{N,bcc}(\alpha)-\Phi_{cc}(\alpha)\right| \nonumber \\
&& \leq  \sup_{\alpha\in\mathbb{R}^k}\left|(\hat{\mathbb{P}}_N^{\pi,(2)}-P_0)G_{cc}(V;\alpha)V \right|
+\sup_{\alpha\in\mathbb{R}^k}\left|(\hat{\mathbb{P}}_N^{\pi,(2)}-P_0)G_{cc}(V;\alpha) \right|  | \mathbb{P}_NV|\\ &&\quad
+\sup_{\alpha\in\mathbb{R}^k}\left| P_0\{G_{cc}(V;\alpha)(V-\mathbb{P}_NV) -G_{cc,\infty}(V;\alpha)\tilde{V}\}\right|.
\end{eqnarray*}
The first two terms are $o_{P_W^*}(1)$ in $P^*$-probability.
To see this, note that the set $\mathcal{G}_1\equiv \{\sum_{j=1}^JI_{\mathcal{V}_j}(v)(c_{j,1}v^T\alpha-c_{j,2}):c_{j,1},c_{j,2}\in\mathbb{R},\alpha \in\mathbb{R}^k\}$ is a VC subgraph class by Lemma 2.6.15 of \cite{MR1385671} since $\mathcal{G}_1$ is a finite-dimensional vector space of functions of $v$.
Thus, $\mathcal{G}_2\equiv \{G(g):g\in\mathcal{G}_1\}$ is also a VC subgraph class by Lemma 2.6.18 of \cite{MR1385671} because of the monotonicity of $G$ (Conditions \ref{cond:cal} and \ref{cond:bootcal}).
Because $\mathcal{G}_2$ has an integrable envelope because of the boundedness of $G$ (Conditions \ref{cond:cal} and \ref{cond:bootcal}), it is $P_0$-Glivenko-Cantelli.
Thus, $\mathcal{G}_3\equiv \{G(\sum_{j=1}^JI_{\mathcal{V}_j}(v)\{(1-p_j)/p_j\}(v-\mu)^T\alpha):p_j\in [\sigma,1],\mu,\alpha \in\mathbb{R}^k\}\subset \mathcal{G}_2$ implies that the set $\mathcal{G}_3$ is $P_0$-Glivenko-Cantelli.
Because the multiplication $(x,y)\mapsto xy$ is continuous and that $\mathcal{G}_4\equiv \{g(v)v: g\in\mathcal{G}_3\}$ has an integrable envelope, $\mathcal{G}_4$ is $P_0$-Glivenko-Cantelli by the Glivenko-Cantelli preservation theorem of \cite{MR1857319}.
It follows from the Bootstrap Glivenko-Cantelli theorem for two-phase sampling (Theorem \ref{thm:fpsPiGCboot}) that  $\sup_{\alpha\in\mathbb{R}^k}|(\hat{\mathbb{P}}_N^{\pi,(2)}-P_0)G_{cc}(V;\alpha)V|=o_{P_W^*}(1)$ and $\sup_{\alpha\in\mathbb{R}^k}|(\hat{\mathbb{P}}_N^{\pi,(2)}-P_0)G_{cc}(V;\alpha)|=o_{P_W^*}(1)$ in $P^*$-probability.
Since $\mathbb{P}_NV =O_{P_W^*}(1)$ in $P^*$-probability by the weak law of large numbers and Lemma \ref{lemma:bootorder}, the second term is $o_{P_W^*}(1)$ in $P^*$-probability by Lemma \ref{lemma:bootorder} (3).
For the third term, suppose to the contrary that this term does not converges to zero.
Note that this term is bounded because $G$ is bounded and $V$ is square-integrable (Conditions \ref{cond:cal} and \ref{cond:bootcal}).
Thus, there exists a subsequence $N'$ of $N$ and a sequence $\{\alpha^{(m)}\}\in\mathbb{R}^k$ such that
\begin{equation*}
\label{eqn:contradictionAlpha}
\left| P_0\{G_{cc}(V;\alpha^{(N')})(V-\mathbb{P}_{N'}V) -G_{cc,\infty}(V;\alpha^{(N')})\tilde{V}\}\right|
\rightarrow c>0
\end{equation*}
for some $c$ and $\alpha^{(N')}$ converges to some vector $\alpha^{(\infty)}\in\overline{\mathbb{R}}^k$ whose elements are extended real numbers.
Because $\mathbb{P}_NV\rightarrow P_0V$, $P^*$-almost surely by the strong law of large numbers, $\pi_0(v)\rightarrow \pi_\infty(v)$ by assumption, and $G$ is continuous and bounded, we have
\begin{equation*}
|G_{cc}(v;\alpha^{(N^{''})})(V-\mathbb{P}_{N^{''}}V)
- G_{cc,\infty}(v;\alpha^{(N^{''})})(v-P_0V)|\rightarrow 0.
\end{equation*}
Noting the boundedness of $G$ and the square integrability of $V$, the dominated convergence theorem yields
\begin{equation*}
\left| P_0\{G_{cc}(V;\alpha^{(N^{''})})(V-\mathbb{P}_{N^{''}}V) -G_{cc,\infty}(V;\alpha^{(N^{''})})\tilde{V}\}\right|
\rightarrow 0
\end{equation*}
as $N^{''}\rightarrow \infty$, which is a contradiction to (\ref{eqn:contradictionAlpha}).
This establishes the first condition of Lemma \ref{lemma:thm5.9}.
The second condition of Lemma \ref{lemma:thm5.9}
was verified in the proof of Proposition A.1 of \cite{SW2013supp}, and hence $\hat{\hat{\alpha}}_N\rightarrow_{P^*_W}\alpha_0$ in $P^*$-probability.

We apply Lemma \ref{lemma:thm3.3.1}
to show the asymptotic normality of $\hat{\hat{\alpha}}_N$.
For the asymptotic equicontinuity condition, Taylor's theorem yields
\begin{eqnarray*}
&&\sqrt{N}(\hat{\Phi}_{N,bcc}-\Phi_{cc})(\hat{\hat{\alpha}}_N)-\sqrt{N}(\hat{\Phi}_{N,bcc}-\Phi_{cc})(\alpha_0)\\
&&=\sqrt{N}(\hat{\mathbb{P}}_N^{\pi,(2)}-P_0)(G_{cc}(V;\hat{\hat{\alpha}}_N)-1)(V-\mathbb{P}_NV)\\
&&\quad +\sqrt{N}P_0(G_{cc}(V;\hat{\hat{\alpha}}_N)-1)(V-\mathbb{P}_NV)\\
&&\quad -\sqrt{N}P_0\{G_{cc,\infty}(V;\hat{\hat{\alpha}}_N)-1\}(V-P_0V)\\
&&=(\hat{\mathbb{P}}_N^{\pi,(2)}-P_0)\dot{G}_{cc}(V;\tilde{\alpha})\tilde{\pi}_0(V)(V-\mathbb{P}_NV)^{\otimes 2}\sqrt{N}(\hat{\hat{\alpha}}_N-\alpha_0)\\
&&\quad +P_0\dot{G}_{cc}(V;\tilde{\alpha})\tilde{\pi}_0(V)(V-\mathbb{P}_NV)^{\otimes 2}\sqrt{N}(\hat{\hat{\alpha}}_N-\alpha_0),\\
&&\quad -P_0\dot{G}_{cc,\infty}(V;\tilde{\alpha})\tilde{\pi}_0(V)(V-\mathbb{P}_NV)(V-P_0V)^T\sqrt{N}(\hat{\hat{\alpha}}_N-\alpha_0)\\
&&\equiv (J_1+J_2+J_3)\sqrt{N}(\hat{\hat{\alpha}}_N-\alpha_0),
\end{eqnarray*}
where $\tilde{\alpha}$ is some convex combination of $\hat{\hat{\alpha}}_N$ and $\alpha_0$.
Note that $\dot{G}$ is bounded (Conditions \ref{cond:cal} and \ref{cond:bootcal}).
Thus, we can proceed in the same way as in a consistency proof to obtain $J_1 = o_{P_W^*}(1)$  in $P^*$-probability by Theorem \ref{thm:fpsPiGCboot}.
For $J_2$ and $J_3$, note that $\tilde{\alpha}\rightarrow_{P_W^*}\alpha_0$ in $P^*$-probability because of $\hat{\hat{\alpha}}_N\rightarrow_{P_W^*}\alpha_0$ in $P^*$-probability.
Then we have by Lemma \ref{lemma:bootorder} that $\tilde{\alpha}\rightarrow \alpha_0$ in outer $P_{W|\infty}$-almost surely for every subsequence $\{N'\}$ of $\{N\}$.
Note also that $\mathbb{P}_NV\rightarrow P_0V$, $P^\infty$-almost surely, that $\tilde{\pi}_0(v)\rightarrow \tilde{\pi}_\infty(v)$, and that $\dot{G}$ is bounded.
Again, a similar argument based on the dominated convergence theorem used above for a consistency proof yields that $J_2+J_3 = o_{P_W^*}(1)$ in $P^*$-probability.
Thus, the last display is $o_{P_W^*}(1+ \sqrt{N}| \hat{\hat{\alpha}}_N-\alpha_0|)$ in $P^*$-probability.
Next, we show weak convergence of the process $\sqrt{N}(\hat{\Phi}_{N,bcc}-\Phi_{cc})(\alpha)$ at $\alpha_0=0$.
Because $\hat{\mathbb{P}}_N^{\pi,(2)}c=c, \mathbb{P}_N^{\pi}c=c, P_0c=c$ for any constant $c$, we have by Lemma \ref{lemma:ph2D}, Theorem 5.3 of \cite{MR3059418}, and Lemma \ref{lemma:bootorder} that
\begin{eqnarray*}
\sqrt{N}(\hat{\Phi}_{N,bcc}-\Phi_{cc})(\alpha_0)
&=&\sqrt{N}\hat{\mathbb{P}}_N^{\pi,(2)}(V-\mathbb{P}_NV)
=(\tilde{\mathbb{G}}_N^{\pi,(2)} + \mathbb{G}_N^{\pi,(2)}) \tilde{V}\\
&=& O_{P_W^*}(1),\quad \mbox{in $P^*$-probability.}
\end{eqnarray*}
Hence, it follows by Lemma \ref{lemma:thm3.3.1}
, Lemma \ref{lemma:ph2D} and Proposition A.1 of \cite{SW2013supp} that
\begin{eqnarray*}
&&\sqrt{N}(\hat{\hat{\alpha}}_N-\alpha_0)\\
&&= -\dot{\Phi}_{cc}(0)\sqrt{N}(\Phi_{N,bcc}-\Phi_{cc})(0)+o_{P_W^*}(1)\\
&&= -\dot{G}(0)^{-1}\left\{P_0 \tilde{\pi}_\infty(V)\tilde{V}^{\otimes 2}\right\}^{-1}\tilde{\mathbb{G}}^{\pi,(2)}_N\tilde{V}\\ &&\quad
 -\dot{G}(0)^{-1}\left\{P_0 \tilde{\pi}_\infty(V)\tilde{V}^{\otimes 2}\right\}^{-1}\mathbb{G}_N^{\pi,(2)} V+o_{P_W^*}(1)\\
&&= -\dot{G}(0)^{-1}\left\{P_0 \tilde{\pi}_\infty(V)\tilde{V}^{\otimes 2}\right\}^{-1}\mathbb{G}^{\pi,(2)}\tilde{V} +\sqrt{N}(\hat{\alpha}_N-\alpha_0)+o_{P_W^*}(1)
\end{eqnarray*}
in $P^*$-probability.
Rearrangement of terms yields the desired result.

We consider bootstrap calibration with $\hat{\hat{\alpha}}_N=\hat{\hat{\alpha}}_N^{bc}$ obtained as the solution to the equation (\ref{eqn:survbootcaleqn}).
Define $\hat{\Phi}_{N,bc}(\alpha)\equiv \hat{\mathbb{P}}_N^{\pi,(2)} G_{c}(V;\alpha)V -\mathbb{P}_N V$
and $\Phi_{c}(\alpha)\equiv P_0(G_{c,\infty}(V;\alpha)-1)V$.
Note that $\hat{\Phi}_{N,bc}(\hat{\hat{\alpha}}_N)=0$ by (\ref{eqn:survbootcaleqn}) and $\Psi_{c}(0)=0$.
We apply Lemma \ref{lemma:thm5.9}
for a consistency proof.
For the first condition of the lemma, we have
\begin{eqnarray*}
\sup_{\alpha\in\mathbb{R}^k}\left| \hat{\Phi}_{N,bc}(\alpha)-\Phi_{c}(\alpha)\right|
& \leq&  \sup_{\alpha\in\mathbb{R}^k}\left|(\hat{\mathbb{P}}_N^{\pi,(2)}-P_0)G_{c}(V;\alpha)V \right|\\ &&
+\sup_{\alpha\in\mathbb{R}^k}\left|(\mathbb{P}_N-P_0)V \right|\\ &&
+\sup_{\alpha\in\mathbb{R}^k}\left| P_0\{G_{c}(V;\alpha)V -G_{c,\infty}(V;\alpha)V\}\right|.
\end{eqnarray*}
The first and third terms in the last display are $o_{P_W^*}(1)$ in $P^*$-probability by a similar argument in the consistency proof for $\hat{\alpha}_N^{bcc}$.
The second term is also $o_{P_W^*}(1)$ in $P^*$-probability by the law of large numbers and Lemma \ref{lemma:bootorder}.
This verifies the first condition.
The second condition was verified in the proof of Proposition A.1 of \cite{SW2013supp}.
This proves the consistency $\hat{\hat{\alpha}}_N\rightarrow_{P_W^*}\alpha_0$ in $P^*$-probability.

We apply Lemma \ref{lemma:thm3.3.1}
to show the asymptotic normality of $\hat{\hat{\alpha}}_N$.
For the asymptotic equicontinuity condition, Taylor's theorem yields
\begin{eqnarray*}
&&\sqrt{N}(\hat{\Phi}_{N,bc}-\Phi_{c})(\hat{\hat{\alpha}}_N)-\sqrt{N}(\hat{\Phi}_{N,bc}-\Phi_{c})(\alpha_0)\\
&&=\sqrt{N}(\hat{\mathbb{P}}_N^{\pi,(2)}-P_0)(G_{c}(V;\hat{\hat{\alpha}}_N)-1)V
+\sqrt{N}P_0(G_{c}(V;\hat{\hat{\alpha}}_N)-1)V\\
&&\quad -\sqrt{N}P_0\{G_{c,\infty}(V;\hat{\hat{\alpha}}_N)-1\}V\\
&&=(\hat{\mathbb{P}}_N^{\pi,(2)}-P_0)\dot{G}_{c}(V;\tilde{\alpha})V^{\otimes 2}\sqrt{N}(\hat{\hat{\alpha}}_N-\alpha_0)\\
&&\quad
 +P_0\dot{G}_{c}(V;\tilde{\alpha})V^{\otimes 2}\sqrt{N}(\hat{\hat{\alpha}}_N-\alpha_0) -P_0\dot{G}_{c,\infty}(V;\tilde{\alpha})V^{\otimes 2}\sqrt{N}(\hat{\hat{\alpha}}_N-\alpha_0),
\end{eqnarray*}
where $\tilde{\alpha}$ is some convex combination of $\hat{\hat{\alpha}}_N$ and $\alpha_0$.
Proceeding in the same way as in a proof for the asymptotic equicontinuity regarding $\hat{\hat{\alpha}}_N^{bcc}$, the last display is shown to be $o_{P_W^*}(1+ \sqrt{N}| \hat{\hat{\alpha}}_N-\alpha_0|)$ in $P^*$-probability.
Next, we show weak convergence of the process $\sqrt{N}(\hat{\Phi}_{N,bc}-\Phi_{c})(\alpha)$ at $\alpha_0=0$.
As in the case for $\hat{\hat{\alpha}}_N^{bcc}$ we have
\begin{eqnarray*}
&&\sqrt{N}(\hat{\Phi}_{N,bc}-\Phi_{c})(\alpha_0)
=\sqrt{N}\hat{\mathbb{P}}_N^{\pi,(2)}V-\mathbb{P}_NV
=(\tilde{\mathbb{G}}_N^{\pi,(2)}+\mathbb{G}_N^{\pi,(2)}) V
\end{eqnarray*}
Hence, it follows by Lemma \ref{lemma:thm3.3.1},
Lemma \ref{lemma:ph2D} and Proposition A.1 of \cite{MR3059418} that
\begin{eqnarray*}
&&\sqrt{N}(\hat{\hat{\alpha}}_N-\alpha_0)
= -\dot{G}(0)^{-1}\left\{P_0 V^{\otimes 2}\right\}^{-1}\mathbb{G}^{\pi,(2)}V + \sqrt{N}(\hat{\alpha}_N-\alpha_0)+o_{P_W^*}(1)
\end{eqnarray*}
in $P^*$-probability.

We consider bootstrap single centered calibration with $\hat{\hat{\alpha}}_N=\hat{\hat{\alpha}}_N^{bscc}$ obtained as the solution to the equation (\ref{eqn:singlebootccaleqn}).
Define $\hat{\Phi}_{N,bscc}(\alpha)\equiv \hat{\mathbb{P}}_N^{\pi,(2)} G_{cc}(V;\alpha)(V -\mathbb{P}_N^{\pi} V)$
and $\Phi_{cc}(\alpha)\equiv P_0G_{cc,\infty}(V;\alpha)\tilde{V}$.
Note that $\hat{\Phi}_{N,bscc}(\hat{\hat{\alpha}}_N^{bscc})=0$ by (\ref{eqn:singlebootccaleqn}) and $\Psi_{cc}(0)=0$.
We apply Lemma \ref{lemma:thm5.9}
for a consistency proof.
For the first condition of the lemma, we have
\begin{eqnarray*}
&&\sup_{\alpha\in\mathbb{R}^k}\left| \hat{\Phi}_{N,bscc}(\alpha)-\Phi_{cc}(\alpha)\right| \nonumber \\
&& \leq  \sup_{\alpha\in\mathbb{R}^k}\left|(\hat{\mathbb{P}}_N^{\pi,(2)}-P_0)G_{cc}(V;\alpha)V \right|
+\sup_{\alpha\in\mathbb{R}^k}\left|(\hat{\mathbb{P}}_N^{\pi,(2)}-P_0)G_{cc}(V;\alpha) \right|  | \mathbb{P}_N^\pi V|\\ &&\quad
+\sup_{\alpha\in\mathbb{R}^k}\left| P_0\{G_{cc}(V;\alpha)(V-\mathbb{P}_N^\pi V) -G_{cc,\infty}(V;\alpha)\tilde{V}\}\right|.
\end{eqnarray*}
As in the consistency proof for $\hat{\alpha}_N^{bcc}$ all terms in the last display is $o_{P^*_W}(1)$ in $P^*$-probability.
This establishes the first condition of Lemma \ref{lemma:thm5.9}.
The second condition of Lemma \ref{lemma:thm5.9}
was verified in the proof of Proposition A.1 of \cite{SW2013supp}.
Thus, $\hat{\hat{\alpha}}_N\rightarrow_{P^*_W}\alpha_0$ in $P^*$-probability.

We apply Lemma \ref{lemma:thm3.3.1}
to show the asymptotic normality of $\hat{\hat{\alpha}}_N$.
For the asymptotic equicontinuity condition, Taylor's theorem yields
\begin{eqnarray*}
&&\sqrt{N}(\hat{\Phi}_{N,bscc}-\Phi_{cc})(\hat{\hat{\alpha}}_N)-\sqrt{N}(\hat{\Phi}_{N,bscc}-\Phi_{cc})(\alpha_0)\\
&&=\sqrt{N}(\hat{\mathbb{P}}_N^{\pi,(2)}-P_0)(G_{cc}(V;\hat{\hat{\alpha}}_N)-1)(V-\mathbb{P}_N^\pi V)\\
&&\quad +\sqrt{N}P_0(G_{cc}(V;\hat{\hat{\alpha}}_N)-1)(V-\mathbb{P}_N^\pi V)\\
&&\quad -\sqrt{N}P_0\{G_{cc,\infty}(V;\hat{\hat{\alpha}}_N)-1\}(V-P_0V)\\
&&=(\hat{\mathbb{P}}_N^{\pi,(2)}-P_0)\dot{G}_{cc}(V;\tilde{\alpha})\tilde{\pi}_0(V)(V-\mathbb{P}_N V)(V-\mathbb{P}_N^\pi V)^T\sqrt{N}(\hat{\hat{\alpha}}_N-\alpha_0)\\
&&\quad +P_0\dot{G}_{cc}(V;\tilde{\alpha})\tilde{\pi}_0(V)(V-\mathbb{P}_N^\pi V)(V-\mathbb{P}_N V)^T\sqrt{N}(\hat{\hat{\alpha}}_N-\alpha_0),\\
&&\quad -P_0\dot{G}_{cc,\infty}(V;\tilde{\alpha})\tilde{\pi}_0(V)(V-\mathbb{P}_NV)(V-P_0V)^T\sqrt{N}(\hat{\hat{\alpha}}_N-\alpha_0),
\end{eqnarray*}
where $\tilde{\alpha}$ is some convex combination of $\hat{\hat{\alpha}}_N$ and $\alpha_0$.
Note that $\mathbb{P}_{N'}^\pi V\rightarrow P_0V$, outer $P_{W|\infty}$-almost surely for every subsequence $\{N'\}$ of $\{N\}$ by Theorem 5.1 of \cite{MR3059418} and Lemma \ref{lemma:bootorder}.
Proceeding in the same way as in a proof for the asymptotic equicontinuity regarding $\hat{\hat{\alpha}}_N^{bcc}$, the last display is shown to be $o_{P_W^*}(1+ \sqrt{N}| \hat{\hat{\alpha}}_N-\alpha_0|)$ in $P^*$-probability.
Next, we show weak convergence of the process $\sqrt{N}(\hat{\Phi}_{N,bscc}-\Phi_{cc})(\alpha)$ at $\alpha_0=0$.
It follows from Lemma \ref{lemma:ph2D} that
\begin{eqnarray*}
&&\sqrt{N}(\hat{\Phi}_{N,bscc}-\Phi_{cc})(\alpha_0)
=\sqrt{N}(\hat{\mathbb{P}}_N^{\pi,(2)}-\mathbb{P}_N^\pi)\tilde{V}
\rightsquigarrow \mathbb{G}^{\pi,(2)}\tilde{V} ,
\end{eqnarray*}
in $P^*$-probability.
Here we used the fact that $\hat{\mathbb{P}}_N^{\pi,(2)}c=c,\mathbb{P}_N^{\pi}c=c,$ for any constant $c$.
Thus, by Lemma \ref{lemma:thm3.3.1}
we obtain
\begin{eqnarray*}
\sqrt{N}(\hat{\hat{\alpha}}_N-\alpha_0)
&=& -\dot{\Phi}_{c}(0)\sqrt{N}(\hat{\Phi}_{N,bscc}-\Phi_{cc})(0)+o_{P^*_W}(1)\\
&\rightsquigarrow& -\dot{G}(0)^{-1}\left\{P_0 \tilde{\pi}_\infty(V)\tilde{V}^{\otimes 2}\right\}^{-1}\mathbb{G}^{\pi,(2)}\tilde{V}
\end{eqnarray*}
in $P^*$-probability.

We consider bootstrap single calibration with $\hat{\hat{\alpha}}_N=\hat{\hat{\alpha}}_N^{bsc}$ obtained as the solution to the equation (\ref{eqn:singlebootcaleqn}).
Define $\hat{\Phi}_{N,bsc}(\alpha)\equiv \hat{\mathbb{P}}_N^{\pi,(2)} G_{c}(V;\alpha)V -\mathbb{P}_N^\pi V$
and $\Phi_{c}(\alpha)\equiv P_0(G_{c,\infty}(V;\alpha)-1)V$.
Note that $\hat{\Phi}_{N,bsc}(\hat{\hat{\alpha}}_N)=0$ by (\ref{eqn:singlebootcaleqn}) and $\Psi_{c}(0)=0$.
We apply Lemma \ref{lemma:thm5.9}
for a consistency proof.
For the first condition of the lemma, we have
\begin{eqnarray*}
\sup_{\alpha\in\mathbb{R}^k}\left| \hat{\Phi}_{N,bsc}(\alpha)-\Phi_{c}(\alpha)\right|
& \leq&  \sup_{\alpha\in\mathbb{R}^k}\left|(\hat{\mathbb{P}}_N^{\pi,(2)}-P_0)G_{c}(V;\alpha)V \right|\\ &&
+\sup_{\alpha\in\mathbb{R}^k}\left|(\mathbb{P}_N^\pi-P_0)V \right|\\ &&
+\sup_{\alpha\in\mathbb{R}^k}\left| P_0\{G_{c}(V;\alpha)V -G_{c,\infty}(V;\alpha)V\}\right|.
\end{eqnarray*}
The first and third terms in the last display are $o_{P_W^*}(1)$ in $P^*$-probability by a similar argument in the consistency proof for $\hat{\alpha}_N^{bcc}$.
The second term is also $o_{P_W^*}(1)$ in $P^*$-probability by Theorem 5.1 of \cite{MR3059418} and Lemma \ref{lemma:bootorder}.
This verifies the first condition.
The second condition was verified in the proof of Proposition A.1 of \cite{SW2013supp}.
This proves the consistency $\hat{\hat{\alpha}}_N\rightarrow_{P_W^*}\alpha_0$ in $P^*$-probability.

We apply Lemma \ref{lemma:thm3.3.1}
to show the asymptotic normality of $\hat{\hat{\alpha}}_N$.
For the asymptotic equicontinuity condition, Taylor's theorem yields
\begin{eqnarray*}
&&\sqrt{N}(\hat{\Phi}_{N,bsc}-\Phi_{c})(\hat{\hat{\alpha}}_N)-\sqrt{N}(\hat{\Phi}_{N,bsc}-\Phi_{c})(\alpha_0)\\
&&=\sqrt{N}(\hat{\mathbb{P}}_N^{\pi,(2)}-P_0)(G_{c}(V;\hat{\hat{\alpha}}_N)-1)V
+\sqrt{N}P_0(G_{c}(V;\hat{\hat{\alpha}}_N)-1)V\\
&&\quad -\sqrt{N}P_0\{G_{c,\infty}(V;\hat{\hat{\alpha}}_N)-1\}V\\
&&=(\hat{\mathbb{P}}_N^{\pi,(2)}-P_0)\dot{G}_{c}(V;\tilde{\alpha})V^{\otimes 2}\sqrt{N}(\hat{\hat{\alpha}}_N-\alpha_0)\\
&&\quad
 +P_0\dot{G}_{c}(V;\tilde{\alpha})V^{\otimes 2}\sqrt{N}(\hat{\hat{\alpha}}_N-\alpha_0) -P_0\dot{G}_{c,\infty}(V;\tilde{\alpha})V^{\otimes 2}\sqrt{N}(\hat{\hat{\alpha}}_N-\alpha_0),
\end{eqnarray*}
where $\tilde{\alpha}$ is some convex combination of $\hat{\hat{\alpha}}_N$ and $\alpha_0$.
Proceeding in the same way as in a proof for the asymptotic equicontinuity regarding $\hat{\hat{\alpha}}_N^{bcc}$, the last display is shown to be $o_{P_W^*}(1+ \sqrt{N}| \hat{\hat{\alpha}}_N-\alpha_0|)$ in $P^*$-probability.
Next, we show weak convergence of the process $\sqrt{N}(\hat{\Phi}_{N,bsc}-\Phi_{c})(\alpha)$ at $\alpha_0=0$.
As in the case for $\hat{\hat{\alpha}}_N^{bscc}$ we have
\begin{eqnarray*}
&&\sqrt{N}(\hat{\Phi}_{N,bsc}-\Phi_{c})(\alpha_0)
=\sqrt{N}(\hat{\mathbb{P}}_N^{\pi,(2)}-\mathbb{P}_N^\pi)V
\rightsquigarrow \mathbb{G}^{\pi,(2)}V,
\end{eqnarray*}
in $P^*$-probability.
Hence, it follows by Lemma \ref{lemma:thm3.3.1}
that
\begin{eqnarray*}
\sqrt{N}(\hat{\hat{\alpha}}_N-\alpha_0)
&=& -\dot{\Phi}_{c}(0)\sqrt{N}(\Phi_{N,bsc}-\Phi_{c})(0)+o_{P_W^*}(1)\\
&\rightsquigarrow&-\dot{G}(0)^{-1}\left\{P_0 V^{\otimes 2}\right\}^{-1}\mathbb{G}^{\pi,(2)} V
\end{eqnarray*}
in $P^*$-probability.
\end{proof}

The following is the bootstrap version of Theorem 5.9 of \cite{MR1652247}.
\begin{lemma}
\label{lemma:thm5.9}
Let $\Theta$ be a parameter space with semimetric $d$.
Let $\Psi_n(\theta)$ be random vector valued functions on $\Theta$ and let $\Psi$ be a fixed vector-valued function on $\Theta$ such that for every $\epsilon >0$
\begin{eqnarray*}
&&\sup_{\theta\in\Theta}\left|\Psi_n(\theta) -\Psi(\theta)\right|  = o_{P_W^*}(1), \quad \mbox{in $P^*$-probability},\\
&&\inf_{\theta:d(\theta,\theta_0)\geq \epsilon}\left| \Psi(\theta)\right| > 0=\left|\Psi(\theta_0)\right|.
\end{eqnarray*}
Then any sequence of estimators $\hat{\theta}_n$ such that $\Psi_n(\hat{\theta}_n) =o_{P_W^*}(1)$ in $P^*$-probability is consistent for $\theta_0$ in $P^*$-probability.
\end{lemma}
\begin{proof}
A proof is essentially the same as that of Theorem 5.9 of \cite{MR1652247} and omitted.
\end{proof}

The following is the bootstrap version of Theorem 3.3.1 of \cite{MR1385671}.
\begin{lemma}
\label{lemma:thm3.3.1}
Let $\Theta$ be a parameter space with semimetric $d$.
Let $\Psi_n$ and $\Psi$ be random and fixed maps, respectively, from $\Theta$ to a Banach space such that
\begin{equation*}
\sqrt{n}(\Psi_n-\Psi)(\hat{\theta}_n) - \sqrt{n}(\Psi_n-\Psi)(\theta_0) = o_{P^*_W}(1+\sqrt{n}d(\hat{\theta}_n,\theta_0))
\end{equation*}
in $P^*$-probability and such that $\sqrt{n}(\Psi_n-\Psi)(\theta_0)\rightsquigarrow Z$ in $P^*$-probability where $Z$ is a tight random element.
Let $\theta\mapsto \Psi(\theta)$ be Fr\'{e}chet differentiable at $\theta_0$ with a continuously invertible derivative $\dot{\Psi}_{\theta_0}$.
If $\Psi(\theta_0) = 0$ and $\hat{\theta}_n$ satisfies $\Psi_n(\hat{\theta}_n)=o_{P^*_W}(n^{-1/2})$ in $P^*$-probability and consistent for $\theta_0$, then
\begin{eqnarray*}
&&\sqrt{n}\dot{\Psi}_{\theta_0}(\hat{\theta}_n-\theta_0) = -\sqrt{n}(\Psi_n-\Psi)(\theta_0) + o_{P^*_W}(1),\\
&&\sqrt{n}(\hat{\theta}_n-\theta_0) \rightsquigarrow -\dot{\Psi}_{\theta_0}^{-1}Z, \quad \mbox{in $P^*$-probability}.
\end{eqnarray*}
\end{lemma}
\begin{proof}
A proof is essentially the same as that of Theorem 3.3.1 of \cite{MR1385671} and omitted.
\end{proof}

\subsubsection{Phase I Bootstrap}
We present several results concerning the uncentered bootstrap empirical process with independent bootstrap weights (see Lemma \ref{lemma:ph1D}).
We first consider the uncentered version of the conditional multiplier central limit theorem.
This requires the (conditional) finite dimensional convergence of $n^{-1/2}\sum_{i=1}^n w_i\delta_{X_i}$ (compare the following lemma with Lemma 2.9.5 of \cite{MR1385671}).
\begin{lemma}
\label{lemma:margbm}
Let $Y_1,Y_2,\ldots,$ be i.i.d. random vectors with $E| Y_i|^2<\infty$ independent of the i.i.d. $w_1,w_2,\ldots,$ with $Ew_i=0$ and $Ew_i^2=c^2>0$.
Then, conditionally on $Y_1,Y_2\ldots,$
\begin{equation*}
\frac{1}{\sqrt{n}}\sum_{i=1}^nw_iY_i\rightsquigarrow N(0,c^2EY_1^{\otimes 2}),
\end{equation*}
for almost every sequence $Y_1,Y_2,\ldots$.
\end{lemma}
\begin{proof}
We apply the Lindeberg central limit theorem.
Note that $\mu_i\equiv E_w w_iY_i =0$ and $\sigma_i^2\equiv E_w w_i^2Y_i^{\otimes 2} -\{E_w w_iY_i\}^{\otimes 2} =c^2Y_i^{\otimes 2}$ where $E_w$ denotes the expectation with respect to $w$.
Thus,
$n^{-1}sd_n^2\equiv n^{-1}\sum_{i=1}^n\sigma^2_i
\rightarrow c^2EY^{\otimes 2}_1$,
for almost all sequences by the strong law of large numbers.
For every $\epsilon >0$,
\begin{eqnarray*}
&&n^{-1}\sum_{i=1}^n| Y_i|^2E_w w_i^2\{|w_i|| Y_i|>\epsilon\sqrt{n}\}\\
&&\leq
n^{-1}\sum_{i=1}^n| Y_i|^2E_w w_1^2\{|w_1|\max_{1\leq j\leq n}| Y_j|>\epsilon\sqrt{n}\}
\rightarrow 0,
\end{eqnarray*}
for almost all sequences, because $E| Y_i|^2<\infty$ implies $\max_{1\leq i\leq n}| Y_i|/\sqrt{n}\rightarrow 0$ for almost all sequences.
This completes the proof.
\end{proof}

The next lemma concerns integrability of the empirical process when the $L_2(P_0)$-metric is used.
This lemma is used to prove the uncentered conditional multiplier central limit theorem (Lemma \ref{lemma:ph1D}).
\begin{lemma}
\label{lemma:bmint}
Let  $\mathcal{F}$ be a $P_0$-Donsker class with $\lVert P_0\rVert_{\mathcal{F}}<\infty$.
Let $X_1,X_2,\ldots$ be i.i.d. $P_0$, independent of  i.i.d. Rademacher variables $\epsilon_1,\epsilon_2,\ldots$.
Define the process $\tilde{\mathbb{G}}_n'=n^{1/2}\sum_{i=1}^n\epsilon_i\delta_{X_i}$.
Let $\rho(f,g)=\{P_0(f-g)^2\}^{1/2}$ and $\mathcal{F}_\delta=\{f-g:\rho(f,g)<\delta,f,g\in\mathcal{F}\}$.
Then $E^*\lVert \tilde{\mathbb{G}}_n'\rVert_{\mathcal{F}_\delta}\rightarrow 0$ for every $\delta_n\downarrow 0$.
\end{lemma}
\begin{proof}
Since $\mathcal{F}$ is Donsker with $\lVert P_0\rVert_{\mathcal{F}}<\infty$, it follows from Corollary 2.9.4 of \cite{MR1385671} that $\tilde{\mathbb{G}}_n'$ weakly converges to the Brownian motion process in $\ell^\infty(\mathcal{F})$ and $\tilde{\mathbb{G}}_n'$ is asymptotically equicontinuous in probability with respect to the $L_2(P_0)$-metric $\rho$.
Moreover, $\mathcal{F}$ possesses an envelope $F$ with $P(F>x)=o(x^{-2})$ by Corollary 2.3.13 of \cite{MR1385671}.
This implies that $P(\lVert \epsilon_1\delta_{X_1}\rVert_{\mathcal{F}}>x)= P(F>x) =o(x^{-2})$.
In view of Problem 2.3.3 of \cite{MR1385671}, this implies
$E^*\max_{1\leq i\leq n}\lVert \epsilon_i\delta_{X_i}\rVert_{\mathcal{F}}/\sqrt{n}\rightarrow 0.$
It follows from the triangle inequality that the same is true with $\mathcal{F}$ replaced by $\mathcal{F}_{\delta_n}$.
Because asymptotic equicontinuity in probability implies $\lVert \tilde{\mathbb{G}}_n'\rVert_{\mathcal{F}_{\delta_n}}\rightarrow_P0$ for every $\delta_n\downarrow 0$, the sequence of quantile functions of $\tilde{\mathbb{G}}_n'$ converges to zero pointwise.
Apply the Hoffmann-J\o rgensen inequality (see A.1.5 of \cite{MR1385671}) to obtain the desired result.
\end{proof}

We prove the uncentered conditional multiplier central limit theorem.
\begin{proof}[Proof of Lemma \ref{lemma:ph1D}]
The sequence $\tilde{\mathbb{G}}_n$ converges to a $c$ times a $P_0$-Brownian motion process $\tilde{\mathbb{G}}$ in $\ell^\infty(\mathcal{F})$ by Corollary 2.9.4 of \cite{MR1385671}, and thus it is asymptotically measurable.

A Donsker class $\mathcal{F}$ is totally bounded for the $L_2(P_0)$ metric since $\lVert P_0\rVert_{\mathcal{F}}<\infty$ (Problem 2.1.1 of \cite{MR1385671}).
For each fixed $\delta>0$ and $f\in\mathcal{F}$, let $\Pi_\delta f$ denote a closest element in a given finite $\delta$-net for $\mathcal{F}$.
By continuity of the limit process $\tilde{\mathbb{G}}$, we have $\tilde{\mathbb{G}}\circ \Pi_\delta\mapsto \tilde{\mathbb{G}}$ almost surely as $\delta\downarrow 0$.
Hence it follows that
\begin{equation*}
\sup_{h\in BL_1}|Eh(c\tilde{\mathbb{G}}\circ\Pi_\delta)-Eh(c\tilde{\mathbb{G}})|
\rightarrow 0, \quad \delta\downarrow 0.
\end{equation*}
Also, it follows from Lemma \ref{lemma:margbm} that for every fixed $\delta>0$
\begin{equation*}
\sup_{h\in BL_1}|E_w h(\tilde{\mathbb{G}}_n\circ \Pi_\delta)-Eh(c\tilde{\mathbb{G}}\circ \Pi_\delta)|\rightarrow 0, \quad n\rightarrow \infty,
\end{equation*}
for almost every sequence $X_1,X_2,\ldots$ as in a proof of Theorem 2.9.6 of \cite{MR1385671} where $E_w$ denotes the expectation with respect to $w$.
Next,
\begin{eqnarray*}
\sup_{h\in BL_1}|E_w h(\tilde{\mathbb{G}}_n\circ \Pi_\delta)- E_wh(\tilde{\mathbb{G}}_n)|
\leq E_w \lVert \tilde{\mathbb{G}}_n\circ\Pi_\delta -\tilde{\mathbb{G}}_n\rVert_{\mathcal{F}}
\leq E_w \lVert \tilde{\mathbb{G}}_n\rVert_{\mathcal{F}_\delta},
\end{eqnarray*}
where $\mathcal{F}_\delta=\{f-g:f,g\in\mathcal{F},P_0(f-g)^2<\delta^2\}$.
Thus, the outer expectation of the left side is bounded above by $E^*\lVert \tilde{\mathbb{G}}_n\rVert_{\mathcal{F}_\delta}$.

Since $\lVert w_1\rVert_{2,1}<\infty$ implies $E w_1^2<\infty$, we have $E\max_{1\leq i\leq n}|w_i|/\sqrt{n}\rightarrow 0$.
Thus, taking a limit on $n$ on both sides of the second part of the multiplier inequality in Lemma 2.9.1 of \cite{MR1385671} yields
\begin{eqnarray*}
\lim_{n\rightarrow \infty}E^*\left\lVert \frac{1}{\sqrt{n}}\sum_{i=1}^nw_i\delta_{X_i}\right\rVert_{\mathcal{F}_\delta}
&\leq& 2\sqrt{2}\lVert w_1\rVert_{2,1}\sup_{n_0\leq k}E^*\left\lVert\frac{1}{\sqrt{k}}\sum_{i=1}^k\epsilon_i\delta_{X_i}\right\rVert_{\mathcal{F}_\delta},
\end{eqnarray*}
for every $n_0$ and $\delta>0$ where $\epsilon_i$ are i.i.d. Rademacher random variables independent of $w_i$ and $X_i$.
The left hand side of the inequality converges to zero as $n_0\rightarrow \infty$ followed by $\delta \downarrow 0$ because $\lim_{k\rightarrow\infty}E\lVert \tilde{\mathbb{G}}_k'\rVert_{\mathcal{F}_\delta}\rightarrow 0$ as $\delta \downarrow 0$ by Lemma \ref{lemma:bmint} where $\tilde{\mathbb{G}}_n'=n^{-1/2}\sum_{i=1}^n\epsilon_i\delta_{X_i}$.
Combining this with the previous display with the triangle inequality yields the first part of the claim.

For the second part of the claim, the proof of the first part applies except that it must be argued that $E_w\lVert \tilde{\mathbb{G}}\rVert_{\mathcal{F}_\delta}^*$ converges to zero outer almost surely as $n\rightarrow \infty$ followed by $\delta \downarrow 0$.
Since $P_0\lVert f-P_0f\rVert_{\mathcal{F}}^2<\infty$ and $\lVert P_0\rVert_{\mathcal{F}}<\infty$ implies
\begin{eqnarray*}
&&P_0\lVert f(X_1)\rVert_{\mathcal{F}}^2
\leq P_0\lVert f(X_1)-P_0f+P_0f\rVert_{\mathcal{F}}^2\\
&&\leq P_0\{\lVert f(X_1)-P_0f\rVert_{\mathcal{F}}^2+\lVert P_0f\rVert_{\mathcal{F}}^2 + 2\lVert f-P_0f\rVert_{\mathcal{F}}\lVert P_0f\rVert_{\mathcal{F}}\}
<\infty,
\end{eqnarray*}
it follows from Corollary 2.9.9 of \cite{MR1385671} that
\begin{eqnarray*}
\limsup_{n\rightarrow \infty}E_w\lVert\tilde{\mathbb{G}}_n\rVert_{\mathcal{F}_\delta}^*
\leq 6\sqrt{2}\limsup_{n\rightarrow \infty}E^*\lVert \tilde{\mathbb{G}}_n\rVert_{\mathcal{F}_\delta},
\end{eqnarray*}
almost surely.
The right-hand side decreases to zero as $\delta\downarrow 0$ as shown above.
To see that the sequence $E_w h(\tilde{\mathbb{G}}_n)$ is strongly asymptotically measurable, obtain first by the same proof, but with a star added, that
\begin{equation*}
|E_w h(\tilde{\mathbb{G}}_n)^*-Eh(c\mathbb{\tilde{G}})|\rightarrow_{as^*}0.
\end{equation*}
The same proof also shows that this is true with a lower star.
Thus, the sequence $E_w h(\mathbb{\tilde{G}}_n)^*-h(\mathbb{\tilde{G}}_n)_*$ converges to zero almost surely.
\end{proof}

Since we condition on $X_1,X_2,\ldots,$ and $W_{N1}^{(1)},W_{N2}^{(1)},\ldots,$ in the proof of Theorem \ref{thm:fpsDboot}, the following lemma allows us to freely apply Lemma \ref{lemma:ph2D}.

\begin{lemma}
\label{lemma:wfD}
Let $\mathcal{F}$ be a $P_0$-Donsker class with $\lVert P_0\rVert_{\mathcal{F}}<\infty$.
Let $X_1,\ldots,X_n$ be i.i.d. $P_0$.
Let $w_1,\ldots,w_n$ be i.i.d. $P_W$ with $Ew_1=1$, $\mathrm{Var}(w_1)=c^2<\infty$ and $\lVert w_1\rVert_{2,1}<\infty$ that are independent of $X_1,\ldots,X_n$.
Then the class of functions $\mathcal{F}_{\mathcal{W}}=\{g:g(x,w)=wf(x),f\in\mathcal{F}\}$ is $P_0\times P_W$-Donsker.
\end{lemma}
\begin{proof}
Define the empirical process $\mathbb{G}_n = n^{-1/2}\sum_{i=1}^n (\delta_{X_i}-P_0)$.
Note that for $g(x,w)=wf(x)\in\mathcal{F}_{\mathcal{W}}$,
\begin{eqnarray*}
\mathbb{G}_ng
&=& n^{-1/2}\sum_{i=1}^n(\delta_{X_i}-P_0)f
+ \ n^{-1/2}\sum_{i=1}^n(w_i-1)(\delta_{X_i}-P_0)f\\
&&+ \ n^{-1/2}\sum_{i=1}^n(w_i-1)P_0f.
\end{eqnarray*}
Thus in view of Corollary 2.9.4 of \cite{MR1385671},
\begin{equation*}
\mathbb{G}_n \rightsquigarrow \mathbb{G} + c\mathbb{G}' + cZ_0P_0,
\quad \mbox{in }\ell^\infty(\mathcal{F}_{\mathcal{W}}),
\end{equation*}
where $\mathbb{G}$ and $\mathbb{G}'$ are independent Brownian bridge processes that are independent of the standard normal random variable $Z_0$.
\end{proof}

Several results (Lemmas \ref{lemma:margbm}-\ref{lemma:wfD}) regarding the uncentered conditional multiplier central limit theorem provide useful tools to study the phase I bootstrap IPW empirical process.
We first prove a Glivenko-Cantelli theorem for the phase I bootstrap IPW empirical process.
\begin{lemma}
\label{lemma:phIbootIPWGC}
Let $\mathcal{F}$ be a $P_0$-Glivenko-Cantelli class with $\lVert P_0\rVert_{\mathcal{F}}<\infty$.
Then
\begin{eqnarray*}
&&\lVert \hat{\mathbb{P}}_N^{\pi,(1)}-\mathbb{P}_N^\pi\rVert_{\mathcal{F}} \rightarrow_{P_W^*} 0 ,\\
&&\lVert \hat{\mathbb{P}}_N^{\pi,(1),b\#}-\mathbb{P}_N^\pi\rVert_{\mathcal{F}} \rightarrow_{P_W^*} 0,\quad \mbox{in $P^*$-probability}.
\end{eqnarray*}
where $\# \in \{c,cc\}$.

The same holds if we replace $\mathbb{P}_N^\pi$ by $\mathbb{P}_N^{\pi,\#}$ or $P_0$.
\end{lemma}
\begin{proof}
Note that $X_{j,i}$, $i=1,\ldots,N_j$, with $\xi_{j,i}=1$ are i.i.d. $P_{0|j}$.
Thus, conditionally on $\xi$ we can view the sample in the $j$th stratum as the i.i.d. sample of size $n_j$.
In the following we proceed conditionally on $\xi$ and then take expectation with respect to $\xi$ for bootstrap order notations.
Since conditional probabilities given $\xi$ is bounded, the unconditional order notations follow by Vitali's theorem.
Hence we do not explicitly discuss the step from conditional to unconditional arguments.

Note that the decomposition of the phase I bootstrap IPW empirical process is given by
$\hat{\mathbb{P}}_N^{\pi,(1)}
=\sum_{j=1}^J(N_j/N)(\mathbb{P}^{\xi,(1)}_{j,n_j}-\mathbb{P}_{j,n_j}^\xi).$
The triangle inequality yields
\begin{eqnarray*}
\left\lVert\hat{\mathbb{P}}_N^{\pi,(1)}-\mathbb{P}^\pi_N\right\rVert_{\mathcal{F}}
\leq \sum_{j=1}^J\left\lVert\hat{\mathbb{P}}_{j,n_j}^{\xi,(1)}-\mathbb{P}^\xi_{j,n_j}\right\rVert_{\mathcal{F}}.
\end{eqnarray*}
Fix $j$.
Let $\overline{W}_j^{(1)}=n_j^{-1}\sum_{i=1}^{N_j}W_{N_j,j,i}^{(1)}\xi_{j,i}$.
We have
\begin{eqnarray}
\label{eqn:ph1GCdecomp}
\hat{\mathbb{P}}_{j,n_j}^{\xi,(1)}-\hat{\mathbb{P}}_{j,n_j}^{\xi}
&=&\overline{W}_j^{(1)}\left(\frac{1}{n_j}\sum_{i=1}^{N_j}\frac{W_{N_j,j,i}^{(1)}}{\overline{W}_j^{(1)}}\xi_{j,i}\delta_{X_{j,i}}- \frac{1}{n_j}\sum_{i=1}^{N_j}\xi_{j,i}\delta_{X_{j,i}}\right)\nonumber \\
&& + (\overline{W}_j^{(1)}-1)\frac{1}{n_j}\sum_{i=1}^{N_j}\xi_{j,i}(\delta_{X_{j,i}}-P_{0|j}) + (\overline{W}_j^{(1)}-1)P_{0|j}
\end{eqnarray}
Note that  $\sum_{i=1}^{N_j}W_{N_j,j,i}\xi_{j,i}/\overline{W}_j^{(1)}=n_j$ and that $\max_{1\leq i\leq n_j}\xi_{j,i}W_{N_j,j,i}^{(1)}/n_j\rightarrow_{P_W^*} 0$ since $E|W_{N_j,j,i}^{(1)}|<\infty$ for all $i$.
Thus, we can apply Theorem 3.3 of \cite{MR1836584} to obtain
\begin{eqnarray*}
\left\lVert\frac{1}{n_j}\sum_{i=1}^{N_j}\frac{W_{N_j,j,i}^{(1)}}{\overline{W}_j^{(1)}}\xi_{j,i}\delta_{X_{j,i}}- \frac{1}{n_j}\sum_{i=1}^{N_j}\xi_{j,i}\delta_{X_{j,i}}\right\rVert_{\mathcal{F}}
=o_{P_W^*}(1),\ \mbox{in $P^*$-probability.}
\end{eqnarray*}
Since $\overline{W}_j^{(1)}=O_{P_W^*}(1)$, the first term in (\ref{eqn:ph1GCdecomp}) is $o_{P_W^*}(1)$ in $P^*$-probability.
For the second term, the Glivenko-Cantelli theorem yields that $\lVert n_j^{-1}\sum_{i=1}^{N_j}\xi_{j,i}\delta_{X_{j,i}}-P_{0|j}\rVert_{\mathcal{F}} \rightarrow 0$, $P^\infty$-almost surely.
Since $\overline{W}_j^{(1)}\rightarrow_{P_W^*}1$ by the law of large numbers, the second term in (\ref{eqn:ph1GCdecomp}) is $o_{P_W^*}(1)$ in $P^*$-probability.
So is the third term in (\ref{eqn:ph1GCdecomp})
because $\lVert P_{0|j}\rVert_{\mathcal{F}}<\infty$.
To see this, notice that Jensen's inequality yields
\begin{eqnarray*}
&&\lVert P_{0|j}\rVert_{\mathcal{F}}
\leq \nu_j^{-1}E^*\lVert \delta_XI_{\mathcal{V}_j}(V)-\nu_j P_{0|j} \rVert_{\mathcal{F}}+\nu_j^{-1}E^*\lVert\delta_XI_{\mathcal{V}_j}(V)\rVert_{\mathcal{F}}\\
&&\leq \nu_j^{-1} E^*\left\lVert  \sum_{j=1}^J(\delta_XI_{\mathcal{V}_j}(V)-\nu_j P_{0|j})\right\rVert_{\mathcal{F}} + \nu_j^{-1}E^*\lVert\delta_X\rVert_{\mathcal{F}}\\
&&= \nu_j^{-1}E^*\lVert f-P_0f\rVert_{\mathcal{F}}+\nu_j^{-1}E^*\lVert\delta_X\rVert_{\mathcal{F}}.\end{eqnarray*}
Because $\mathcal{F}$ is $P_0$-Glivenko-Cantelli and $\lVert P_0\rVert_{\mathcal{F}}<\infty$, it follows from the result of Problem 2.4.1 of \cite{MR1385671} that both terms in the last display are bounded.

We consider $\lVert \hat{\mathbb{P}}_N^{\pi,(1),bc}-\mathbb{P}_N^{\pi,c}\rVert_{\mathcal{F}}$.
Note that  $\tilde{\mathcal{F}}\equiv\{G_{c}(\cdot;\alpha)f:f\in\mathcal{F},\alpha\in\mathbb{R}^k\}$
is $P_0$-Glivenko-Cantelli by the Glivenko-Cantelli preservation theorem (Theorem 3, \cite{MR1857319}).
We have
\begin{eqnarray*}
&&\left\lVert \hat{\mathbb{P}}_N^{\pi,(1),bc}-\mathbb{P}_N^{\pi,c}\right\rVert_{\mathcal{F}}
=\left\lVert (\hat{\mathbb{P}}_N^{\pi,(1)}-\mathbb{P}_N^{\pi})G_c(V;\hat{\alpha}_N^c)f\right\rVert_{\mathcal{F}}
\leq  \left\lVert \hat{\mathbb{P}}_N^{\pi,(1)}-\mathbb{P}_N^\pi\right\rVert_{\tilde{\mathcal{F}}}.
\end{eqnarray*}
The last term is $o_{P^*_W}(1)$ in $P^*$-probability by the result we just established above.
Because $\lVert \mathbb{P}_N^\pi - P_0\rVert_{\mathcal{F}}=o_{P_W^*}(1)$ and $\lVert \mathbb{P}_N^{\pi,\#} - P_0\rVert_{\mathcal{F}}=o_{P_W^*}(1)$ in $P^*$-probability by Theorem 5.1 of \cite{MR3059418} and Lemma \ref{lemma:bootorder}, the triangle inequality yields the desired results when replacing $\mathbb{P}_N^{\pi,c}$ by $\mathbb{P}_N^{\pi}$, $\mathbb{P}_N^{\pi,\#}$ or $P_0$.
The other case is similar.
This completes the proof.
\end{proof}

Next, we prove conditional weak convergence of the phase I bootstrap IPW empirical process.
\begin{lemma}
\label{lemma:phIbootIPWD}
Let $\mathcal{F}$ be a Donsker class with $\lVert P_0\rVert_{\mathcal{F}}<\infty$.
Then,
\begin{eqnarray*}
\tilde{\mathbb{G}}_N^{\pi,(1)}
&\rightsquigarrow&
\sum_{j=1}^J\sqrt{\frac{\nu_j}{2-p_j}}\tilde{\mathbb{G}}_j^{(1)}
,\quad \mbox{in }\ell^\infty(\mathcal{F}),
\end{eqnarray*}
where the $P_{0|j}$-Brownian motion processes $\tilde{\mathbb{G}}_j^{(1)}$ are all independent.

The same holds when $\tilde{\mathbb{G}}_N^{\pi,(1)}$ is replaced by $\tilde{\mathbb{G}}_N^{\pi,(1),b\#}$ with $\#\in\{c,cc\}$.
\end{lemma}

\begin{proof}
As in the proof of Lemma \ref{lemma:phIbootIPWGC}, we proceed by conditioning on $\xi$ and then take expectations with respect to $\xi$.

First, we prove the claim for $\tilde{\mathbb{G}}_N^{\pi,(1)}$.
Recall the decomposition of the phase I bootstrap IPW empirical process given by
\begin{eqnarray*}
\tilde{\mathbb{G}}_N^{\pi,(1)}
=\sum_{j=1}^J\sqrt{\frac{N_j}{N}}\sqrt{\frac{N_j}{n_j}}\tilde{\mathbb{G}}_{j,n_j}^{\xi,(1)},
\end{eqnarray*}
where
$\tilde{\mathbb{G}}_{j,n_j}^{\xi,(1)}
=n_j^{-1/2}\sum_{i=1}^{N_j}(W_{N_j,j,i}^{(1)}-1)\xi_{j,i}\delta_{X_{j,i}}$.
Note that $\tilde{W}_{N_j,j,i}^{(1)}=W_{N_j,j,i}^{(1)}-1$ has mean zero and variance $c_j^2$ satisfying $\lVert \tilde{W}_{N_j,j,i}^{(1)}\rVert_{2,1}<\infty$ in view of Problem 2.9.2 of \cite{MR1385671}.
Because we showed $\lVert P_{0|j}\rVert_{\mathcal{F}}<\infty$ for $j=1,\ldots,J$, in the proof of Lemma \ref{lemma:phIbootIPWGC}, it follows from the uncentered conditional multiplier central limit theorem (Lemma \ref{lemma:ph1D}) applied to each of the phase I bootstrap IPW empirical processes $\tilde{\mathbb{G}}_{j,n_j}^{\xi,(1)}$ for the $j$th stratum  $j=1,\ldots,J,$ that
\begin{equation*}
\tilde{\mathbb{G}}_{j,n_j}^{\xi,(1)}
\rightsquigarrow
c_j(\mathbb{G}_j^{(1)} + Z_jP_{0|j})
\quad \mbox{in }\ell^\infty(\mathcal{F}),
\end{equation*}
in $P^*$-probability where $\mathbb{G}_j^{(1)}$ is a $P_{0|j}$-Brownian bridge process independent of the standard normal random variables $Z_j$.
The seminorm for asymptotic equicontinuity is $\rho_j^{(1)}(f,g)=\{P_{0|j}(f-g)^2\}^{1/2}$.
Note that $\mathbb{G}_j^{(1)}$ and $Z_j$, $j=1,\ldots,J$, are all independent and that $\mathbb{G}_j^{(1)}+Z_jP_{0|j}$ are $P_{0|j}$-Brownian motion processes.
Hence
\begin{eqnarray*}
\tilde{\mathbb{G}}_N^{\pi,(1)}
\rightsquigarrow
\sum_{j=1}^J\sqrt{\frac{\nu_j}{p_j}}\sqrt{\frac{p_j}{2-p_j}}(\mathbb{G}_j^{(1)}+Z_jP_{0|j})
=\sum_{j=1}^J\sqrt{\frac{\nu_j}{2-p_j}}(\mathbb{G}_j^{(1)}+Z_jP_{0|j}),
\end{eqnarray*}
in $\ell^{\infty}(\mathcal{F})$ in $P^*$-probability.

Next, we prove the claim for $\tilde{\mathbb{G}}_N^{\pi,(1),bc}$.
Other cases are similar.
For a finite-dimensional convergence, we have for $f\in\mathcal{F}$ that
\begin{eqnarray*}
\tilde{\mathbb{G}}_N^{\pi,(1),bc}f
&=&\tilde{\mathbb{G}}_N^{\pi,(1)}f+ (\tilde{\mathbb{G}}_N^{\pi,(1),bc}-\tilde{\mathbb{G}}_N^{\pi,(1)})f\\
&=&\tilde{\mathbb{G}}_N^{\pi,(1)}f +(\hat{\mathbb{P}}_N^{\pi,(1)}-\mathbb{P}_N^\pi)\dot{G}_{c}(V;\tilde{\alpha})fV^T\sqrt{N}(\hat{\alpha}_N^c-\alpha_0)
\end{eqnarray*}
where $\tilde{\alpha}$ is some convex combination of $\hat{\alpha}_N^c$ and $\alpha_0$.
It follows from Lemma \ref{lemma:phIbootIPWGC}, Proposition A.1 of \cite{SW2013supp} and Lemma \ref{lemma:bootorder} that the second term is $o_{P^*_W}(1)$ in $P^*$-probability.
For asymptotic equicontinuity, let $h_N \in\mathcal{F}_{\delta_N}\equiv \{f-g:f,g\in\mathcal{F},\sum_{j=1}^J\rho_{j}^{(1)}(f,g)\leq \delta_N\}$
for an arbitrary sequence $\delta_N\downarrow 0$.
We have by the triangle inequality and Taylor's theorem that $\lVert \tilde{\mathbb{G}}_N^{\pi,(1),bc} \rVert_{\mathcal{F}_{\delta_N}}$ is bounded above by
\begin{eqnarray*}
\lVert \tilde{\mathbb{G}}_N^{\pi,(1)} \rVert_{\mathcal{F}_{\delta_N}}
+\lVert (\hat{\mathbb{P}}_N^{\pi,(1)}-\mathbb{P}_N^\pi)\dot{G}_{c}(V;\tilde{\alpha})h_NV^T \rVert_{\mathcal{F}_{\delta_N}}\sqrt{N}(\hat{\alpha}_N^c-\alpha_0),
\end{eqnarray*}
where $\tilde{\alpha}$ is some convex combination of $\hat{\alpha}_N^c$ and $\alpha_0$.
The first term in the last display is $o_{P_W^*}(1)$ in $P^*$-probability because of the asymptotic equicontinuity of $\tilde{\mathbb{G}}_N^{\pi,(1)}$.
The second term is also  $o_{P_W^*}(1)$ in $P^*$-probability by the Glivenko-Cantelli preservation theorem \cite{MR1857319}, Lemma \ref{lemma:phIbootIPWGC}, Proposition A.1 of \cite{SW2013supp} and Lemma \ref{lemma:bootorder} as above.
This completes the proof.
\end{proof}

\subsubsection{Phase II Bootstrap}
We present results concerning the bootstrap empirical process based on \cite{gross} and \cite{MR740906} with the phase II bootstrap weights only (see Lemma \ref{lemma:ph2D}).
We first prove the Glivenko-Cantelli theorem for the phase II bootstrap IPW empirical processes.
\begin{lemma}
\label{lemma:phIIbootIPWGC}
Let $\mathcal{F}$ be a $P_0$-Glivenko-Cantelli class.
Then
\begin{eqnarray*}
&&\lVert \hat{\mathbb{P}}_N^{\pi,(2)}-\mathbb{P}_N^\pi\rVert_{\mathcal{F}} \rightarrow_{P_W^*} 0 ,\quad \mbox{in $P^*$-probability}.
\end{eqnarray*}
Suppose moreover that $\lVert P_0\rVert_{\mathcal{F}}<\infty$.
Then
\begin{eqnarray*}
&&\lVert \hat{\mathbb{P}}_N^{\pi,(2),*\#}-\mathbb{P}_N^\pi\rVert_{\mathcal{F}} \rightarrow_{P_W^*} 0, \quad \mbox{in $P^*$-probability}
\end{eqnarray*}
where $*\in \{b,bs\}$ and $\# \in \{c,cc\}$.

The statements above hold if we replace $\mathbb{P}_N^\pi$ by $\mathbb{P}_N^{\pi,\#}$ or $P_0$.
\end{lemma}
\begin{proof}
The first statement was proved in \cite{Saegusa} (Theorem 7.1.1) and \cite{Saegusa-Var:2014} (Lemma 4.1).

We consider $\lVert \hat{\mathbb{P}}_N^{\pi,(2),bc}-\mathbb{P}_N^\pi\rVert_{\mathcal{F}}$. Other cases are similar.
Note that
$\tilde{\mathcal{F}}_1=\{g(x,v)=G_{c}(v;\alpha)f(x):f\in\mathcal{F},\alpha\in\mathbb{R}^k\}$ and $\tilde{\mathcal{F}}_2=\{g(x,v) = \dot{G}_{c}(v;\alpha)v^Tf(x):f\in\mathcal{F},\alpha\in\mathbb{R}^k\}$ are $P_0$-Glivenko-Cantelli by the Glivenko-Cantelli preservation theorem (Theorem 3, \cite{MR1857319}).
Taylor's theorem yields
\begin{eqnarray*}
&&\left\lVert \hat{\mathbb{P}}_N^{\pi,(2),bc}-\mathbb{P}_N^\pi\right\rVert_{\mathcal{F}}\\
&&\leq \left\lVert (\hat{\mathbb{P}}_N^{\pi,(2)}-\mathbb{P}_N^\pi)G_{c}(V;\hat{\hat{\alpha}}_N^{bc})f\right\rVert_{\mathcal{F}}+ \left \lVert \mathbb{P}_N^\pi \dot{G}_{c}(V;\tilde{\alpha})V^Tf(\hat{\hat{\alpha}}_N^{bc}-\alpha_0) \right\rVert_{\mathcal{F}}\\
&&\leq \left\lVert \hat{\mathbb{P}}_N^{\pi,(2)}-\mathbb{P}_N^\pi\right\rVert_{\tilde{\mathcal{F}}_1}+ (\left \lVert \mathbb{P}_N^\pi-P_0 \right\rVert_{\mathcal{\tilde{F}}_2}+\left \lVert P_0 \right\rVert_{\mathcal{\tilde{F}}_2})(\hat{\hat{\alpha}}_N^{bc}-\alpha_0)
\end{eqnarray*}
where $\tilde{\alpha}$ is some convex combination of $\hat{\hat{\alpha}}_N^{bc}$ and $\alpha_0$.
The first term in the last display is $o_{P^*_W}(1)$ in $P^*$-probability by the first part of the theorem.
The second term is also $o_{P^*_W}(1)$ in $P^*$-probability because
$\lVert \mathbb{P}_N^{\pi}-P_0\rVert_{\tilde{\mathcal{F}}_2}$ is $o_{P_W^*}(1)$ in $P^*$-probability by Theorem 5.1 of \cite{MR3059418} and Lemma \ref{lemma:bootorder}, and the facts that $\lVert P_0\rVert_{\tilde{\mathcal{F}}_2}<\infty$ and that $\hat{\hat{\alpha}}_N^{bc}-\alpha_0=o_{P_W^*}(1)$ in $P^*$-probability by Proposition \ref{prop:alphaCalboot}.
The last statement holds by the triangle inequality.
\end{proof}

We prove weak convergence of the phase II bootstrap IPW empirical processes.
\begin{proof}[Proof of Lemma \ref{lemma:ph2D}]
The first statement was proved in \cite{Saegusa} (Theorem 7.3.1) and \cite{Saegusa-Var:2014} (Lemma 4.1).

For bootstrap calibrations, we prove the claim for $\tilde{\mathbb{G}}_N^{\pi,(2),bc}$.
The case for $\tilde{\mathbb{G}}_N^{\pi,(2),bcc}$ is similar.
Let $\hat{\hat{\alpha}}_N=\hat{\hat{\alpha}}_N^{bc}$ and $\hat{\alpha}_N = \hat{\alpha}_N^c$.
We have
\begin{eqnarray}
\tilde{\mathbb{G}}_N^{\pi,(2),bc}f
&=&\tilde{\mathbb{G}}_N^{\pi,(2)}f+ \sqrt{N}\hat{\mathbb{P}}_N^{\pi,(2)}(G_{c}(V;\hat{\hat{\alpha}}_N)-1)f\nonumber - \sqrt{N}\mathbb{P}_N^{\pi}(G_{c}(V;\hat{\alpha}_N)-1)f\nonumber\\
&=&\tilde{\mathbb{G}}_N^{\pi,(2)}f+ \sqrt{N}\mathbb{P}_N^{\pi}(G_{c}(V;\hat{\hat{\alpha}}_N)-G_{c}(V;\hat{\alpha}_N))f\nonumber\\
&& + \sqrt{N}(\hat{\mathbb{P}}_N^{\pi,(2)}-\mathbb{P}_N^{\pi})(G_{c}(V;\hat{\hat{\alpha}}_N)-1)f. \label{eqn:bootph2D}
\end{eqnarray}
The first term in the last display converges to $\mathbb{G}^{\pi}f$ in $P^*$-probability.
The second term can be written as
\begin{eqnarray*}
\sqrt{N}\mathbb{P}_N^{\pi}(G_{c}(V;\hat{\hat{\alpha}}_N)-G_{c}(V;\hat{\alpha}_N))f
=\mathbb{P}_N^{\pi}\dot{G}_{c}(V;\tilde{\alpha})V^Tf\sqrt{N}(\hat{\hat{\alpha}}_N-\hat{\alpha}_N)
\end{eqnarray*}
where $\tilde{\alpha}$ is some convex combination of $\hat{\alpha}_N$ and $\hat{\hat{\alpha}}_N$.
Consistency of $\hat{\alpha}_N$ and $\hat{\hat{\alpha}}_N$ for $\alpha_0$ in $P^*$-probability implies consistency of $\tilde{\alpha}$ for $\alpha_0$ in $P^*$-probability.
It follows from Theorem 5.1 of \cite{MR3059418} and Lemma \ref{lemma:bootorder} that
\begin{equation*}
\mathbb{P}_N^{\pi}\dot{G}_{c}(V;\tilde{\alpha})V^Tf
= P\dot{G}_{c}(V;\tilde{\alpha})V^Tf + o_{P^*_W}(1), \quad \mbox{in $P^*$-probability.}
\end{equation*}
As in the consistency proof for $\hat{\hat{\alpha}}_N^{bcc}$ in Proposition \ref{prop:alphaCalboot} we can show that $P\dot{G}_{c}(V;\tilde{\alpha})V^Tf=P_0\dot{G}_{c,\infty}(V;\alpha_0)V^Tf+o_{P^*_W}(1)$ in $P^*$-probability.
Thus, it follows from Proposition \ref{prop:alphaCalboot} that the second term in (\ref{eqn:bootph2D}) converges to
$-\sum_{j=1}^J\sqrt{\nu_j}\sqrt{(1-p_j)/p_j}\mathbb{G}_jQ_cf$
in $P^*$-probability.
The third term in (\ref{eqn:bootph2D}) is $o_{P_W^*}(1)$ in $P^*$-probability by proceeding in the same way as in the consistency proof for $\hat{\hat{\alpha}}_N^{bcc}$ in Proposition \ref{prop:alphaCalboot} using Taylor's theorem.
This verifies the finite-dimensional convergence.
For asymptotic equicontinuity, proceed in the same way as in the proof of Lemma \ref{lemma:phIbootIPWD}.

For bootstrap single calibrations, we prove the claim for $\tilde{\mathbb{G}}_N^{\pi,(2),bsc}$.
The case for $\tilde{\mathbb{G}}_N^{\pi,(2),bscc}$ is similar.
We have by Taylor's theorem that
\begin{eqnarray*}
\tilde{\mathbb{G}}_N^{\pi,(2),bsc}f
&=&\tilde{\mathbb{G}}_N^{\pi,(2),bsc}f-\tilde{\mathbb{G}}_N^{\pi,(2)}f
+\tilde{\mathbb{G}}_N^{\pi,(2)}f\\
&=&\tilde{\mathbb{G}}_N^{\pi,(2)}f + \sqrt{N}(\hat{\mathbb{P}}_N^{\pi,(2),bsc}-\hat{\mathbb{P}}_N^{\pi,(2)})f \\
&=&\tilde{\mathbb{G}}_N^{\pi,(2)}f + \hat{\mathbb{P}}_N^{\pi,(2)} \dot{G}_{c}(V;\tilde{\alpha})V^Tf\sqrt{N}(\hat{\hat{\alpha}}_N^{bsc}-\alpha_0)
\end{eqnarray*}
where $\tilde{\alpha}$ is some convex combination of $\alpha_0$ and $\hat{\hat{\alpha}}_N^{bsc}$.
Apply the bootstrap Glivenko-Cantelli theorem (Theorem \ref{thm:fpsPiGCboot}) and Proposition \ref{prop:alphaCalboot} to obtain the finite-dimensional convergence.
For asymptotic equicontinuity, proceed in the same way as in the proof of Lemma \ref{lemma:phIbootIPWD}.
This completes the proof.
\end{proof}

\subsubsection{Bootstrap Glivenko-Cantelli and Donsker theorems}
We now combine results regarding phase I and II bootstrap to prove our bootstrap Glivenko-Cantelli and Donsker Theorems (Theorems \ref{thm:fpsPiGCboot} and \ref{thm:fpsDboot}).

\begin{proof}[Proof of Theorem \ref{thm:fpsPiGCboot}]
We consider $\lVert \hat{\mathbb{P}}_N^\pi - \mathbb{P}_N^\pi \rVert_{\mathcal{F}}$. Proofs for other cases are similar.
For $f\in\mathcal{F}$, we have
\begin{eqnarray*}
\lVert \hat{\mathbb{P}}_N^\pi - \mathbb{P}_N^\pi \rVert_{\mathcal{F}}
\leq \lVert \hat{\mathbb{P}}_N^{\pi,(1)} - \mathbb{P}_N^\pi \rVert_{\mathcal{F}} + \lVert (\hat{\mathbb{P}}_N^{\pi,(2)} - \mathbb{P}_N^\pi)W_N^{(1)}f \rVert_{\mathcal{F}}
\end{eqnarray*}
The first term is $o_{P_W^*}(1)$ in $P^*$-probability by Lemma \ref{lemma:phIbootIPWGC}.
Because $\mathcal{F}_{\mathcal{W}}=\{g:g(x,w)=wf(x),f\in\mathcal{F}\}$ is $P_0\times P_{W^{(1)}}$-Glivenko-Cantelli by the Glivenko-Cantelli preservation theorem of \cite{MR1857319}, we apply Lemma \ref{lemma:phIIbootIPWGC} to the second term in the last display to conclude all terms in the last display are $o_{P_W^*}(1)$ in $P^*$-probability.
\end{proof}

To establish weak convergence of the bootstrap IPW empirical processes, we decompose them to the phase I and II bootstrap IPW empirical processes.
For the phase II bootstrap IPW empirical processes, their weak convergence is obtained conditional on the phase I bootstrap weights as well as data, as treated in the following lemma.
\begin{lemma}
\label{lemma:phIIbootIPWD}
Let $\mathcal{F}$ be a $P_0$-Donsker class with $\lVert P_0\rVert_{\mathcal{F}}<\infty$.
Then,
\begin{eqnarray*}
&&\tilde{\mathbb{G}}_N^{\pi,(2)}W^{(1)}_N\cdot
\rightsquigarrow\sum_{j=1}^J\sqrt{\nu_j}\sqrt{\frac{1-p_j}{p_j}}\mathbb{G}_j(W_j^{(1)}\cdot),\\
&&\tilde{\mathbb{G}}_N^{\pi,(2),*\#}W^{(1)}_N\cdot
\rightsquigarrow\sum_{j=1}^J\sqrt{\nu_j}\sqrt{\frac{1-p_j}{p_j}}\mathbb{G}_j\{(I-Q_{\#})W_j^{(1)}\cdot\},
%
\end{eqnarray*}
in $\ell^\infty(\mathcal{F})$ in $P^*\times P_W^{(1)}$-probability where $*\in \{b,bs\}$ and $\# \in \{c,cc\}$, $P_{0|j}$-Brownian bridge processes $\mathbb{G}_j$ and $\mathbb{G}_j^{(1)}$,
$W_j^{(1)}$s are independent with mean 1 and variance $c_j^2$ that are independent of $(X,V)$,
and $Q_{\#}$ are defined in Theorem \ref{thm:fpsD}.
\end{lemma}
\begin{proof}
We prove the claim for $\tilde{\mathbb{G}}_N^{\pi,(2)}$. Proofs for other cases are similar.
Note that $(W_{N_j,j,i}^{(1)},X_{j,i})$, $i=1,\ldots,N_j$, with $\xi_{j,i}=1$ are independent.
Since $\mathcal{F}_{\mathcal{W}_j}=\{g(x,w)=wf(x):f\in\mathcal{F}\}$ is $P_{0|j}\times P_{W_j}^{(1)}$-Donsker by Lemma \ref{lemma:wfD}, it follows from Lemma \ref{lemma:ph2D} applied to a single stratum with $(N_j/n_j)^{1/2}\rightarrow p_j^{-1/2}$ that
$\tilde{\mathbb{G}}_{j,N_j}^{\xi,(2)}
\rightsquigarrow \sqrt{1-p_j}\mathbb{G}_j$ in $\ell^\infty(\mathcal{F}_{\mathcal{W}_j})$,
conditionally on $(X_{j,1},W_{N_j,j,1}^{(1)}),(X_{j,2},W_{N_j, j,2}^{(1)})\ldots$.
Note that $\mathbb{G}_j$ are independent of $\mathbb{G}_j^{(1)}$ and $Z_j$ for $j=1,\ldots,J$.
Hence it follows that
\begin{eqnarray*}
\tilde{\mathbb{G}}_N^{\pi,(2)}
\rightsquigarrow \sum_{j=1}^J\sqrt{\nu_j}\sqrt{\frac{1-p_j}{p_j}}\mathbb{G}_j(W_j^{(1)}\cdot)\quad \quad \mbox{in }\ell^\infty(\mathcal{F}),
\end{eqnarray*}
in $P^*\times P_W^{(1)}$-probability.
\end{proof}

We prove weak convergence of the bootstrap IPW empirical processes.
\begin{proof}[Proof of Theorem \ref{thm:fpsDboot}]

We prove the claims for $\tilde{\mathbb{G}}_N^\pi$ and $\tilde{\mathbb{G}}_N^{\pi,bc}$.
Other cases are similar.
For $\tilde{\mathbb{G}}_N^\pi$, decompose $\tilde{\mathbb{G}}_N^\pi$ into $\tilde{\mathbb{G}}_N^{\pi,(1)}+\tilde{\mathbb{G}}_N^{\pi,(2)}W_N^{(1)}\cdot$.
Apply Lemma \ref{lemma:phIIbootIPWD} conditionally on the phase I bootstrap weights and data to obtain weak convergence of the second term.
Then  apply Lemma \ref{lemma:phIbootIPWD} to obtain weak convergence of the first term.
Recall that $c_j^2=p_j/(2-p_j)$ and $W_j^{(1)}$ is independent of $X$.
Its covariance function evaluated at $f,g\in\mathcal{F}$ is given by
\begin{eqnarray*}
&&\sum_{j=1}^J\left\{\frac{\nu_j}{2-p_j} P_{0|j}(f-g)^2 + \nu_j\frac{1-p_j}{p_j}\mathrm{Var}_{0|j}(W_j^{(1)}f-W_j^{(1)}g)\right\}\\
\end{eqnarray*}
Since
$\mathrm{Var}_{0|j}(W_j^{(1)}f-W_j^{(1)}g) = (c_j^2+1)P_{0|j}(f-g)^2-\{P_{0|j}(f-g)\}^2$,
the covariance function reduces to
\begin{equation*}
P_0(f-g)^2 + \sum_{j=1}^J\nu_j\frac{1-p_j}{p_j}\mathrm{Var}_{0|j}(f-g).
\end{equation*}
This is the same as the covariance function for the process of our claim, and hence the result follows.

Next, we consider the claim for $\tilde{\mathbb{G}}_N^{\pi,bc}$.
Since the conditional independence of $W_j^{(1)}$ and $(X,V)$ given stratum membership and $P_{0|j}W_j^{(1)}=1$ yields
\begin{eqnarray*}
P_0(W^{(1)}_Nf V^T)= \sum_{j=1}^JP_{0|j} (W^{(1)}_j) P_{0|j}(f V^T)\nu_j = P_0(f V^T),
\end{eqnarray*}
we have $Q_{c}W^{(1)}_jf = P_0(W^{(1)}_Nf V^T)\{P_0V^{\otimes 2}\}^{-1}V = Q_{c}f$ and that $Q_{c}W^{(1)}_j f-Q_{c}W^{(1)}_jg = Q_{c}(f-g)$.
It follows that
\begin{eqnarray*}
&&\mathrm{Var}_{0|j}((I-Q_{c})W_j^{(1)}f-(I-Q_{c})W_j^{(1)}g)\\
&&= c_j^2 P_{0|j}(f-g)^2 + \mathrm{Var}_{0|j}(f-g)+ \mathrm{Var}(Q_{c}(f-g)) \\
&&\quad -2 [P_{0|j}W_j^{(1)} P_{0|j}\{(f-g)Q_{c}(f-g)\}
- P_{0|j}W^{(1)} P_{0|j}(f-g)P_{0|j}Q_{c}(f-g)]\\
&&= c_j^2 P_{0|j}(f-g)^2 +\mathrm{Var}_{0|j}(f-g)+ \mathrm{Var}(Q_{c}(f-g))\\
&&\quad  -2 [P_{0|j}\{(f-g)Q_{c}(f-g)\}
-P_{0|j}(f-g)P_{0|j}Q_{c}(f-g)] \\
&&=c_j^2 P_{0|j}(f-g)^2 + \mathrm{Var}_{0|j}((I-Q_{c})f-(I-Q_{c})g).
\end{eqnarray*}
Proceed similarly to the case for $\tilde{\mathbb{G}}_N^\pi$ to compute the covariance function of the limiting process for $\tilde{\mathbb{G}}_N^{\pi,bc}$ evaluated at $f$ and $g$, and verify that the covariance function is the same as that of $\tilde{\mathbb{G}}^{\pi,c}$ as desired.
\end{proof}

We prove weak convergence of the centered bootstrap IPW empirical processes based on Theorem \ref{thm:fpsDboot}.
\begin{proof}[Proof of Theorem \ref{thm:fpsDboot2}]
We only prove the claim for $\hat{\mathbb{G}}_N^{\pi,bc}$.
The other cases are similar.
We have
\begin{eqnarray*}
\hat{\mathbb{G}}_{N}^{\pi,bc}f
&=&\tilde{\mathbb{G}}_{N}^{\pi,bc}(f-P_0f)  - \tilde{\mathbb{G}}_{N}^{\pi,bc}(\mathbb{P}_N^{\pi,c}f-P_0f)\\
&=&\tilde{\mathbb{G}}_{N}^{\pi,bc}(f-P_0f)  + (\mathbb{P}_N^{\pi,c}-P_0)f\tilde{\mathbb{G}}_{N}^{\pi,bc}1\\
&=&\tilde{\mathbb{G}}_{N}^{\pi,bc}(f-P_0f)  + \mathbb{G}_N^{\pi,c}f(\hat{\mathbb{P}}_{N}^{\pi,bc}-\mathbb{P}_N^{\pi,c})1.
\end{eqnarray*}
The first term in the last display converges to $\tilde{\mathbb{G}}^{\pi,c}(f-P_0f) = \mathbb{G}^{\pi,c}f$ as desired.
The second term is $o_{P_W^*}(1)$ in $P^*$-probability.
To see this, note that $(\hat{\mathbb{P}}_{N}^{\pi,bc}-\mathbb{P}_N^{\pi,c})1 = o_{P_W^*}(1)$ in $P^*$-probability by Theorem \ref{thm:fpsPiGCboot}, and that $\mathbb{G}_N^{\pi,c}f=O_{P_W^*}(1)$ by Theorem \ref{thm:fpsD} and Lemma \ref{lemma:bootorder}.
The asymptotic equicontinuity can be established in the same way as in the proof of \ref{lemma:phIbootIPWD} together with the above decomposition.
\end{proof}

\subsubsection{General Semiparametric Models}
\label{subsec:semipara}
We prove the lemma below to prove Theorem \ref{thm:bootzthm1} as its corollary.
Suppose $\mathcal{P}$ is the collection of probability measures on $(\mathcal{X},\mathcal{A})$ parametrized by $\theta\in \Theta$ where $\Theta$ is a subset of a Banach space $(\mathcal{B},\lVert\cdot \rVert)$. 
The true distribution is $P_0 = P_{\theta_0}\in \mathcal{P}$.
Let $\hat{\theta}_{N}$ $\hat{\theta}_{N,c}$ and $\hat{\theta}_{N,cc}$ be estimators of $\theta$ obtained as solutions to the IPW estimating equations given by
\begin{eqnarray*}
&&\left\lVert\Psi^\pi_{N}(\theta)\right\rVert_{\mathcal{H}}\equiv \left\lVert \mathbb{P}_N^{\pi} B(\theta)\right\rVert_{\mathcal{H}} = o_{P^*}(N^{-1/2}),\\
&&\left\lVert\Psi^\pi_{N,\#}(\theta)\right\rVert_{\mathcal{H}}\equiv \left\lVert \mathbb{P}_N^{\pi,\#} B(\theta)\right\rVert_{\mathcal{H}} = o_{P^*}(N^{-1/2}),\quad \#\in \{c,cc\},
\end{eqnarray*}
respectively where $B(\theta)$ is a map from some index set $\mathcal{H}$ to $\mathbb{R}$ indexed by $\theta$.
Let also $\hat{\hat{\theta}}_{N}$ and $\hat{\hat{\theta}}_{N,*\#}$, $*\in\{b,bs\}$, be bootstrap estimators of $\theta$ obtained as solutions to the bootstrap IPW estimating equations given by
\begin{eqnarray*}
&&\left\lVert\hat{\Psi}_{N}^{\pi}(\theta)\right\rVert_{\mathcal{H}}\equiv \left\lVert \hat{\mathbb{P}}_N^{\pi} B(\theta)\right\rVert_{\mathcal{H}} = o_{P^*_W}(N^{-1/2}),\\
&&\left\lVert\hat{\Psi}_{N,*\#}^{\pi}(\theta)\right\rVert_{\mathcal{H}}\equiv \left\lVert \hat{\mathbb{P}}_N^{\pi,*\#} B(\theta)\right\rVert_{\mathcal{H}} = o_{P^*_W}(N^{-1/2}),\quad \#\in \{c,cc\},
\end{eqnarray*}
in $P^*$-probability, respectively.
Let $\Psi(\theta)\equiv P_0B(\theta)$ and $\Psi_N(\theta)\equiv \mathbb{P}_NB(\theta)$ be maps from $\Theta$ to $\ell^\infty(\mathcal{H})$.

\begin{cond}
\label{cond:survbootwz1}
For the true parameter $\theta_0\in\Theta$, $\Psi(\theta_0)=0$. The set $\{B(\theta_0)h:h\in\mathcal{H}\}$ is $P_0$-Donsker and $\{(B(\theta)-B(\theta_0))h:\theta\in\Theta,h\in\mathcal{H}\}$ is $P_0$-Glivenko-Cantelli with an integrable envelope.
\end{cond}
\begin{cond}
\label{cond:survbootwz2}
Suppose that $\Psi$ is Fr\'{e}chet differentiable at $\theta_0$;
\begin{equation*}
\left\lVert \Psi(\theta)-\Psi(\theta_0)-\dot{\Psi}_0(\theta-\theta_0)\right\rVert_{\mathcal{H}}=o\left(\lVert \theta-\theta_0\rVert\right).
\end{equation*}
Moreover, $\dot{\Psi}_0$ is continuously invertible at $\theta_0$ with inverse denoted as $\Psi_0^{-1}$
\end{cond}
\begin{cond}
\label{cond:survbootwz3}
For any $\delta_N\rightarrow 0$, the following stochastic equicontinuity condition holds at $\theta_0$;
\begin{equation*}
\sup_{\lVert\theta-\theta_0\rVert\leq \delta_N}\lVert \sqrt{N}(\Psi_N-\Psi)(\theta)-\sqrt{N}(\Psi_N-\Psi)(\theta_0)\rVert_{\mathcal{H}}  = o_{P^*}(1+\sqrt{N}\lVert\theta-\theta_0\rVert).
\end{equation*}
\end{cond}

\begin{thm}
\label{thm:survbootzthm1}
Suppose that Conditions \ref{cond:survbootwz1}-\ref{cond:survbootwz3} hold and that estimators $\hat{\theta}_N,\hat{\theta}_{N,\#},\hat{\hat{\theta}}_{N},\hat{\hat{\theta}}_{N,*\#}$ with $*\in\{b,bs\}$ and $\#\in \{c,cc\}$ are consistent for $\theta_0$ (in $P^*$-probability).
Then
\begin{eqnarray*}
&&\sqrt{N}(\hat{\hat{\theta}}_{N}-\hat{\theta}_N)
\rightsquigarrow -\dot{\Psi}_0^{-1}\tilde{\mathbb{G}}^{\pi}B(\theta_0)\\
&&\sqrt{N}(\hat{\hat{\theta}}_{N,b\#}-\hat{\theta}_{N,\#})
\rightsquigarrow -\dot{\Psi}_0^{-1}\tilde{\mathbb{G}}^{\pi,\#}B(\theta_0) ,\\
&&\sqrt{N}(\hat{\hat{\theta}}_{N,s\#}-\hat{\theta}_N)
\rightsquigarrow -\dot{\Psi}_0^{-1}\tilde{\mathbb{G}}^{\pi,\#}B(\theta_0) ,
\quad \mbox{in $P^*$-probability.}
\end{eqnarray*}
\end{thm}

\begin{proof}
We prove the claim for $\hat{\hat{\theta}}_{N,bc}$.
First, Theorem \ref{thm:fpsDboot} together with Condition \ref{cond:survbootwz1} yields
\begin{equation*}
\tilde{\mathbb{G}}^{\pi,bc}_NB(\theta_0)
\rightsquigarrow
\tilde{\mathbb{G}}^{\pi,c}B(\theta_0), \quad \mbox{in }\ell^\infty(\mathcal{H}),\ \mbox{in $P^*$-probability.}
\end{equation*}

For a fixed arbitrary sequence $\{\delta_N\}$ with $\delta_N\rightarrow 0$, let
\begin{eqnarray*}
\mathcal{D}_N
&\equiv&\left\{\frac{B(\theta)(h)-B(\theta_0)(h)}{1+\sqrt{N}\lVert\theta-\theta_0\rVert}:h\in\mathcal{H},\lVert\theta-\theta_0\rVert \leq \delta_N\right\}\\
&\equiv&\left\{B_N(\theta,\theta_0)(h):h\in\mathcal{H},\lVert\theta-\theta_0\rVert \leq \delta_N\right\}.
\end{eqnarray*}
Condition \ref{cond:survbootwz3} can be written as $\lVert \mathbb{G}_N\rVert_{\mathcal{D}_N}=o_{P^*}(1)$.
Since $E^*\lVert \delta_X-P_0\rVert_{\mathcal{D}_N} \leq 2E^*\sup_{\theta\in\Theta,h\in\mathcal{H}} |(\delta_X-P_0)B(\theta)h|<\infty$ by assumption, we can apply Lemma \ref{lemma:vw239} to obtain $E\lVert \mathbb{G}_N\rVert_{\mathcal{D}_N}=o(1)$ as $N\rightarrow \infty$.
It follows by Lemma \ref{lemma:gpleqgboot} that $E\lVert \tilde{\mathbb{G}}_N^\pi\rVert_{\mathcal{D}_N}=o(1)$ and hence
$\lVert \tilde{\mathbb{G}}_N^\pi\rVert_{\mathcal{D}_N}=o_{P^*_W}(1)$ in $P^*$-probability by Markov's inequality and Lemma \ref{lemma:bootorder}.
Taylor's theorem yields that for $f\in\mathcal{D}_N$
\begin{eqnarray*}
&&\tilde{\mathbb{G}}_N^{\pi,bc}f-\tilde{\mathbb{G}}_N^{\pi}f\\
&&=\tilde{\mathbb{G}}_N^{\pi}(G_c(V;\hat{\hat{\alpha}}_N^{bc})-1)f
+(\mathbb{G}_N^\pi+P_0) (G_c(V;\hat{\hat{\alpha}}_N^{bc})-G_c(V;\hat{\alpha}_N^c))f\\
&&=(\hat{\mathbb{P}}_N^{\pi}-\mathbb{P}_N^{\pi})\dot{G}_c(V;\tilde{\alpha}_1)fV^T\sqrt{N}(\hat{\hat{\alpha}}_N^{bc}-\alpha_0)\\ && \quad
+(\mathbb{P}_N^{\pi}-P_0)\dot{G}_c(V;\tilde{\alpha}_2)fV^T\sqrt{N}(\hat{\hat{\alpha}}_N^{bc}-\hat{\alpha}_N^c)\\
&&\quad +P_0 (G_c(V;\hat{\hat{\alpha}}_N^{bc})-G_c(V;\hat{\alpha}_N^c))f
\end{eqnarray*}
where $\tilde{\alpha}_1$ and $\tilde{\alpha}_2$ are  some convex combinations of $\hat{\hat{\alpha}}_N^{bc}$ and $\alpha_0$, and $\hat{\hat{\alpha}}_N^{bc}$ and $\hat{\alpha}_N^{c}$, respectively.
For the first two terms in the last display, note that $\sqrt{N}(\hat{\hat{\alpha}}_N^{bc}-\hat{\alpha}_N^c)$ and $\sqrt{N}(\hat{\hat{\alpha}}_N^{bc}-\alpha_0)$ are $O_{P_W^*}(1)$ in $P^*$-probability by Proposition \ref{prop:alphaCalboot}, the result in its proof, and Lemma \ref{lemma:bootorder}.
Note also that $G$ and $\dot{G}$ are bounded and $V$ has a bounded support.
Thus, we can apply the Glivenko-Cantelli preservation theorem of \cite{MR1857319}, Theorem \ref{thm:fpsPiGCboot} together with Condition \ref{cond:survbootwz1}, Theorem 5.1 of \cite{MR3059418} and Lemma \ref{lemma:bootorder} to show the supremum of the absolute values of the first two terms over $\mathcal{D}_N$ are $o_{P_W^*}(1)$ in $P^*$-probability.
For the third term, note that $\hat{\hat{\alpha}}_N^{bc}$ and $\hat{\alpha}_N^c$ are consistent for $\alpha_0$ by Proposition \ref{prop:alphaCalboot} and Proposition A.1 of \cite{SW2013supp}.
Since $\mathcal{D}_N$ has an integrable envelope by assumption and $G_c$ is bounded, a subsequence argument with the dominated convergence theorem as in the proof of Proposition \ref{prop:alphaCalboot} implies that the supremum of the absolute value of the third term over $\mathcal{D}_N$ is $o_{P_W^*}(1)$ in $P^*$-probability.
Hence we have $\lVert \tilde{\mathbb{G}}_N^{\pi,bc}-\tilde{\mathbb{G}}_N^{\pi}\rVert_{\mathcal{D}_N} = o_{P_W^*}(1)$ in $P^*$-probability.
It follows by the triangle inequality that $\lVert\tilde{\mathbb{G}}_N^{\pi,bc}\rVert_{\mathcal{D}_N} = o_{P_W^*}(1)$ in $P^*$-probability.
A similar argument with the help of Lemma \ref{lemma:bootorder} shows that $\lVert\mathbb{G}_N^{\pi,c}\rVert_{\mathcal{D}_N} = o_{P_W^*}(1)$ in $P^*$-probability.
Thus, consistency of $\hat{\hat{\theta}}_{N,bc}$ and $\hat{\theta}_{N,c}$ to $\theta_0$ in $P^*$-probability and Condition \ref{cond:survbootwz3} imply that
\begin{eqnarray}
&&\lVert \mathbb{G}_N^{\pi,c} (B(\hat{\theta}_{N,c})-B(\theta_0))\rVert_{\mathcal{H}} = o_{P^*_W}(1+\sqrt{N}\lVert \hat{\theta}_{N,c}-\theta_0\rVert), \nonumber \\
&&\lVert \mathbb{G}_N^{\pi,c} (B(\hat{\hat{\theta}}_{N,bc})-B(\theta_0))\rVert_{\mathcal{H}} = o_{P^*_W}(1+\sqrt{N}\lVert \hat{\hat{\theta}}_{N,bc}-\theta_0\rVert),\nonumber \\
&&\lVert \tilde{\mathbb{G}}_N^{\pi,bc} (B(\hat{\hat{\theta}}_{N,bc})-B(\theta_0))\rVert_{\mathcal{H}} = o_{P^*_W}(1+\sqrt{N}\lVert \hat{\hat{\theta}}_{N,bc}-\theta_0\rVert),\label{eqn:bootzthmaec}
\end{eqnarray}
in $P^*$-probability.

We prove $\sqrt{N}\lVert \hat{\hat{\theta}}_{N,bc}-\theta_0\rVert =O_{P^*_W}(1)$ in $P^*$-probability.
We have
\begin{eqnarray*}
&&\tilde{\mathbb{G}}^{\pi,bc}_N B(\theta_0)+\mathbb{G}^{\pi,c}_NB(\theta_0) +\sqrt{N}(\Psi(\hat{\hat{\theta}}_{N,bc})-\Psi(\theta_0))\\
&&=\tilde{\mathbb{G}}^{\pi,bc}_N (B(\theta_0)-B(\hat{\hat{\theta}}_{N,bc}))+\mathbb{G}^{\pi,c}_N(B(\theta_0)-B(\hat{\hat{\theta}}_{N,bc})) \\
&&\quad +\sqrt{N}\hat{\mathbb{P}}^{\pi,bc}_NB(\hat{\hat{\theta}}_{N,bc})
+\sqrt{N}P_0B(\theta_0)).
\end{eqnarray*}
Because  $\lVert \hat{\mathbb{P}}_N^{\pi,bc}B(\hat{\hat{\theta}}_{N,bc})\rVert_{\mathcal{H}}=o_{P^*_W}(N^{-1/2})$ in $P^*$-probability and $P_0B(\theta_0) =0$ by assumption, the last display and (\ref{eqn:bootzthmaec}) imply that
\begin{eqnarray*}
&&\lVert \sqrt{N}(\Psi(\hat{\hat{\theta}}_{N,bc})-\Psi(\theta_0))\rVert_{\mathcal{H}}
-\lVert \tilde{\mathbb{G}}_N^{\pi,bc} B(\theta_0)\rVert_{\mathcal{H}}
-\lVert \mathbb{G}_N^{\pi,c} B(\theta_0)\rVert_{\mathcal{H}}\\
&&\leq \lVert \tilde{\mathbb{G}}_N^{\pi,bc} (B(\theta_0)-B(\hat{\hat{\theta}}_{N,bc}))\rVert_{\mathcal{H}} +\lVert\mathbb{G}_N^{\pi,c} (B(\theta_0)-B(\hat{\hat{\theta}}_{N,bc})) \rVert_{\mathcal{H}} + o_{P^*_W}(1)\\
&&=o_{P^*_W}(1)(1+\sqrt{N}\lVert \hat{\hat{\theta}}_{N,bc}-\theta_0\rVert),
\quad \mbox{in $P^*$-probability.}
\end{eqnarray*}
Since the continuous invertibility of $\Psi_0$ at $\theta_0$ implies that there is some constant $c>0$ such that
$c\lVert \hat{\hat{\theta}}_{N,bc}-\theta_0\rVert \leq \lVert \Psi(\hat{\hat{\theta}}_{N,bc})-\Psi(\theta_0)\rVert_{\mathcal{H}}$,
we have
\begin{eqnarray*}
&&c\sqrt{N}\lVert \hat{\hat{\theta}}_{N,bc}-\theta_0\rVert
\leq \lVert \sqrt{N}(\Psi(\hat{\hat{\theta}}_{N,bc})-\Psi(\theta_0))\rVert_{\mathcal{H}}\\
&&\leq \lVert \tilde{\mathbb{G}}_N^{\pi,bc} B(\theta_0)\rVert_{\mathcal{H}}
+\lVert \mathbb{G}_N^{\pi,c} B(\theta_0)\rVert_{\mathcal{H}}
+o_{P^*_W}(1)(1+\sqrt{N}\lVert \hat{\hat{\theta}}_{N,bc}-\theta_0\rVert),
\end{eqnarray*}
in $P^*$-probability.
Note that $\lVert \mathbb{G}_N^{\pi,c} B(\theta_0)\rVert_{\mathcal{H}}=O_{P^*_W}(1)$ and $\lVert\tilde{\mathbb{G}}_N^{\pi,bc} B(\theta_0)\rVert_{\mathcal{H}}=O_{P^*_W}(1)$ in $P^*$-probability by Condition \ref{cond:survbootwz1}, Theorem 5.3 of \cite{MR3059418}, Lemma \ref{lemma:bootorder} and Theorem \ref{thm:fpsDboot}.
Thus, the claim $\sqrt{N}\lVert \hat{\hat{\theta}}_{N,bc}-\theta_0\rVert =O_{P^*_W}(1)$ in $P^*$-probability follows.

Now we prove the asymptotic normality of $\hat{\hat{\theta}}_{N,bc}$.
We have
\begin{eqnarray}
&&\sqrt{N}(\Psi(\hat{\hat{\theta}}_{N,bc})-\Psi(\hat{\theta}_{N,c}))+\tilde{\mathbb{G}}_N^{\pi,bc} B(\theta_0)\nonumber \\
&&=\sqrt{N}\hat{\mathbb{P}}_N^{\pi,bc} B(\hat{\hat{\theta}}_{N,dc}) -\sqrt{N}\mathbb{P}_N^{\pi,c} B(\hat{\theta}_{N,c})
+\mathbb{G}_N^{\pi,c}(B(\hat{\theta}_{N,c})-B(\theta_0))\nonumber \\
&&\quad -\mathbb{G}_N^{\pi,c}(B(\hat{\hat{\theta}}_{N,bc})-B(\theta_0))
-\tilde{\mathbb{G}}_N^{\pi,bc}(B(\hat{\hat{\theta}}_{N,bc})-B(\theta_0)). \label{eqn:bootzthmaec2}
\end{eqnarray}
Since $\sqrt{N}\lVert\hat{\hat{\theta}}_{N,bc} -\theta_0\rVert =O_{P^*_W}(1)$ in $P^*$-probability, we have for the first equation of (\ref{eqn:bootzthmaec}) that
\begin{equation*}
\lVert \tilde{\mathbb{G}}_N^{\pi,bc} (B(\hat{\hat{\theta}}_{N,bc})-B(\theta_0))\rVert_{\mathcal{H}}
=o_{P^*_W}(1)(1+O_{P^*_W}(1))=o_{P^*_W}(1),
\end{equation*}
in $P^*$-probability.
Similar reasoning together with Lemma \ref{lemma:bootorder} implies
\begin{eqnarray*}
&&\lVert \mathbb{G}_N^{\pi,c} (B(\hat{\theta}_{N,c})-B(\theta_0))\rVert_{\mathcal{H}}
=o_{P^*_W}(1)(1+O_{P^*_W}(1))=o_{P^*_W}(1),\\
&&\lVert \mathbb{G}_N^{\pi,c} (B(\hat{\hat{\theta}}_{N,bc})-B(\theta_0))\rVert_{\mathcal{H}}
=o_{P^*_W}(1)(1+O_{P^*_W}(1))=o_{P^*_W}(1),
\end{eqnarray*}
in $P^*$-probability for the last two equations of (\ref{eqn:bootzthmaec}).
Moreover, $\hat{\mathbb{P}}_N^{\pi,bc} B(\hat{\hat{\theta}}_{N,bc})=o_{P^*_W}(N^{-1/2})$ and
$\mathbb{P}_N^{\pi,c} B(\hat{\theta}_{N,c})=o_{P_W^*}(N^{-1/2})$ (by Lemma \ref{lemma:bootorder}) in $P^*$-probability.
Thus,  (\ref{eqn:bootzthmaec2}) becomes
\begin{equation}
\label{eqn:wz414dc}
\sqrt{N}(\Psi(\hat{\hat{\theta}}_{N,bc})-\Psi(\hat{\theta}_{N,c}))
=-\tilde{\mathbb{G}}_N^{\pi,bc} B(\theta_0) +o_{P^*_W}(1)
\end{equation}
in $P^*$-probability.

Fr\'{e}chet differentiability of $\Psi(\theta)$ at $\theta_0$ and $\sqrt{N}$-consistency of $\hat{\theta}_N$ and $\hat{\hat{\theta}}_{N,c}$ together with Lemma \ref{lemma:bootorder} imply that
\begin{eqnarray*}
&&\sqrt{N}(\Psi(\hat{\theta}_{N,c})-\Psi(\theta_0))
=\dot{\Psi}_0\left(\sqrt{N}(\hat{\theta}_{N,c}-\theta_0)\right) +o_{P_W^*}(1),\\
&&\sqrt{N}(\Psi(\hat{\hat{\theta}}_{N,bc})-\Psi(\theta_0))
=\dot{\Psi}_0\left(\sqrt{N}(\hat{\hat{\theta}}_{N,bc}-\theta_0)\right) +o_{P^*_W}(1)
\end{eqnarray*}
in $P^*$-probability.
Subtraction gives
\begin{equation*}
\sqrt{N}(\Psi(\hat{\hat{\theta}}_{N,bc})-\Psi(\hat{\theta}_{N,c}))
=\dot{\Psi}_0\left(\sqrt{N}(\hat{\hat{\theta}}_{N,bc}-\hat{\theta}_{N,c})\right) +o_{P^*_W}(1)
\end{equation*}
in $P^*$-probability.
Combine this with (\ref{eqn:wz414dc}) and use the invertibility of $\dot{\Psi}(\theta)$ at $\theta_0$ to obtain
\begin{equation*}
\sqrt{N}(\hat{\hat{\theta}}_{N,bc}-\hat{\theta}_{N,c})
=-\dot{\Psi}_0^{-1}\tilde{\mathbb{G}}_N^{\pi,bc} B(\theta_0)+o_{P^*_W}(1)
\end{equation*}
in $P^*$-probability.
Apply Theorem \ref{thm:fpsDboot} to obtain a desired result.
A proof for $\hat{\hat{\theta}}_{N,bcc}$ is similar.

For $\hat{\hat{\theta}}_{N,bsc}$, replace $\mathbb{P}_N^{\pi,c}$, $\mathbb{G}_N^{\pi,c}$, and $\hat{\theta}_{N,c}$ by $\mathbb{P}_N^\pi$, $\mathbb{G}_N^\pi$, and $\hat{\theta}_{N}$ in the argument above and proceed in the same way to obtain
\begin{equation*}
\sqrt{N}(\hat{\hat{\theta}}_{N,bsc}-\hat{\theta}_{N})
=-\dot{\Psi}_0^{-1}\tilde{\mathbb{G}}_N^{\pi,bsc} B(\theta_0)+o_{P^*_W}(1)
\end{equation*}
in $P^*$-probability.
A proof for $\hat{\hat{\theta}}_{N,bscc}$ is similar.
\end{proof}

\begin{proof}[Proof of Theorem \ref{thm:bootzthm1}]
This is a corollary of Theorems \ref{thm:survbootzthm1}.
Details are similar to the proof of Theorem 3.1 of \cite{MR3059418}.
\end{proof}

We give a proof of Theorem \ref{thm:bootzthm2}.
\begin{proof}[Proof of Theorem \ref{thm:bootzthm2}]
We consider $\hat{\hat{\theta}}_{N,bc}$ and $\hat{\hat{\theta}}_{N,bsc}$.
Proofs for other cases are similar.
We first consider $\hat{\hat{\theta}}_{N,bc}$.
Since $\hat{\mathbb{P}}_N^{\pi,bc}\dot{\ell}_{\hat{\hat{\theta}}_{N,bc},\hat{\hat{\eta}}_{N,bc}} = o_{P^*_W}(N^{-1/2})$ in $P^*$-probability and $P_0\dot{\ell}_{\theta_0,\eta_0}=0$, we have
\begin{eqnarray*}
&&\sqrt{N}\hat{\mathbb{P}}_N^{\pi,bc}\dot{\ell}_{\theta_0,\eta_0}
  +\sqrt{N}P_0\dot{\ell}_{\hat{\hat{\theta}}_{N,bc},\hat{\hat{\eta}}_{N,bc}}\\
&&=-(\tilde{\mathbb{G}}_N^{\pi,bc}+\mathbb{G}_N^{\pi,c})( \dot{\ell}_{\hat{\hat{\theta}}_{N,bc},\hat{\hat{\eta}}_{N,bc}}-\dot{\ell}_{\theta_0,\eta_0}) +o_{P^*_W}(1),
\quad \mbox{in $P^*$-probability.}
\end{eqnarray*}
Since $(\hat{\hat{\theta}}_{N,c},\hat{\hat{\eta}}_{N,c})$ is consistent for $(\theta_0,\eta_0)$ in $P^*$-probability, it follows from Lemmas 5.4 of \cite{MR3059418} and \ref{lemma:wza5boot} that the above display is $o_{P^*_W}(1)$ in $P^*$-probability.
Similarly,
$\sqrt{N}\hat{\mathbb{P}}_N^{\pi,c}B_{\theta_0,\eta_0}\left[\underline{h}_0\right]
+\sqrt{N}P_0B_{\hat{\hat{\theta}}_{N,c},\hat{\hat{\eta}}_{N,c}}[\underline{h}_0]=o_{P^*_W}(1)$ in $P^*$-probability.
These and Condition \ref{cond:wlenonreg4} imply that
\begin{eqnarray}
\label{eqn:huangaboot}
&&P_0\left\{-\dot{\ell}_{\theta_0,\eta_0}(\dot{\ell}_{\theta_0,\eta_0}^T(\hat{\hat{\theta}}_{N,bc}-\theta_0)
        +B_{\theta_0,\eta_0}(\hat{\hat{\eta}}_{N,bc}-\eta_0))\right\}\nonumber\\
&&\quad + \ o\left(|\hat{\hat{\theta}}_{N,bc}-\theta_0|\right)+O\left(\lVert \hat{\hat{\eta}}_{N,bc}-\eta_0\rVert^{\alpha}\right)
          +\hat{\mathbb{P}}_N^{\pi,bc}\dot{\ell}_{\theta_0,\eta_0}\nonumber\\
&&=P_0\{-\dot{\ell}_{\theta_0,\eta_0}(\dot{\ell}_{\theta_0,\eta_0}^T(\hat{\hat{\theta}}_{N,bc}-\theta_0)
          +B_{\theta_0,\eta_0}(\hat{\hat{\eta}}_{N,bc}-\eta_0)) -\dot{\ell}_{\hat{\hat{\theta}}_{N,bc},\hat{\hat{\eta}}_{N,bc}}
            +\dot{\ell}_{\theta_0,\eta_0}\}\nonumber\\
&&\quad +\ o\left(|\hat{\hat{\theta}}_{N,bc}-\theta_0|\right)+O\left(\lVert \hat{\hat{\eta}}_{N,bc}-\eta_0\rVert^{\alpha}\right)
        +P_0\dot{\ell}_{\hat{\hat{\theta}}_{N,bc},\hat{\hat{\eta}}_{N,bc}}+\hat{\mathbb{P}}_N^{\pi,bc}\dot{\ell}_{\theta_0,\eta_0}\nonumber\\
&& =o_{P^*_W}(N^{-1/2})
\end{eqnarray}
in $P^*$-probability, and, furthermore, that
\begin{eqnarray}
\label{eqn:huangbboot}
&&P_0\left\{-B_{\theta_0,\eta_0}\left[\underline{h}_0\right](\dot{\ell}_{\theta_0,\eta_0}^T(\hat{\hat{\theta}}_{N,bc}-\theta_0)
       +B_{\theta_0,\eta_0}(\hat{\hat{\eta}}_{N,bc}-\eta_0))\right\}\nonumber \\
&& \quad + \ o\left(|\hat{\hat{\theta}}_{N,bc}-\theta_0|\right)+O\left(\lVert \hat{\hat{\eta}}_{N,bc}-\eta_0\rVert^{\alpha}\right)
        +\hat{\mathbb{P}}_N^{\pi,bc}B_{\theta_0,\eta_0}\left[\underline{h}_0\right]\nonumber\\
&&=o_{P^*_W}(N^{-1/2}), \quad \mbox{in $P^*$-probability.}
\end{eqnarray}

By Condition \ref{cond:wlenonregboot1} and $\alpha\beta >1/2$,
$\sqrt{N}O_{P^*_W}\left(\lVert \hat{\hat{\eta}}_N-\eta_0\rVert^\alpha\right)=o_{P^*_W}(1)$ in $P^*$-probability.
So by Condition \ref{cond:wlenonreg2} and taking the difference of (\ref{eqn:huangaboot}) and (\ref{eqn:huangbboot}), we have
\begin{eqnarray*}
&&-P_0\left(\left\{\dot{\ell}_{\theta_0,\eta_0}-B_{\theta_0,\eta_0}
       \left[\underline{h}_0\right]\right\}\dot{\ell}_{\theta_0,\eta_0}^T\right)\left(\hat{\hat{\theta}}_{N,bc}-\theta_0\right)
+ \ o\left(|\hat{\hat{\theta}}_{N,bc}-\theta_0|\right) \\
&& \quad+ o_{P_W^*}(N^{-1/2})- o_{P_W^*}(N^{-1/2})
       +\hat{\mathbb{P}}_N^{\pi,bc}\left(\dot{\ell}_{\theta_0,\eta_0}-B_{\theta_0,\eta_0}\left[\underline{h}_0\right]\right)\\ &&
=o_{P_W^*}(N^{-1/2})-o_{P_W^*}(N^{-1/2}),
\end{eqnarray*}
in $P^*$-probability or
\begin{equation*}
-I_0 (\hat{\hat{\theta}}_{N,bc}-\theta_0)
=\hat{\mathbb{P}}_N^{\pi,bc}\left(\dot{\ell}_{\theta_0,\eta_0}
   -B_{\theta_0,\eta_0}\left[\underline{h}_0\right]\right)
    +o_{P^*_W}(N^{-1/2}),
\end{equation*}
in $P^*$-probability.
It follows by the invertibility of $I_0$ that
\begin{eqnarray*}
\sqrt{N}\left(\hat{\hat{\theta}}_{N,bc}-\theta_0\right)
=-\sqrt{N}\hat{\mathbb{P}}_N^{\pi,c}I_0^{-1}\left(\dot{\ell}_{\theta_0,\eta_0}
-B_{\theta_0,\eta_0}\left[\underline{h}_0\right]\right)+o_{P^*_W}(1),
\end{eqnarray*}
in $P^*$-probability.
Since we have by Theorem 3.2 of \cite{MR3059418} that
\begin{eqnarray}
\label{eqn:thm3.2}
\quad \sqrt{N}\left(\hat{\theta}_{N,c}-\theta_0\right)
=-\sqrt{N}\mathbb{P}_N^{\pi,c}I_0^{-1}\left(\dot{\ell}_{\theta_0,\eta_0}
-B_{\theta_0,\eta_0}\left[\underline{h}_0\right]\right)+o_{P^*_W}(1),
\end{eqnarray}
in $P^*$-probability, taking a difference yields
\begin{eqnarray*}
\sqrt{N}\left(\hat{\hat{\theta}}_{N,bc}-\hat{\theta}_{N,c}\right)
=-\hat{\mathbb{G}}_N^{\pi,bc}I_0^{-1}\left(\dot{\ell}_{\theta_0,\eta_0}
-B_{\theta_0,\eta_0}\left[\underline{h}_0\right]\right)+o_{P^*_W}(1),
\end{eqnarray*}
in $P^*$-probability.
Apply Theorem \ref{thm:fpsDboot} to obtain a desired result.

For $\hat{\hat{\theta}}_{N,bsc}$, replace (\ref{eqn:thm3.2}) by
\begin{eqnarray*}
\sqrt{N}\left(\hat{\theta}_{N}-\theta_0\right)
=-\sqrt{N}\mathbb{P}_N^{\pi}I_0^{-1}\left(\dot{\ell}_{\theta_0,\eta_0}
-B_{\theta_0,\eta_0}\left[\underline{h}_0\right]\right)+o_{P^*_W}(1),
\end{eqnarray*}
in $P^*$-probability, and proceed in the same way as above to obtain
\begin{eqnarray*}
\sqrt{N}\left(\hat{\hat{\theta}}_{N,bsc}-\hat{\theta}_{N}\right)
=-\hat{\mathbb{G}}_N^{\pi,bsc}I_0^{-1}\left(\dot{\ell}_{\theta_0,\eta_0}
-B_{\theta_0,\eta_0}\left[\underline{h}_0\right]\right)+o_{P^*_W}(1),
\end{eqnarray*}
in $P^*$-probability. Apply Theorem \ref{thm:fpsDboot} to obtain a desired result.
\end{proof}

\begin{proof}[Proof of Lemma \ref{lemma:rateboot}]
With the help of Lemma \ref{lemma:gpleqgboot} (2), a proof is similar to the proof of Theorem 5.2 of \cite{MR3059418}.
\end{proof}

\begin{lemma}
\label{lemma:gpleqgboot}
Let $\mathcal{F}$ be a class of integrable functions that possibly depends on $N$.

\noindent (1)  Suppose that $E^*\lVert \mathbb{G}_N\rVert_{\mathcal{F}} \rightarrow 0$ as $n\rightarrow \infty$.
Then $E^*\lVert \tilde{\mathbb{G}}_N^\pi\rVert_{\mathcal{F}} \rightarrow 0$ as $n\rightarrow \infty$.

\noindent (2)  Suppose that the phase I bootstrap weights $W_N^{(1)}$ are bounded in $N$.
Then $E^*\left\lVert \tilde{\mathbb{G}}_N^{\pi}\right\rVert_{\mathcal{F}} \lesssim E^*\left\lVert \mathbb{G}_N\right\rVert_{\mathcal{F}}$.
\end{lemma}
\begin{proof}
(1) We have
\begin{eqnarray*}
&&\tilde{\mathbb{G}}_N^\pi
=\sqrt{N}\left(\hat{\mathbb{P}}_N^\pi-\mathbb{P}^\pi_N\right)
=\sum_{j=1}^J\frac{N_j}{n_j\sqrt{N}}\left(\sum_{i=1}^{N_j}W_{N_j,n_j,j,i}\xi_{j,i}\delta_{X_{j,i}}-\sum_{i=1}^{N_j}\xi_{j,i}\delta_{X_{j,i}}\right)\\
&&=\sum_{j=1}^J\frac{N_j}{n_j\sqrt{N}}\left[\sum_{i=1}^{N_j}\left\{\left(W_{n_j,j,i}^{(2)}\xi_{j,i}(W_{N_j,j,i}^{(1)}\delta_{X_{j,i}}-P_{0|j})\right)-\left(\xi_{j,i}(\delta_{X_{j,i}}-P_{0|j})\right)\right\}\right].
\end{eqnarray*}
Here we used the fact that $n_j^{-1}\sum_{i=1}^{N_j}\xi_{j,i}W_{n_j,j,i}^{(2)}=1$ and $n_j^{-1}\sum_{i=1}^{N_j}\xi_{j,i}=1$.
It follows from the triangle inequality and $n_j/N_j\geq \sigma$ that
\begin{eqnarray}
\label{eqn:bootgpidcmps}
\left\lVert\tilde{\mathbb{G}}_N^\pi\right\rVert_{\mathcal{F}}
&\leq& \frac{1}{\sigma\sqrt{N}}\sum_{j=1}^J \left\lVert \sum_{i=1}^{N_j}W_{n_j,j,i}^{(2)}\xi_{j,i}(W_{N_j,j,i}^{(1)}\delta_{X_{j,i}}-P_{0|j})\right\rVert_{\mathcal{F}}\nonumber\\
&&+\frac{1}{\sigma\sqrt{N}}\sum_{j=1}^J\left\lVert\sum_{i=1}^{N_j}\xi_{j,i}(\delta_{X_{j,i}}-P_{0|j})\right\rVert_{\mathcal{F}}.
\end{eqnarray}
We bound the expectation of $\left\lVert\tilde{\mathbb{G}}_N^\pi\right\rVert_{\mathcal{F}}$ by bounding each term on the right hand side of the inequality above.
For the first term in the last display we have
\begin{eqnarray*}
&&E^*\frac{1}{\sigma\sqrt{N}}\sum_{j=1}^J \left\lVert \sum_{i=1}^{N_j}W_{n_j,j,i}^{(2)}\xi_{j,i}(W_{N_j,j,i}^{(1)}\delta_{X_{j,i}}-P_{0|j})\right\rVert_{\mathcal{F}}\\
&&\leq E^*\frac{1}{\sigma\sqrt{N}}\sum_{j=1}^J E_{\xi}^*E_{0|\mathcal{V}_j,\xi}^*\left\lVert \sum_{i=1}^{N_j}W_{n_j,j,i}^{(2)}\xi_{j,i}(W_{N_j,j,i}^{(1)}\delta_{X_{j,i}}-P_{0|j})\right\rVert_{\mathcal{F}}\nu_j,
\end{eqnarray*}
where $E_\xi$ and $E_{0|\mathcal{V}_j,\xi}$ are the expectation with respect to $\xi$ and the conditional expectation given stratum membership and sampling indicators, respectively.
For the $j$th summand in the first term in (\ref{eqn:bootgpidcmps}),
note that we can rewrite the summand as
\begin{eqnarray*}
\left\lVert\sum_{i=1}^{N_j}W_{n_j,j,i}^{(2)}\xi_{j,i}(W_{N_j,j,i}^{(1)}\delta_{X_{j,i}}-P_{0|j})\right\rVert_{\mathcal{F}}
=\left\lVert \sum_{i=1}^{N_j}W_{n_j,j,i}^{(2)}\xi_{j,i}(\delta_{X_{j,i}}-P_{0|j})\right\rVert_{\mathcal{F}_{\mathcal{W}}},
\end{eqnarray*}
where $\mathcal{F}_{\mathcal{W}}=\{g(x,w)=wf(x):f\in\mathcal{F}\}$ because $P_{0|j}W^{(1)}_Nf(X)=P_{0|j}f(X)$ by the conditional independence of $X$ and $W^{(1)}$  given stratum membership and $P_{0|j}W^{(1)}_N=1$.
Note also that $W_{n_j,j,i}^{(2)}$'s with $\xi_{j,i}=1$ are exchangeable by construction and that $W_N^{(2)}$ are bounded in $N$ (see the proof of Lemma 7.4.1 of \cite{Saegusa} or Lemma 6.1 of \cite{Saegusa-Var:2014}).
Apply the multiplier inequality for bounded weights $W_{n_j,j,i}^{(2)}$ (Lemma 5.1 of \cite{MR3059418}) with $n_0=1$ and $Z_{ni}=\delta_{X_{j,i}}-P_{0|j}$ to the $j$th summand in the first term in (\ref{eqn:bootgpidcmps}) conditionally on stratum membership and sampling indicators, and then apply Jensen's inequality to the conditional expectation of the right hand side of the last display to obtain
\begin{eqnarray*}
&&\max_{1\leq k \leq N_j}E_{0|\mathcal{V}_j,\xi}^*\left\lVert \frac{1}{\sqrt{N}}\sum_{i=1}^{k}\xi_{j,i}(\delta_{X_{j,i}}-P_{0|j})\right\rVert_{\mathcal{F}_{\mathcal{W}}}\\
&&\lesssim E_{0|\mathcal{V}_j,\xi}^*\left\lVert \frac{1}{\sqrt{N}}\sum_{i=1}^{N_j}\xi_{j,i}(\delta_{X_{j,i}}-P_{0|j})\right\rVert_{\mathcal{F}_{\mathcal{W}}}.
\end{eqnarray*}
Take expectations with respect to sampling indicators to obtain
\begin{eqnarray*}
&&E_{0|\mathcal{V}_j}^*\frac{1}{\sqrt{N}}\left\lVert \sum_{i=1}^{N_j}W_{n_j,j,i}^{(2)}\xi_{j,i}(W_{N_j,j,i}^{(1)}\delta_{X_{j,i}}-P_{0|j})\right\rVert_{\mathcal{F}}\\
&&\lesssim E_{0|\mathcal{V}_j}^*\left\lVert \frac{1}{\sqrt{N}}\sum_{i=1}^{N_j}\xi_{j,i}(\delta_{X_{j,i}}-P_{0|j})\right\rVert_{\mathcal{F}_{\mathcal{W}}}.
\end{eqnarray*}
Apply again the multiplier inequality for bounded weights $\xi_i$ with $n_0=1$ and $Z_{ni}=\delta_{X_{j,i}}-P_{0|j}$ conditionally on $(V_i,X_i,W_i^{(1)})$ to the term in the last display to obtain its upper bound (up to some constant)
\begin{eqnarray*}
 \max_{1\leq k \leq N_j}E_{0|\mathcal{V}_j}^*\left\lVert \frac{1}{\sqrt{N}}\sum_{i=1}^{k}(\delta_{X_{j,i}}-P_{0|j})\right\rVert_{\mathcal{F}_{\mathcal{W}}}.
\end{eqnarray*}
Apply Jensen's inequality, Lemma A.1 of \cite{SW2013supp} and the triangle inequality to this term to obtain its upper bound
\begin{eqnarray*}
E_{0|\mathcal{V}_j}^*\left\lVert \sqrt{\frac{N_j}{N}}\mathbb{G}_{j,N_j}\right\rVert_{\mathcal{F}_{\mathcal{W}}}
\lesssim E^*\left\lVert \mathbb{G}_{N}\right\rVert_{\mathcal{F}_{\mathcal{W}}}
\leq E^*\left\lVert \mathbb{G}_{N}(W_N^{(1)}-1)f\right\rVert_{\mathcal{F}} + E^*\left\lVert \mathbb{G}_{N}\right\rVert_{\mathcal{F}}.
\end{eqnarray*}
Here we used the fact that $E|X|=\sum_{j=1}^JE_{0|j}|X|\nu_j \geq E_{0|j'}|X|\nu_{j'},j'\in\{1,\ldots,J\}$.
Apply the multiplier inequality (Lemma 2.9.1 of \cite{MR1385671}) with weights $W_{N}^{(1)}-1$ to $E^*\left\lVert \mathbb{G}_{N}(W_N^{(1)}-1)f\right\rVert_{\mathcal{F}}$ to obtain its upper bound
\begin{eqnarray}
\label{eqn:mutlineqboot}
&&2(N_0-1)E^*\lVert \delta_{X}-P_0\rVert_{\mathcal{F}} E\max_{1\leq i\leq N}\frac{|W_N^{(1)}-1|}{\sqrt{N}} \nonumber\\
&&+2\sqrt{2}\lVert W_{N}^{(1)}-1\rVert_{2,1} \max_{1\leq k\leq N}E^*\left\lVert \frac{1}{\sqrt{k}}\sum_{i=1}^{k}\epsilon_{i}(\delta_{X_{i}}-P_{0})\right\rVert_{\mathcal{F}}
\end{eqnarray}
where $\epsilon_{i},i=1,\ldots,N$, are independent Rademacher variables and $N_0$ is any natural number less than or equal to $N$.
Note that $\lVert W_{N}^{(1)}-1\rVert_{2,1} $ is bounded in view of Problem 2.9.2 of \cite{MR1385671} and by (\ref{eqn:ph1bootwt}), and that $E^*\lVert \delta_X-P_0\rVert_{\mathcal{F}}<\infty$ by assumption.
Note also that the symmetrization inequality (Lemma 2.3.6 of \cite{MR1385671}) yields that  by $E^*\lVert k^{-1/2}\sum_{i=1}^k(\delta_{X_i}-P_0)\rVert_{\mathcal{F}}^*\lesssim E\lVert \mathbb{G}_k\rVert_{\mathcal{F}}$.
For the second term in (\ref{eqn:bootgpidcmps}), apply the multiplier inequality, Jensen's inequality, and Lemma Lemma A.1 of \cite{SW2013supp} to the $j$th summand as above to obtain
\begin{eqnarray*}
&&E_{0|\mathcal{V}_j}^*\left\lVert \frac{1}{\sqrt{N}}\sum_{i=1}^{N_j}\xi_{j,i}(\delta_{X_{j,i}}-P_{0|j})\right\rVert_{\mathcal{F}}
\lesssim \max_{1\leq k \leq N_j}E_{0|\mathcal{V}_j}^*\left\lVert \frac{1}{\sqrt{N}}\sum_{i=1}^{k}(\delta_{X_{j,i}}-P_{0|j})\right\rVert_{\mathcal{F}}\\
&&\leq \max_{1\leq k \leq N_j}E_{0|\mathcal{V}_j}^*\left\lVert \frac{1}{\sqrt{N}}\sum_{i=1}^{N_j}(\delta_{X_{j,i}}-P_{0|j})\right\rVert_{\mathcal{F}}
\leq E_{0|\mathcal{V}_j}^*\left\lVert \mathbb{G}_{j,N_j}\right\rVert_{\mathcal{F}}
\lesssim E^*\left\lVert \mathbb{G}_{N}\right\rVert_{\mathcal{F}}.
\end{eqnarray*}
Thus, we have by Fubini's theorem (Lemma 1.2.7 of \cite{MR1385671}) that
\begin{eqnarray*}
E^*\left\lVert\tilde{\mathbb{G}}_N^\pi\right\rVert_{\mathcal{F}}
\lesssim E^*\left\lVert \mathbb{G}_{N}\right\rVert_{\mathcal{F}} + (N_0-1)E\max_{1\leq i\leq N}\frac{|W_N^{(1)}-1|}{\sqrt{N}}+\max_{N_0\leq k\leq N}E^* \left\lVert \mathbb{G}_{k}\right\rVert_{\mathcal{F}}.
\end{eqnarray*}
Because $E^*\left\lVert \mathbb{G}_{N}\right\rVert_{\mathcal{F}}\rightarrow 0$ and $E\max_{1\leq N}|W_N^{(1)}-1|/\sqrt{N}\rightarrow 0$, we can take $N\rightarrow \infty$ followed by $N_0\rightarrow \infty$ to conclude that the right-hand side of the last display converges to zero.

(2) We proceed in the same way as above except that we apply the multiplier inequality for the bounded exchangeable weights (Lemma 5.1 of \cite{MR3059418}) with $N_0=1$ to replace (\ref{eqn:mutlineqboot}) by
\begin{equation*}
\max_{1\leq k\leq N}E^*\left\lVert \frac{1}{\sqrt{N}}\sum_{i=1}^{k}(\delta_{X_{i}}-P_{0})\right\rVert_{\mathcal{F}}.
\end{equation*}
Apply Jensen's inequality in order to bound this term by $E^*\lVert \mathbb{G}_N\rVert_{\mathcal{F}}$ to obtain the desired result.
\end{proof}

The following is Lemma A.4 of \cite{SW2013supp} with correction that $\mathbb{S}_N=\sum_{i=1}^NN^{-1/2}\mathbb{Z}_{i}$ instead of $\mathbb{S}_N=\sum_{i=1}^N\mathbb{Z}_{i}$.
\begin{lemma}
\label{lemma:vw239}
Let $\mathbb{Z}_{1},\mathbb{Z}_{2},\ldots$ be i.i.d. stochastic processes indexed
by $\mathcal{F}_N$ with $E^*\lVert \mathbb{Z}_{1}\rVert_{\mathcal{F}_N}$ uniformly bounded in $N$.
Suppose that  $\lVert \mathbb{S}_N\rVert_{\mathcal{F}_N} \ $\\
$ \ \equiv\lVert \sum_{i=1}^NN^{-1/2}\mathbb{Z}_{i}\rVert_{\mathcal{F}_N}=o_{P^*}(1)$.
Then
$E^*\lVert \mathbb{S}_N\rVert_{\mathcal{F}_N}\rightarrow 0$, as $N\rightarrow \infty$.
\end{lemma}

The following is the bootstrap version of Lemma 5.4 of \cite{MR3059418}.
\begin{lemma}
\label{lemma:wza5boot}
Let $\mathcal{F}_N$ be a sequence of decreasing classes of functions such that
$\lVert \mathbb{G}_N\rVert_{\mathcal{F}_N}=o_{P^*}(1)$.
Assume that there exists an integrable envelope for $\mathcal{F}_{N_0}$ for some $N_0$.
Then $\lVert \tilde{\mathbb{G}}_N^{\pi}\rVert_{\mathcal{F}_N}=o_{P^*_W}(1)$ in $P^*$-probability.

Suppose, moreover, that $\mathcal{F}_N$ is $P_0$-Glivenko-Cantelli  with
$\| P_0 \|_{{\cal F}_{N_1}} < \infty$ for some $N_1$,
and that every $f = f_N \in\mathcal{F}_N$ converges to zero either pointwise or in
$L_1 (P_0)$ as $N\rightarrow \infty$.
Then $\lVert \tilde{\mathbb{G}}_N^{\pi,*\#}\rVert_{\mathcal{F}_N}=o_{P^*_W}(1)$ in $P^*$-probability with $*\in\{b,bs\}$ and $\#\in\{c,cc\}$, assuming Condition \ref{cond:bootcal}.
\end{lemma}
\begin{proof}
A proof is similar to the proof of Lemma 5.4 of \cite{MR3059418}.
\end{proof}


\newpage

\begin{table}[htbp]
\tiny
\begin{tabular}{cclccc}
&Data Set&Estimators &WLE & CalY & CCalY \\\hline
$N=400$&&&&& \\
$\alpha=\beta=.9$&&Truth (Empirical Var) &$\log 2$(.183)& $\log 2$(.183)& $\log 2$(.177) \\\hline
&1& Standard Est &.688(.170) &.666(.180) & .776(.173) \\\hline
&& Bootstrap  & .721(.181)& & \\
&& Bootstrap(Cal)  & & .716(.206)&.794(.178)\\
&& Bootstrap(Single Cal)  & & .719(.182)&.697(.168)\\\hline

&2& Standard Est & .908(.205) &.908(.205) & 1.048(.201)\\\hline
&& Bootstrap  & .976(.222)& & \\
&& Bootstrap(Cal)  &  & .976(.223)&1.118(.210) \\
&& Bootstrap(Single Cal)  & & .976(.223)&.973(.207) \\\hline

&3& Standard Est & .276(.131) &.275(.132) & .263(.124)\\\hline
&& Bootstrap  & .277(.131)& & \\
&& Bootstrap(Cal)  & & .278(.132)&.258(.116)\\
&& Bootstrap(Single Cal)  & & .278(.132)&.272(.116)\\\hline

$N=800$&&&&& \\
$\alpha=\beta=.9$&&Truth (Empirical Var) &$\log 2$(.0815)& $\log 2$(.0788)& $\log 2$(.0788) \\\hline
&4& Standard Est & .412 (.0797) &.414(.0792) & .409(.0768) \\\hline
&& Bootstrap  & .423(.0811)& & \\
&& Bootstrap(Cal)  &  &.423(.0793)&.422(.0748) \\
&& Bootstrap(Single Cal)  & &.422(.0815)&.425(.0751) \\\hline
&5& Standard Est & .849 (.0918) &.858(.0918) & .871(.0889) \\\hline
&& Bootstrap  & .873(.0971)& &\\
&& Bootstrap(Cal)  & &  .875(.0951)&.903(.0918) \\
&& Bootstrap(Single Cal)  & &  .873(.0977)&.880(.0917) \\\hline
&6& Standard Est & .140 (.0699) &.124(.0688) & .089(.0656)\\\hline
&& Bootstrap  & .142(.0720)& & \\
&& Bootstrap(Cal)  & & .137(.0683)&.094(.0656) \\
&& Bootstrap(Single Cal)  & & .142(.0724)&.145(.0657) \\\hline

$N=400$&&&&& \\
$\alpha=\beta=.5$&&Truth (Empirical Var) & $\log 2$(.192)& $\log 2$(.195)& $\log 2$(.193)\\\hline
&7& Standard Est & .740(0.184) &.740(0.184) &.691(0.184) \\\hline
&& Bootstrap  & .779(0.188)& & \\
&& Bootstrap(Cal)  & & .777(0.188)&.728(0.180)  \\
&& Bootstrap(Single Cal)  & & .777(0.188)&.777(0.175) \\\hline
&8& Standard Est & .640(.134) &.641(.134) & .592(.131) \\\hline
&& Bootstrap  & .653(.136) & & \\
&& Bootstrap(Cal)  & & .655(.138)&.620(.130) \\
&& Bootstrap(Single Cal)  & & .655(.138)&.672(.129)\\\hline
&9& Standard Est & .168(.173) &.169(.174) & .156(.174)  \\\hline
&& Bootstrap  &  .165(.181)& & \\
&& Bootstrap(Cal)  & &  .166(.183)&.139(.179) \\
&& Bootstrap(Single Cal)  & &  .166(.183)&.150(.179) \\\hline

$N=800$&&&&& \\
$\alpha=\beta=.5$&&Truth (Empirical Var) & $\log 2$(.0940)& $\log 2$(.0945)& $\log 2$(.0943)\\\hline
&10& Standard Est & .956 (.0979) &.962(.0967) & .970(.0967) \\\hline
&& Bootstrap  & .983(.0961)& & \\
&& Bootstrap(Cal)  & &  .984(.0920)&1.001(.0910) \\
&& Bootstrap(Single Cal)  & & .983(.0966)&.986(.0936)\\\hline
&11& Standard Est &.511 (.0877) &.503(.0888) & .533(.0882)  \\\hline
&& Bootstrap  & .518(.0917)& & \\
&& Bootstrap(Cal)  & & .516(.0949)&.546(.0888) \\
&& Bootstrap(Single Cal)  & & .518(.0919)&.523(.0868)\\\hline
&12& Standard Est & .281(.0822) &.275(.0848) & .295(.0840) \\\hline
&& Bootstrap  & .277(.0823)& & \\
&& Bootstrap(Cal)  & &  .275(.0901)&.291(.0874) \\
&& Bootstrap(Single Cal)  & & .277(.0831)&.276(.0822)\\\hline

\end{tabular}
\caption{Simulation Results for Bootstrap Inference (bootstrap means with variances in the parentheses): WLE denotes a plain vanilla WLE, CalY denotes a WLE with calibration on $Y$, and CCalY denotes a WLE with within-stratum centered calibration on $Y$.}
\label{tbl:res}
\end{table}

\begin{table}[htbp]
\caption[Comparison of WLEs for NWTS data.]{Results of point estimates with estimates of their standard deviations in the parentheses. Estimates considered are the MLE with the complete data, the WLE, the calibrated WLE, and the within-stratum centered calibrated WLE with corresponding bootstrap estimators. UH stands for unfavorable histology from the central reference laboratory, age1 and age2 are piecewise linear terms for age at diagnosis (years) before and after 1 year respectively, stg34 is a indicator of the stage III-IV, tumdiam is tumor diameter, and a colon stands for the interaction.}
\centering
\tiny
\begin{tabular}{|c||c||c|l||c|l|}\hline
\multicolumn{6}{|c|}{Estimators Based on a Full Cohort}\\\hline
&\multicolumn{5}{|l|}{MLE}\\\hline
UH     &\multicolumn{5}{|l|}{ 4.042 (0.413)} \\
age1   &\multicolumn{5}{|l|}{-0.661 (0.326) } \\
age2   &\multicolumn{5}{|l|}{0.104 (0.017)  } \\
stg34    &\multicolumn{5}{|l|}{1.346 (0.244) } \\
tumdiam  &\multicolumn{5}{|l|}{0.069 (0.014) }  \\
stg34:diam &\multicolumn{5}{|l|}{-0.076 (0.019) }  \\
UH:age1 &\multicolumn{5}{|l|}{-2.635 (0.464) }  \\
UH:age2 &\multicolumn{5}{|l|}{-0.058 (0.034) }  \\\hline
\multicolumn{6}{|c|}{Estimators Based on an Original Sample}\\\hline
       &  WLE             & \multicolumn{2}{|l||}{Cal}     & \multicolumn{2}{|l|}{CCal }\\\hline
UH     &4.054 (0.554)     & \multicolumn{2}{|l||}{4.083 (0.556) }&\multicolumn{2}{|l|}{ 4.065 (0.543)  }\\
age1   &-0.627 (0.366)    &\multicolumn{2}{|l||}{ -0.641 (0.368) }&\multicolumn{2}{|l|}{-0.683 (0.339) }\\
age2   &0.096 (0.025)     &\multicolumn{2}{|l||}{0.097 (0.025) }& \multicolumn{2}{|l|}{0.112 (0.017) }\\
stg34    &1.869 (0.352)     &\multicolumn{2}{|l||}{1.855 (0.356) }&\multicolumn{2}{|l|}{1.847 (0.345) }\\
tumdiam  &0.096 (0.021)     &\multicolumn{2}{|l||}{0.094 (0.021)) }&\multicolumn{2}{|l|}{0.096  (0.020) }\\
stg34:diam &-0.124 (0.029)  &\multicolumn{2}{|l||}{-0.123 (0.029) }&\multicolumn{2}{|l|}{-0.123 (0.028) }\\
UH:age1 &-2.781 (0.633)     &\multicolumn{2}{|l||}{-2.810 (0.635)}&\multicolumn{2}{|l|}{ -2.766(0.619) }\\
UH:age2 &-0.037 (0.055)     &\multicolumn{2}{|l||}{0.035 (0.053)}&\multicolumn{2}{|l|}{-0.044 (0.051) }   \\\hline
\multicolumn{6}{|c|}{Bootstrap Estimators}\\\hline
          &  WLE             &Cal (Boot)           & Cal (Boot Single)& CCal (Boot) & CCal (Boot Single) \\\hline
UH        &4.119 (0.543)     & 4.149 (0.546) & 4.121 (0.545)&4.121 (0.530) &4.107 (0.531)\\
age1      &-0.642 (0.363)    &-0.652 (0.368) &-0.640 (0.366)&-0.684 (0.331)& -0.633 (0.330)\\
age2      &0.098 (0.025)     &0.099 (0.025)  &0.098 (0.026)&0.113 (0.017) & 0.097 (0.017)\\
stg34     &1.896 (0.348)     &1.880 (0.347)  &1.895 (0.350)&1.872 (0.338) & 1.894 (0.345)\\
tumdiam   &0.097 (0.021)     &0.095 (0.021)  &0.097 (0.021)&0.096 (0.020) & 0.097 (0.020)\\
stg34:diam  &-0.1246 (0.029) &-0.124 (0.029) &-0.126 (0.029)&-0.124 (0.028) & -0.125 (0.028)\\
UH:age1   &-2.849 (0.621)    &-2.880 (0.623) &-2.851 (0.623)&-2.836 (0.604) & -2.850 (0.605)\\
UH:age2   &-0.035 (0.055)    & -0.034 (0.054)&-0.036 (0.055)&-0.041 (0.049) & -0.034 (0.053) \\\hline

\end{tabular}
\label{tbl:nwt1}
\end{table}

\end{document}